\documentclass[11pt]{amsart}
\usepackage{cgeometry}
\usepackage{physics}
\usepackage{float}
\usepackage{mathpazo}

\usepackage[margin=1.0in]{geometry}

\bibliography{ComplexGeometry}
\title[Non-pluripolar products on vector bundles]{Non-pluripolar products on vector bundles and Chern--Weil formulae}
\author{Mingchen Xia}

\begin{document}

\begin{abstract}
    In this paper, we develop several pluripotential-theoretic techniques for singular metrics on vector bundles.
    We first introduce the theory of non-pluripolar products on holomorphic vector bundles on complex manifolds. Then we define and study a special class of singularities of Hermitian metrics on vector bundles, called $\mathcal{I}$-good singularities, partially extending Mumford's notion of good singularities.
    Next,  we derive a Chern--Weil type formula expressing the Chern numbers of Hermitian vector bundles with $\mathcal{I}$-good singularities in terms of the associated b-divisors. We also define an intersection theory on the Riemann--Zariski space and apply it to reformulate our Chern--Weil formula.
\end{abstract}

 \normalsize

\maketitle

\tableofcontents

\section*{Statements and Declarations}
The author declares no conflict of interest.

\section*{Introduction}
This paper is devoted to lie down the foundation for studying Griffiths quasi-positive singular Hermitian metrics on holomorphic vector bundles on quasi-projective varieties. We expect our work to be the first step in establishing the arithmetic intersection theory on mixed Shimura varieties and in extending Kudla's program to  mixed Shimura varieties.

We explore two different techniques. First of all, we introduce the general theory non-pluripolar products on vector bundles. Secondly, we continue the study of b-divisors associated with Hermitian line bundles initiated in \cite{XiaPPT} and \cite{BBGHdJ21}. 

\subsection{Background}
One of the most striking features of the intersection theory on Shimura varieties is the so-called \emph{Hirzebruch--Mumford proportionality theorem} \cite{Hir58, Mum77}. Consider a locally symmetric space $\Gamma\backslash D$, where $D=G/K$ is a bounded symmetric domain, $G$ is a semi-simple adjoint real Lie group such that there is a $\mathbb{Q}$-algebraic group $\mathcal{G}$ with $G=\mathcal{G}(\mathbb{R})^+$, $K$ is a maximal compact subgroup of $G$ and $\Gamma\subseteq \mathcal{G}(\mathbb{Q})$ is a neat arithmetic subgroup. By \cite{BB66}, $\Gamma\backslash D$ is in fact quasi-projective.
Given any finite-dimensional unitary representation $\rho:K\rightarrow U(n)$ of $K$, one can naturally construct an equivariant Hermitian vector bundle $\hat{E}=(E,h_E)$ on $\Gamma\backslash D$ and an equivariant vector bundle $E'$ on the compact dual $\check{D}$ of $D$.  

Let $\overline{\Gamma\backslash D}$ be a smooth projective toroidal compactification of $\Gamma\backslash D$ in the sense of \cite{AMRT}. Mumford proved that $E$ has a unique extension $\bar{E}$ to $\overline{\Gamma\backslash D}$ such that the metric has \emph{good singularities} at the boundary. Then the proportionality theorem states that the Chern numbers of $\bar{E}$ are proportional to the Chern numbers of $E'$. The same holds if we consider the mixed Chern numbers of various vector bundles obtained as above, although this is not explicitly written down in \cite{Mum77}.

Modulo some easy curvature computations, the essence of Hirzebruch--Mumford's proportionality is the following \emph{Chern--Weil} type result: a Hermitian vector bundle $\hat{E}$ on the quasi-projective variety $\Gamma\backslash D$ admits an extension as a vector bundle $\bar{E}$ with singular metric on the compactification $\overline{\Gamma\backslash D}$, such that the Chern forms of $\hat{E}$ on $\Gamma\backslash D$, regarded as currents on $\overline{\Gamma\backslash D}$, represents the Chern classes of $\bar{E}$. The whole idea is embodied in the notion of \emph{good singularities}. In this paper, we would like to understand this phenomenon in greater generality.

The examples of locally symmetric spaces include all Shimura varieties. It is therefore natural to ask if the same holds on mixed Shimura varieties. It turns out that this phenomenon does not happen even in the simplest examples like the universal elliptic curve with level structures, see \cite{BGKK16}. In fact, by definition of good metrics, the singularities at infinity of good singularities are very mild, which is far from being true in the mixed Shimura case.

In this paper, we want to answer the following general question:
\begin{question}\label{que:mainq}
Consider a quasi-projective variety $X$ and Hermitian vector bundles $\hat{E}_i$ on $X$, how can we interpret the integral of Chern polynomials of $\hat{E}_i$ in terms of certain algebraic data on the compactifications of $X$? 
\end{question}
We will first content ourselves to the special case where $\hat{E}$ is Griffiths positive. 
Already in this case, we have non-trivial examples like the automorphic line bundles on the universal abelian varieties. 
In fact, the positivity assumption is not too restrictive, as we can always twist $E$ by some ample line bundle and our theory will turn out to be insensitive to such perturbations, so we could handle a much more general case than just positive vector bundles.

\cref{que:mainq} cannot have a satisfactory answer for all Griffiths positive singularities. Already on the line bundle $\mathcal{O}(1)$ on $\mathbb{P}^1$, it is easy to construct families of plurisubharmonic metrics whose Chern currents outside a Zariski closed subset are meaningless from the cohomological point of view, see \cite[Example~6.10]{BBJ21}. We will remedy this by introducing two nice classes of singularities: full-mass singularities and $\mathcal{I}$-good singularities.

We will develop two different techniques to handle \cref{que:mainq}, divided into the two parts of the paper.

\subsection{Main results in \cref{part:1}}
The first part concerns the non-pluripolar intersection theory on vector bundles. Recall that the non-pluripolar products of metrics on line bundles were introduced in \cite{BEGZ10}. Consider a compact K\"ahler manifold $X$ of pure dimension $n$ and line bundles $L_1,\ldots,L_m$ on $X$. We equip each $L_i$ with a singular plurisubharmonic (psh) metric $h_i$. We write $\hat{L}_i=(L_i,h_i)$.
The non-pluripolar product
\[
c_1(\hat{L}_1)\wedge \cdots \wedge c_1(\hat{L}_m)
\]
is a closed positive $(m,m)$-current on $X$ that puts no mass on any pluripolar set. When the $h_i$'s are bounded, this product is nothing but the classical Bedford--Taylor product. The non-pluripolar theory is not the only possible extension of Bedford--Taylor theory to unbounded psh metrics. However, there are two key features that single out the non-pluripolar theory among others: first of all, the non-pluripolar products are defined for \emph{all} psh metrics; secondly, the non-pluripolar masses satisfy the monotonicity theorem \cite{WN19}. Both properties are crucial to our theory.

There is a slight extension of the non-pluripolar theory constructed recently by Vu \cite{Vu20}. He defined the so-called the \emph{relative non-pluripolar products}. Here \emph{relative} refers to the extra flexibility of choosing a closed positive current $T$ on $X$ and one can make sense of expressions like
\[
c_1(\hat{L}_1)\wedge \cdots \wedge c_1(\hat{L}_m)\cap T.
\]
The usual non-pluripolar products correspond to the case $T=[X]$, the current of integration along $X$.
For the purpose of defining the non-pluripolar products on vector bundles, we slightly extend Vu's theory by allowing $T$ to be closed dsh currents in \cref{sec:relnpp}, see \cref{def:dsh} for the definition of closed dsh currents.

Now suppose that we are given a Hermitian vector bundle $\hat{E}=(E,h_E)$ on $X$, $\rank E=r+1$ and $h_E$ is probably singular. We assume that $h_E$ is Griffiths positive in the sense of \cref{def:grif}. As in the usual intersection theory, one first investigates the Segre classes $s_i(\hat{E})$. We will realize $s_i(\hat{E})$ as an operator $\widehat{Z}_a(X)\rightarrow \widehat{Z}_{a-i}(X)$, where $\widehat{Z}_a(X)$ is the vector space of closed dsh currents of bidimension $(a,a)$ on $X$. Let $p\colon \mathbb{P}E^{\vee}\rightarrow X$ be the natural projection (here $\mathbb{P}E^{\vee}$ is the projective space of lines in $E^{\vee}$). From a simple computation, one finds that the natural map $p^*E\rightarrow \mathcal{O}(1)$ induces a psh metric $h_{\mathcal{O}(1)}$ on $\mathcal{O}(1)$. We write $\hO(1)$ for $\mathcal{O}(1)$ equipped with this metric. Then the natural definition of $s_i(\hat{E})\cap\colon \widehat{Z}_a(X)\rightarrow \widehat{Z}_{a-i}(X)$ is
\[
s_i(\hat{E})\cap T\coloneqq (-1)^i p_* \left(c_1(\hO(1))^{r+i}\cap p^*T\right).
\]
 Here in the bracket, the product is the relative non-pluripolar product.
Of course, one needs to make sense of $p^*T$. This is possible as $p$ is a flat morphism. The detailed construction is provided by Dinh--Sibony in \cite{DS07}. In our situation, the pull-back can also be constructed elementarily as the dual of the fibral integration.
We will prove several important functoriality results of the pull-back in \cref{sec:pull}.

We will prove that the Segre classes behave like the usual Segre classes in \cref{sec:SegChern}. By iteration, we can make sense of expressions like
\[
P(s_i(\hat{E}_j))\cap T,
\]
where $P$ is a polynomial and the $\hat{E}_j$'s are Griffiths positive Hermitian vector bundles on $X$. Due to the non-linearity of the relative non-pluripolar product, we need a technical assumption that $T$ does not have mass on the polar loci of any $\hat{E}_j$. We express this condition as $T$ is \emph{transversal} to the $\hat{E}_j$'s.

In particular, this construction allows us to make sense of the Chern classes $c_i(\hat{E})\cap T$ as long as $T$ is transversal to $\hat{E}$. Observe that we do not have splitting principles in the current setting, so the vanishing of higher Chern classes is not clear. In fact, we only managed to prove it in the following special but important case of small unbounded locus:
\begin{theorem}[{\cref{cor:vanishingChern}}]Let $T\in \widehat{Z}_a(X)$, $\hat{E}$ be a Griffiths positive Hermitian vector bundles on $X$ having small unbounded locus. 
Assume that $T$ is transversal to $\hat{E}$.
Then for any $i>\rank E$, $c_i(\hat{E})\cap T=0$.
\end{theorem}
We point out that the corresponding result is not known in the theory of Chern currents of \cite{LRSW18}. 

See \cref{def:small} for the definition of small unbounded locus. This assumption is not too restrictive for most applications. A common place where singular metrics arise is on the compactifications of moduli spaces. In many cases, the singularities of the metric only occur at the boundary of the moduli space, hence naturally having small unbounded locus.

Let us mention that there are (at least) two other methods to make sense of the Chern currents of singular Hermitian vector bundles. Namely, \cite{LRSW18} and \cite{LRRS18}. The former only works for analytic singularities and suffers from the drawback that Segre classes do not commute with each other; the latter puts a strong restriction on the dimension of polar loci. In our non-pluripolar theory, the characteristic classes are defined for all Griffiths positive singularities and behave in the way that experts in the classical intersection theory would expect.

Our theory works pretty well for positive (and by extension quasi-positive) singularities, by duality, it can be easily extended to the negative (and quasi-negative) case. However, some natural singularities belong to neither class. In these cases, we do not possess a satisfactory Hermitian intersection theory.

Next, in \cref{sec:specmet}, we introduce two special classes of singularities on a Griffiths positive Hermitian vector bundle $\hat{E}$. The first class is the \emph{full mass} metrics. We say $\hat{E}$ has full mass if $\hO(1)$ on $\mathbb{P}E^{\vee}$ has full (non-pluripolar) mass in the sense of \cite{DDNL18fullmass}. We will show in \cref{prop:fullmasscrit} that this is equivalent to the condition that $\int_X s_n(\hat{E})$ is equal to the $n$-th Segre number of $E$ if $E$ is nef. Here $s_n(\hat{E})=s_n(\hat{E})\cap [X]$.

The most important feature of a full mass metric is:
\begin{theorem}[{\cref{thm:ChernrepChern}}]\label{thm:Chernmain}
Let $\hat{E}_1,\ldots,\hat{E}_m$ be Griffiths positive Hermitian vector bundles on $X$. Assume that the $\hat{E}_j$'s have full masses and the $E_j$'s are nef. Let $P(c_i(\hat{E}_j))$ be a homogeneous Chern polynomial of degree $n$. Then $P(c_i(\hat{E}_j))$ represents $P(c_i(E_j))$. 
\end{theorem}
This gives the Chern--Weil formula in the full mass case. It is not hard to generalize to non-nef $E_j$. However, one could also derive the general case directly from the more general theorem \cref{thm:Dmixedvec} below.

Nevertheless, the natural metrics in a number of concrete examples (for examples, natural metrics on mixed Shimura varieties) are not always of full mass. We need to relax the notion of full mass metrics. This gives the \emph{$\mathcal{I}$-good metrics}. We say $\hat{E}$ is $\mathcal{I}$-good if $\hO(1)$ is $\mathcal{I}$-good in the sense that its non-pluripolar mass is positive and is equal to the the volume defined using multiplier ideal sheaves. See \cref{def:Igoodvect} for the precise definition. As a consequence of \cite{DX21, DX22}, we have the following characterization of $\mathcal{I}$-good potentials.
\begin{theorem}[{\cref{thm:Igoodvect}}]
Let $\hat{E}=(E,h_E)$ be a Griffiths positive Hermitian vector bundle on $X$ of rank $r+1$. Assume that $\hO(1)$ has positive mass, then $\hat{E}$ is $\mathcal{I}$-good if and only if
\[
    \lim_{k\to\infty} \frac{1}{k^{n+r}} h^0(X, \mathcal{I}_k(h_E))=\frac{(-1)^n}{(n+r)!}\int_X s_n(\hat{E}).
\]
\end{theorem}
Here $\mathcal{I}_k(h_E)\subseteq \Sym^k E$ are multiplier sheaves defined in \cref{def:mis}. These multiplier ideal sheaves are different from the usual one defined by $L^2$-sections.

An example of $\mathcal{I}$-good singularities is given by the so-called \emph{toroidal singularities} in \cref{def:tor}. This definition seems to be the natural generalization of the toroidal singularities on line bundles introduced by Botero--Burgos Gil--Holmes--de Jong in \cite{BBGHdJ21}. 
Another important example is the so-called analytic singularities as studied in \cite{LRSW18}.

In \cref{sec:quapos}, we extend the notion of $\mathcal{I}$-good singularities to not necessarily positively curved case. 
We will establish the following result:
\begin{theorem}[={\cref{prop:Igoodtensor}+\cref{thm:Igoodcancel}}]
Assume that $X$ is projective.

Let $\hat{L},\hat{L}'$ be $\mathcal{I}$-good Hermitian pseudo-effective line bundles on $X$. Then $\hat{L}\otimes \hat{L}'$ is also $\mathcal{I}$-good.

Conversely, if $\hat{L}$, $\hat{L}'$ are Hermitian pseudo-effective line bundles such that $\hat{L}$ has positive mass. Suppose that $\hat{L}\otimes \hat{L}'$ is $\mathcal{I}$-good then so is $\hat{L}$.
\end{theorem}

These results allow us to define a general notion of $\mathcal{I}$-goodness for not necessarily positively curved line bundles: we say $\hat{L}$ is $\mathcal{I}$-good if after tensoring by a suitable $\mathcal{I}$-good positively curved line bundle, it becomes an $\mathcal{I}$-good positively curved line bundle, see \cref{def:Igoodnotp} for the precise definition. There is also a similar extension in the case of vector bundles \cref{def:igoodv}.

We expect $\mathcal{I}$-good singularities to be the natural singularities in the mixed Shimura setting or more generally in other moduli problems. 

\subsection{Main results in \cref{part:2}}
We begin to answer \cref{que:mainq} in greater generality. This question only has a satisfactory answer in the case of $\mathcal{I}$-good singularities. We consider a smooth quasi-projective variety $X$, Griffiths positive smooth Hermitian vector bundles $\hat{E}_i=(E_i,h_i)$ on $X$. We assume that the $\hat{E}_i$'s are \emph{compactifiable} in the sense of \cref{def:compa}.

The notion of $\mathcal{I}$-good metrics extends to the quasi-projective setting, see \cref{def:Igoodquasiproj}.

We want to understand the Chern polynomials of $P(c_j(\hat{E}_i))$. In the case of full mass currents, the solution is nothing but \cref{thm:Chernmain}. In the case of $\mathcal{I}$-good singularities, by passing to some projective bundles, the problem is essentially reduced to the line bundle case.

We first handle the elementary case of line bundles. The solution relies on the so-called b-divisor techniques. Roughly speaking, a b-divisor on $X$ is an assignment of a \emph{numerical class} on each projective resolution $Y\rightarrow X$, compatible under push-forwards between resolutions. To each Hermitian line bundle $\hat{L}$ on $X$, assuming some technical conditions known as \emph{compactfiability}, we construct a b-divisor $\mathbb{D}(\hat{L})$ on $X$ in \cref{def:bdivasso} using the singularities of the metric on $L$.

Recall that a Hermitian pseudo-effective line bundle is just a Griffiths positive Hermitian vector bundle of rank $1$.
We prove that
\begin{theorem}[{\cref{thm:nefbvolume}}]\label{thm:bd}
Assume that $X$ is projective.
Assume that $\hat{L}$ is a Hermitian pseudo-effective line bundle on $X$, then the b-divisor $\mathbb{D}(\hat{L})$ is nef. Assume in addition that the non-pluripolar mass of $\hat{L}$ is positive, then 
\[
\frac{1}{n!}\vol\mathbb{D}(\hat{L})=\vol \hat{L}.
\]
\end{theorem}
See \cref{def:nef} for the notion of nef b-divisors. Nef b-divisors were introduced and studied in \cite{DF20} based on \cite{BFJ09}.

The study of b-divisors associated with singular metrics on line bundles originates from \cite{BFJ08}. In the case of ample line bundles on projective manifolds, this technique was explored in \cite{XiaPPT}. At the time when  \cite{XiaPPT} was written, the general intersection theory in the second version of \cite{DF20} and the general techniques dealing with singular potentials developed in \cite{DX21} were not available yet, so the results in \cite{XiaPPT} were only stated in the special case of ample line bundles. In particular, when $L$ is ample, \cref{thm:bd} is essentially proved in \cite[Theorem~5.2]{XiaPPT}.
Later on, the same technique was rediscovered by Botero--Burgos Gil--Holmes--de Jong in \cite{BBGHdJ21} and more recently by Trusiani \cite{Tru23}. In particular, a special case of \cref{thm:bd} was proved in \cite{BBGHdJ21}, although they made use of a different notion of b-divisors.
For a nice application of \cref{thm:bd} to the the theory of Siegel--Jacobi modular forms, we refer to the recent preprint \cite{BBGHdJ22}. 

In the revised version of \cite{Xia21}, we apply \cref{thm:bd} to prove the Hausdorff convergence property of partial Okounkov bodies, as conjectured in the first version of that paper.


In the quasi-projective setting, we will prove that
\begin{theorem}[={\cref{thm:mixedDDD}}]\label{thm:bdmix}
Assume that $\hat{L}_1,\ldots,\hat{L}_n$ are compactifiable Hermitian line bundles on $X$ with singular psh metrics having positive masses. Then
\begin{equation}\label{eq:1}
\left(\mathbb{D}(\hat{L}_1),\ldots, \mathbb{D}(\hat{L}_n)\right)\geq \int_X c_1(\hat{L}_1)\wedge \cdots \wedge c_1(\hat{L}_n).
\end{equation}
Equality holds if all $\hat{L}_i$'s are $\mathcal{I}$-good.
\end{theorem}
Here we refer to \cref{sec:quasproj} for the relevant notions. This theorem is further generalized to not necessarily positively curved case in \cref{cor:Igoodgeneralintersec}. One may regard the equality case of \eqref{eq:1} as a Chern--Weil formula as in \cite{BBGHdJ21}. 

As a consequence of \cref{thm:bdmix}, one can for example compute the mixed intersection numbers of b-divisors associated with several \emph{different} Siegel--Jacobi line bundles on the universal Abelian varieties, giving new insights into the cohomological aspects of mixed Shimura varieties. As the techniques involved in such computations are quite different from the other parts of this paper, we decide to omit these computations.

This theorem and \cref{cor:Igoodgeneralintersec} suggest that $\mathbb{D}(\hat{L})$ should be regarded as the first Chern class on the Riemann--Zariski space $\mathfrak{X}$ of $X$. Pushing this analogue further, one can actually make sense of all Chern classes on the Riemann--Zariski space and generalize \cref{thm:bdmix} to higher rank. 
We will carry this out in \cref{sec:RZ}. 

Here we briefly recall the main idea. Our approach to the intersection theory on the Riemann--Zariski space is based on K-theory. The reason is that the Riemann--Zariski space is a pro-scheme, coherent sheaves and locally free sheaves on pro-schemes are easy to understand in general, at least when the pro-scheme satisfies Oka's property (namely, the structure sheaf is coherent), which always holds for the Riemann--Zariski space \cite{KST18}.

There are at least three different ways of constructing Chow groups from K-theory. The first approach is via the $\gamma$-filtration as in \cite{SGA6}. This approach relies simply on the augmented $\lambda$-ring structure on the $K$-ring and limits in this setting is well-understood.  The other approaches include using the coniveau filtration or using Bloch's formula. As the procedure of producing the Riemann--Zariski space destroys the notion of codimension, there might not be a coniveau filtration in the current setting. On the other hand, Bloch's formula is less elementary and relies on higher K-theory instead of just $K_0$, but it provides information about torsions in the Chow groups as well.

We will follow the first approach:
\[
\CH^{\bullet}(\mathfrak{X})_{\mathbb{Q}}\coloneqq \Gr^{\bullet}_{\gamma}K(\mathfrak{X})_{\mathbb{Q}}.
\]
It turns out that there is a vector space homomorphism from $\CH^1(\mathfrak{X})_{\mathbb{R}}$ to the space of Cartier b-divisors. Moreover, when $\hat{L}$ is a Hermitian pseudo-effective line bundle with analytic singularities, the b-divisor $\mathbb{D}(\hat{L})$ has a canonical lift $\mathbf{c}_1(\hat{L})$ to $\CH^1(\mathfrak{X})_{\mathbb{R}}$. This gives the notion of first Chern classes we are looking for. With some efforts, this approach leads to the notion of Chern classes of Hermitian vector bundles with analytic singularities as well.

In the case of general $\mathcal{I}$-good singularities, as we will explain in \cref{sec:RZ}, it seems impossible to lift $\mathbb{D}$ to Chow groups. So we are forced to work out the notion of Chern classes modulo numerical equivalence. Now we can have a glance of the final result.
\begin{theorem}[={\cref{cor:ref2}}]\label{thm:Dmixedvec}Assume that $X$ is projective.
Let $\hat{E}_i$ be $\mathcal{I}$-good Griffiths positive vector bundles on $X$ ($i=1,\ldots,m$). 
Consider a homogeneous Chern polynomial $P(c_i(E_j))$ of degree $n$ in $c_i(E_j)$, then
\begin{equation}\label{eq:mainZartoana}
\int_{\mathfrak{X}}P(\mathbf{c}_i(\hat{E}_j))
=\int_X P(c_i(\hat{E}_j)).
\end{equation}
\end{theorem}
The notations will be clarified in \cref{sec:RZ}.

This beautiful formula establishes the relation between algebraic objects on the left-hand side to analytic objects on the right-hand side. This is our final version of the \emph{Chern--Weil formula}. We remark that the assumption of $\mathcal{I}$-goodness is essential. It is probably the most general class of singularities where one can expect something like \eqref{eq:mainZartoana}, as indicated by the line bundle case \cite[Theorem~1.4]{DX22}.

In conclusion, \cref{thm:bd} and \cref{thm:bdmix} tell us that the Chern currents of $\mathcal{I}$-good Hermitian line bundles on a quasi-projective variety represent decreasing limits of Chern numbers of the compactifications or equivalently, Chern numbers on the Riemann--Zariski space. \cref{thm:Dmixedvec} gives a similar result in the case of vector bundles.
This is our answer to \cref{que:mainq}.

\subsection{Auxiliary results in pluripotential theory}

Finally, let us also mention that we also established a few general results about the non-pluripolar products of quasi-psh functions during the proofs of the main theorems. We mention two of them. Note that in both theorems, $X$ can be a general compact K\"ahler manifold, not necessarily projective.

\begin{theorem}[{\cref{thm:convdsmeasures}}]
Let $\theta_1,\ldots,\theta_{a}$ be smooth real closed $(1,1)$-forms on $X$ representing big cohomology classes.
Let $\varphi^i_j\in \PSH(X,\theta_j)$ ($j=1,\ldots,a$) be decreasing (resp. increasing) sequences converging to $\varphi_j\in \PSH(X,\theta_j)_{>0}$ pointwisely (resp. almost everywhere). Assume that $\lim_{i\to\infty}\int_X \theta_{j,\varphi_j^i}^n=\int_X \theta_{j,\varphi_j}^n$ for all $j=1,\ldots,a$. Then we have the weak convergence of currents:
\begin{equation}\label{eq:temp15}
\theta_{1,\varphi^i_1}\wedge \cdots \wedge \theta_{a,\varphi^i_a}\rightharpoonup \theta_{1,\varphi_1}\wedge \cdots \wedge \theta_{a,\varphi_a}
\end{equation}
as $i\to\infty$.
\end{theorem}
Here in \eqref{eq:temp15}, the products are taken in the non-pluripolar sense and $\PSH(X,\theta_j)_{>0}$ denotes the subset of $\PSH(X,\theta_j)$ consisting of potentials with positive non-pluripolar masses.
This theorem is of independent interest as well. The case $a=n$ is proved in \cite{DDNL18mono} under a slightly different assumption. 

\begin{theorem}[{\cref{thm:NAdatacontdS}}]
Let $\varphi_i\in \PSH(X,\theta)$ ($i\in \mathbb{N}$) be a sequence and $\varphi\in \PSH(X,\theta)$. Assume that $\varphi_i\xrightarrow{d_S}\varphi$, then for any prime divisor $E$ over $X$ (namely, a prime divisor on a birational model $\pi\colon Y\rightarrow X$ of $X$),
\begin{equation}
\lim_{i\to\infty}\nu(\varphi_i,E)=\nu(\varphi,E).
\end{equation}
\end{theorem}
Here $d_S$ is the pseudo-metric on $\PSH(X,\theta)$ introduced by Darvas--Di Nezza--Lu in \cite{DDNLmetric}.
We will recall its definition in \cref{sec:prel}. 
The notation $\nu(\varphi,E)$ denotes the generic Lelong number of (the pull-back of) $\varphi$ along $E$.
 
This theorem can be seen as a common (partial) generalization of a number of known results. For example \cite[Exercise~2.7(iii)]{GZ17}, \cite[Theorem~6.1]{DDNLmetric} and \cite[Theorem~4.6]{Xia21}.

This theorem confirms that the map from a quasi-plurisubharmonic function to the associated non-Archimedean data is continuous. When $\theta$ comes from the $c_1$ of a big line bundle, this statement can be made precise using the non-Archimedean language developed by Boucksom--Jonsson \cite{BJ18b}. In general, it allows us to generalize a number of Boucksom--Jonsson's constructions to transcendental classes, as we carry out in \cite{Xia23Operations}.

\subsection{A potential extension of Kudla's program}
As we mentioned in the beginning, the whole paper is a first step in the attempt of extending Kudla's program to mixed Shimura varieties.

Kudla's program is an important program in number theory relating the arithmetic intersection theory of special cycles to Fourier coefficients of Eisenstein series \cite{Kud97}. In the case of Shimura curves, it is worked out explicitly in \cite{KRY06}. In order to carry out Kudla's program in the mixed Shimura setting, we need to handle the following problems.

Firstly, we need to establish an Arakelov theory on mixed Shimura varieties. In the case of Shimura varieties, this is accomplished in \cite{BGKK05}. Their approach relies heavily on the fact that singularities on Shimura varieties are very mild, which fails in our setting. We have to handle $\mathcal{I}$-good singularities directly. This paper handles the infinity fiber. If one wants to establish Arakelov theory following the methods of Gillet--Soulé \cite{GS90, GS90a, GS90b}, one essential difficulty lies in establishing a Bott--Chern theory for $\mathcal{I}$-good Hermitian vector bundles. 

Secondly, our Chern--Weil formula indicates that the concept of special cycles on mixed Shimura varieties should be generalized to involve certain objects on the Riemann--Zariski space at the infinity fiber. In the case of universal elliptic curves, it is not clear to the author what the correct notion should be.

If we managed to solve these problems, then one should be able to study the arithmetic of mixed Shimura varieties following Kudla's idea.

We should mention that in the whole paper, we work with complex manifolds. But in reality, the important arithmetic moduli spaces are usually Deligne--Mumford stacks, so correspondingly the fibers at infinity are usually orbifolds instead of manifolds. However, extending the results in this paper to orbifolds is fairly straightforward, we will stick to the manifold case.

\subsection{Conventions}\label{subsec:conv}
In this paper, all vector bundles are assumed to be holomorphic. When the underlying manifold is quasi-projective, we will emphasize \emph{holomorphic} or \emph{algebraic} only when there is a risk of confusion.

Given a vector bundle $E$ on a manifold $X$, let $\mathcal{E}$ be the corresponding holomorphic locally free sheaf. Then convention for $\mathbb{P}E$ is $\underline{\Proj} \Sym \mathcal{E}^{\vee}$, which is different from the convention of Grothendieck. In general, we do not distinguish $E$ and $\mathcal{E}$ if there is no risk of confusion.

A variety over a field $k$ is a geometrically reduced, separated algebraic $k$-scheme, not necessarily geometrically integral. We choose this convention so that mixed Shimura varieties are indeed systems of varieties.

Given a sequence of rings or modules $A^k$ indexed by $k\in \mathbb{N}$, we will write $A=\bigoplus_{k}A^k$ without explicitly declaring the notation. This convention applies especially to Chow groups and Néron--Severi groups.

We set $\Delta\coloneqq \{z\in \mathbb{C}:|z|<1\}$.

\section*{Acknowledgements}
I would like to thank Elizabeth Wulcan, Yanbo Fang, Yaxiong Liu, Richard L\"ark\"ang, José Burgos Gil, Tam\'as Darvas, David Witt Nystr\"om, Yu Zhao, Dennis Eriksson, Moritz Kerz and Osamu Fujino for discussions. I am grateful to the referees for their valuable suggestions.

The author is supported by Knut och Alice Wallenbergs Stiftelse grant KAW 2021.0231.

\part{Non-pluripolar products on vector bundles}\label{part:1}
\section{Preliminaries}\label{sec:prel}
Most results in this section are known in the literature \cite{DX21, DX22, Xia21}. Readers with background in pluripotential theory can safely skip the whole section except \cref{thm:convdsmeasures} and \cref{lma:Igoodinsenspert}.

Let $X$ be a compact K\"ahler manifold of pure dimension $n$. 
Let $(L,h)$ be a Hermitian pseudo-effective line bundle, namely, $L$ is a holomorphic line bundle on $X$ and $h$ is a possibly singular plurisubharmonic (psh) metric on $L$. 
We write $\widehat{\Pic}(X)$ for the set of Hermitian pseudo-effective line bundles on $X$.

Take a smooth Hermitian metric $h_0$ on $L$. Let $\theta=c_1(L,h_0)$. We can identify $h$ with a function $\varphi\in \PSH(X,\theta)$ such that $h=h_0\exp(-\varphi)$. We write $\mathcal{I}(h)=\mathcal{I}(\varphi)$ for the multiplier ideal sheaf of $\varphi$: namely a local section of $\mathcal{I}(h)$ is a holomorphic function $f$ such that  $|f|_{h_0}^2 e^{-\varphi}$ is locally integrable. We will write 
\[
\ddc h=c_1(L,h)=\theta_{\varphi}=\theta+\ddc\varphi=\theta+\frac{\mathrm{i}}{2\pi}\partial \bar{\partial}\varphi.
\]

We define the volume of $(L,h)$ as
\[
\vol(L,h)=\vol(\theta,\varphi)\coloneqq \lim_{k\to\infty}k^{-n}h^0(X,L^k\otimes \mathcal{I}(kh)).
\]
The existence of the limit is proved in \cite{DX22, DX21}. 

Given $\varphi,\psi\in \PSH(X,\theta)$, write 
\[
\varphi\land \psi\coloneqq \sup \left\{\eta\in \PSH(X,\theta)\colon \eta\leq \varphi,\eta\leq \psi\right\}.
\]
The function $\varphi\land \psi$ is either $-\infty$ or in $\PSH(X,\theta)$ if $X$ is connected.

Recall the following projections:
\[
\begin{split}
P[\varphi]=&\sups_{c\in \mathbb{R}} \left((\varphi+c)\land 0\right),\\
P[\varphi]_{\mathcal{I}}=&\sup \{\psi\in \PSH(X,\theta)\colon \psi\leq 0,\mathcal{I}(k\psi)=\mathcal{I}(k\varphi)\text{ for all }k\in \mathbb{N}_{>0}\}.
\end{split}
\]
Both projections are in $\PSH(X,\theta)$. Here $\sups$ denotes the usc regularization of the supremum. The first projection is introduced in \cite{RWN14} and the second in \cite{DX22}.

The main result of \cite{DX21, DX22} shows 
\begin{equation}\label{eq:volpur}
    \vol(L,h)=\frac{1}{n!}\int_X \theta_{P[\varphi]_{\mathcal{I}}}^n.
\end{equation}
Here and in the whole paper, the Monge--Amp\`ere type products refer to the non-pluripolar products in the sense of \cite{BEGZ10}.

\begin{definition}\label{def:modelandgood}
We say $\varphi$ is \emph{model} (resp. \emph{$\mathcal{I}$-model}) or $h$ is \emph{model} (resp. \emph{$\mathcal{I}$-model}) (with respect to $h_0$ or $\theta$) if $P[\varphi]=\varphi$ (resp. $P[\varphi]_{\mathcal{I}}=\varphi$).

We say $\varphi$ is \emph{$\mathcal{I}$-good} (or $h$ is \emph{$\mathcal{I}$-good}, $(L,h)$ is \emph{$\mathcal{I}$-good}) if $P[\varphi]=P[\varphi]_{\mathcal{I}}$ and $\int_X c_1(L,h)^n>0$.

When we want to emphasize the dependence on the class $\theta$, we also say $\varphi$ is $\mathcal{I}$-good in $\PSH(X,\theta)$.
\end{definition}
Observe that being an $\mathcal{I}$-good metric is independent of the choice of the reference metric $h_0$.

\begin{definition}\label{def:ana}
 A potential $\varphi\in \PSH(X,\theta)$ is said to have \emph{analytic singularities} if for each $x\in X$, there is a neighbourhood $U_x\subseteq X$ of $x$ in the Euclidean topology, such that on $U_x$,
 \[
 \varphi=c\log\left(\sum_{j=1}^{N_x}|f_j|^2\right)+\psi,
 \]
 where $c\in \mathbb{Q}_{\geq 0}$, the $f_j$'s are analytic functions on $U_x$, $N_x\in \mathbb{N}$ is an integer depending on $x$ and $\psi\in L^{\infty}(U_x)$.
\end{definition}
A more special case of singularities is given by analytic singularities along a nc $\mathbb{Q}$-divisor. We define a slightly more general notion here:
\begin{definition}\label{def:anaD}
Let $D$ be an effective nc (normal crossing)  $\mathbb{R}$-divisor on $X$.   Let $D=\sum_i a_i D_i$ with $D_i$ being prime divisors and $a_i\in \mathbb{R}_{>0}$. We say that 
 $\varphi\in \PSH(X,\theta)$ has \emph{analytic singularities along $D$} or \emph{log singularities along $D$} if locally (in the Euclidean topology), 
 \[
 \varphi=\sum_i a_i\log|s_i|^2+\psi,
 \]
 where $s_i$ is a local holomorphic function that defines $D_i$ and $\psi$ is a bounded function.
\end{definition} 

In general, given any potential $\varphi$ having analytic singularities, one can find a composition of blowing-up with smooth centers $\pi\colon Y\rightarrow X$ such that $\pi^*\varphi$ has log singularities along some normal crossing $\mathbb{Q}$-divisor $D$ on $Y$. See \cite[Lemma~2.3.19]{MM07} for the proof. 

\begin{definition}
	Let $\varphi\in \PSH(X,\theta)$. A \emph{quasi-equisingular approximation} is a sequence $\varphi^j\in \PSH(X,\theta+\epsilon_i\omega)$ with $\epsilon_j\to 0$ such that
	\begin{enumerate}
		\item $\varphi^j\to \varphi$ in $L^1$.
		\item $\varphi^j$ has analytic singularities.
		\item $\varphi^{j+1}\leq \varphi^j$.
		\item For any $\delta>0$, $k>0$, there is $j_0>0$ such that for $j\geq j_0$,
		      \[
			      \mathcal{I}(k(1+\delta)\varphi^j)\subseteq \mathcal{I}(k\varphi)\subseteq \mathcal{I}(k\varphi^j).
		      \]
	\end{enumerate}
\end{definition}
The existence of a quasi-equisingular approximation follows from the arguments in \cite{Dem15, DPS01}.

The following is the main theorem in \cite{DX21, DX22}.
\begin{theorem}\label{thm:DXmain}
Assume that $\int_X c_1(L,h)^n>0$, identify $h$ with $\varphi\in \PSH(X,\theta)$ as above, then the following are equivalent:
\begin{enumerate}
    \item $h$ is $\mathcal{I}$-good.
    \item 
    \[
    \vol(L,h)=\frac{1}{n!}\int_X c_1(L,h)^n.
    \]
    \item There exists a sequence of $\varphi_i\in \PSH(X,\theta)$ with analytic singularities such that $\varphi_i \to \varphi$ with respect to the $d_S$-pseudometric.
\end{enumerate}
In case $\ddc h$ is a K\"ahler current (a closed positive $(1,1)$-current dominating some K\"ahler form), these conditions are equivalent to 
\begin{enumerate}[resume]
    \item Any quasi-equisingular approximation of $\varphi$ converges to $\varphi$ with respect to $d_S$.
\end{enumerate}
\end{theorem}
Another equivalent condition is given in \cref{cor:Imodcharbdiv}.

We will recall the definition of $d_S$ later. We will write $\hat{L}\in \widehat{\Pic}_{\mathcal{I}}(X)$ or $\varphi\in \PSH_{\mathcal{I}}(X,\theta)$ when $\hat{L}$ is $\mathcal{I}$-good.

Write
\[
\widehat{\Pic}(X)_{>0}=\left\{\hat{L}\in \widehat{\Pic}(X)\colon \int_X c_1(\hat{L})^n>0\right\}.
\]
Similarly, we can introduce
\[
\PSH(X,\theta)_{>0}\coloneqq \left\{\varphi\in \PSH(X,\theta)\colon \int_X \theta_{\varphi}^n>0\right\}.
\]

Generalizing \eqref{eq:volpur}, we can define the mixed volume of Hermitian pseudo-effective line bundles. Let $\hat{L}_i=(L_i,h_i)\in \widehat{\Pic}(X)_{>0}$ ($i=1,\ldots,n$).

Take smooth metrics $h_i'$ on $L_i$, write $\theta_i=c_1(L_i,h_i')$ and identify $h_i$ with $\varphi_i\in \PSH(X,\theta_i)$. Then we define the mixed volume as
\begin{equation}
\vol(\hat{L}_1,\ldots,\hat{L}_n)=\frac{1}{n!}\int_X (\theta_1+\ddc P[\varphi_1]_{\mathcal{I}})\wedge\cdots\wedge (\theta_n+\ddc P[\varphi_n]_{\mathcal{I}}).
\end{equation}

Let us recall the $d_S$ pseudo-metric defined on $\PSH(X,\theta)$ in \cite{DDNLmetric}. When the cohomology class $[\theta]$ is not big, we set $d_S=0$. If $[\theta]$ is big, $d_S$ is non-trivial.
We do not need the precise definition, it suffices to recall the following inequality:
\begin{equation}
d_S(\varphi,\psi)\leq \sum_{i=0}^n \left(2\int_X\theta^i_{\max\{\varphi,\psi\}}\wedge \theta_{V_{\theta}}^{n-i}-\int_X\theta^i_{\varphi}\wedge \theta_{V_{\theta}}^{n-i}-\int_X\theta^i_{\psi}\wedge \theta_{V_{\theta}}^{n-i}\right)\leq C_0d_S(\varphi,\psi),
\end{equation}
where 
\[
V_{\theta}\coloneqq \sup\{\varphi\in \PSH(X,\theta):\varphi\leq 0\}
\]
and $C_0>0$ is a constant. When we want to emphasize $\theta$, we write $d_{S,\theta}$ instead. 

\begin{lemma}\label{lma:kcapp}
Let $\varphi\in \PSH(X,\theta)_{>0}$. Then there is $\psi\in \PSH(X,\theta)_{>0}$, more singular than $\varphi$, such that $\theta_{\psi}$ is a K\"ahler current.

In particular, $\varphi$ is the increasing limit of a sequence $\varphi_j$ satisfying:
\begin{enumerate}
    \item Each $\theta_{\varphi_j}$ is a K\"ahler current.
    \item $\varphi_j$ converges to $\varphi$ with respect to $d_S$.
\end{enumerate}

\end{lemma}
\begin{proof}
It follows from \cite[Proposition~3.6]{DX21} that there is $\psi\in \PSH(X,\theta)$ such that $\theta_{\psi}$ is a K\"ahler current and $\psi$ is more singular than $\varphi$. 

As for the second part, we may assume that $\psi\leq \varphi$,
it suffices to take
\[
\varphi_j=(1-j^{-1})\varphi+j^{-1}\psi.
\]
(1) is then clear. For (2), it suffices to show the mass of $\varphi_j$ converges to the mass of $\varphi$, which is clear from the construction of $\varphi_j$.

\end{proof}

\begin{lemma}\label{lma:Igoodinsenspert}
Let $\varphi\in \PSH(X,\theta)_{>0}$. Take a K\"ahler form $\omega$ on $X$. Then $\varphi$ is $\mathcal{I}$-good in $\PSH(X,\theta)$ if and only if it is $\mathcal{I}$-good in $\PSH(X,\theta+\omega)$.
\end{lemma}
\begin{proof}
Assume that $\varphi$ is $\mathcal{I}$-good in $\PSH(X,\theta)$, then it is $\mathcal{I}$-good in $\PSH(X,\theta+\omega)$ by the proof of \cite[Corollary~4.4]{Xia21}. Conversely, if $\varphi$ is not $\mathcal{I}$-good in $\PSH(X,\theta)$, so that
\[
\int_X (\theta+\ddc \varphi)^n< \int_X (\theta+\ddc P_{\theta}[\varphi]_{\mathcal{I}})^n.
\]
It follows that
\[
\begin{aligned}
\int_X (\theta+\omega+\ddc \varphi)^n
=& \sum_{i=0}^n\binom{n}{i}\int_X \theta_{\varphi}^i\wedge \omega^{n-i}\\
<& \sum_{i=0}^n\binom{n}{i}\int_X \theta_{P_{\theta}[\varphi]_{\mathcal{I}}}^i\wedge \omega^{n-i}\\
=& \int_X (\theta+\omega+\ddc P_{\theta}[\varphi]_{\mathcal{I}})^n
\\ 
\leq &\int_X (\theta+\omega+\ddc P_{\theta+\omega}[\varphi]_{\mathcal{I}})^n.
\end{aligned}
\]
So $\varphi$ is not $\mathcal{I}$-good in $\PSH(X,\theta+\omega)$.
\end{proof}

\begin{proposition}\label{prop:Igoodsum}
Let $\varphi\in \PSH_{\mathcal{I}}(X,\theta)$, $\psi\in \PSH_{\mathcal{I}}(X,\theta')$, where $\theta'$ is a smooth real closed $(1,1)$-form representing some big cohomology classes. Then $\varphi+\psi\in \PSH_{\mathcal{I}}(X,\theta+\theta')$.
\end{proposition}
\begin{proof}
This follows from \cite[Corollary~4.5]{Xia21} and \cref{thm:DXmain} (3).
\end{proof}

\begin{theorem}\label{thm:convdsmeasures}
Let $\theta_1,\ldots,\theta_{a}$ ($a=0,\ldots,n$) be smooth real closed $(1,1)$-forms on $X$ representing big cohomology classes.
Let $\varphi^i_j\in \PSH(X,\theta_j)$ ($j=1,\ldots,a$) be decreasing (resp. increasing) sequences converging to $\varphi_j\in \PSH(X,\theta_j)_{>0}$ pointwisely (resp. almost everywhere). Assume that $\varphi_j^i\xrightarrow{d_S}\varphi_j$ as well. Then we have the weak convergence of the mixed Monge--Amp\`ere currents:
\begin{equation}\label{eq:weaktheta1toa}
\theta_{1,\varphi^i_1}\wedge \cdots \wedge \theta_{a,\varphi^i_a}\rightharpoonup \theta_{1,\varphi_1}\wedge \cdots \wedge \theta_{a,\varphi_a}
\end{equation}
as $i\to\infty$.
\end{theorem}
\begin{proof}
Let $\alpha$ be a weak limit of a subsequence of $\theta_{1,\varphi^i_1}\wedge \cdots \wedge \theta_{a,\varphi^i_a}$. We will argue that $\alpha=\theta_{1,\varphi_1}\wedge \cdots \wedge \theta_{a,\varphi_a}$.

Assume that $\alpha\neq \theta_{1,\varphi_1}\wedge \cdots \wedge \theta_{a,\varphi_a}$, then we can find K\"ahler forms $\omega_1,\ldots,\omega_{n-a}$ so that
\[
\int_X \alpha \wedge \omega_1\wedge\cdots\wedge \omega_{n-a}\neq \int_X \theta_{1,\varphi_1}\wedge \cdots \wedge \theta_{a,\varphi_a}\wedge \omega_1\wedge\cdots\wedge \omega_{n-a}.
\]
It follows from \cite[Theorem~4.2]{Xia21} that 
\[
\int_X \theta_{1,\varphi_1}\wedge \cdots \wedge \theta_{a,\varphi_a}\wedge \omega_1\wedge\cdots\wedge \omega_{n-a}=\lim_{j\to\infty}\int_X \theta_{1,\varphi_1^j}\wedge \cdots \wedge \theta_{a,\varphi_a^j}\wedge \omega_1\wedge\cdots\wedge \omega_{n-a}.
\]
It follows from \cite[Theorem~2.3]{DDNL18mono} that 
\[
\theta_{1,\varphi_1^j}\wedge \cdots \wedge \theta_{a,\varphi_a^j}\wedge \omega_1\wedge\cdots\wedge \omega_{n-a}\rightharpoonup \theta_{1,\varphi_1}\wedge \cdots \wedge \theta_{a,\varphi_a}\wedge  \omega_1\wedge\cdots\wedge \omega_{n-a}.
\]
So 
\[
\theta_{1,\varphi_1}\wedge \cdots \wedge \theta_{a,\varphi_a}\wedge \omega_1\wedge\cdots\wedge \omega_{n-a}=\alpha\wedge \omega_1\wedge\cdots\wedge \omega_{n-a},
\]
which is a contradiction.
\end{proof}

\begin{proposition}\label{prop:notionsbirational}
Let $\pi\colon Y\rightarrow X$ be a proper birational morphism and $Y$ is smooth. Let $\varphi\in \PSH(X,\theta)_{>0}$, corresponding to a psh metric $h$ on $L$. Then
\begin{enumerate}
    \item $\varphi$ is model (resp. $\mathcal{I}$-model, $\mathcal{I}$-good) if and only if $\pi^*\varphi$ is.
    \item $(L,h)$ is $\mathcal{I}$-good if and only if $(\pi^*L,\pi^*h)$ is.
    \item 
    \[
    \vol (L,h)=\vol (\pi^*L,\pi^*h),\quad \int_X c_1(L,h)^n=\int_Y c_1(\pi^*L,\pi^*h)^n.
    \]
\end{enumerate}
\end{proposition}
\begin{proof}
By Zariski's main theorem, $\pi\colon Y\rightarrow X$ has connected fibers, it follows that $\pi^*\colon \PSH(X,\theta)\rightarrow \PSH(Y,\pi^*\theta)$ is a bijection. From this, it follows that $\varphi$ is model if and only if $\pi^*\varphi$ is. 

For any $k\geq 0$, we have the well-known formula
\[
\pi_* (K_{Y/X}\otimes \mathcal{I}(k\pi^*\varphi))=\mathcal{I}(k\varphi).
\]
It follows that if $\varphi$ is $\mathcal{I}$-model, so is $\pi^*\varphi$. Conversely, if $\pi^*\varphi$ is $\mathcal{I}$-model, consider $\psi\in \PSH(X,\theta)$, $\psi\leq 0$ and $\mathcal{I}(k\psi)=\mathcal{I}(k\varphi)$ for all $k>0$, we want to show that $\psi\leq \varphi$ or equivalently, $\pi^*\psi\leq \pi^*\varphi$. By \cite[Corollary~2.16]{DX22}, we know that $\mathcal{I}(k\psi)=\mathcal{I}(k\varphi)$ for all $k>0$ implies that for all birational model $Z\rightarrow Y$ and any point $z\in Z$, we have $\nu(\psi,z)=\nu(\varphi,z)$. It follows that $\pi^*\psi\leq \pi^*\varphi$ as $\pi^*\varphi$ is $\mathcal{I}$-model.

From the locality of the non-pluripolar product and the fact that it puts no mass on proper analytic sets, we clearly have
\[
\int_X c_1(L,h)^n=\int_Y c_1(\pi^*L,\pi^*h)^n.
\]

It remains to show that
\[
\vol (L,h)=\vol (\pi^*L,\pi^*h).
\]
By \cite[Theorem~1.1]{DX21},
\[
\begin{aligned}
\vol (\pi^*L,\pi^*h)=&\lim_{k\to\infty}k^{-n}h^0(Y,K_{Y/X}\otimes \pi^*L^k\otimes \mathcal{I}(k\pi^*h))\\
=&\lim_{k\to\infty}k^{-n}h^0(X,L^k\otimes \mathcal{I}(kh))=\vol(L,h).
\end{aligned}
\]
\end{proof}

\begin{lemma}\label{lma:decseqplusvolumeimds}
Let $\varphi_i\in \PSH(X,\theta)$ ($i\in \mathbb{N}$) be a decreasing sequence with limit $\varphi\in \PSH(X,\theta)_{>0}$. Assume that 
\begin{equation}\label{eq:volcontviv}
    \lim_{i\to\infty}\int_X \theta_{\varphi_i}^n=\int_X \theta_{\varphi}^n,
\end{equation}
then $\varphi_i\xrightarrow{d_S}\varphi$.
\end{lemma}
\begin{proof}
By \cite[Corollary~4.7]{DDNLmetric}, if we set  $\psi\coloneqq \inf_i P[\varphi_i]$, then $\varphi_i\xrightarrow{d_S}\psi$. But observe that
\[
\varphi_i\leq P[\varphi_i]+\sup_X \varphi_i\leq P[\varphi_i]+\sup_X \varphi_0.
\]
Letting $i\to \infty$, we find that $\varphi$ is more singular that $\psi$. But by \eqref{eq:volcontviv}, $\int_X \theta_{\varphi}^n=\int_X \theta_{\psi}^n$, it follows that $\psi=P[\varphi]$ by \cite[Theorem~3.12]{DDNL18mono}. 
So $\varphi_i\xrightarrow{d_S}\varphi$.
\end{proof}

Let us mention that the notion of $\mathcal{I}$-goodness is global on $X$, but in some geometric situation, $\mathcal{I}$-goodness can be testified by the local growth condition of the metric. One notable example is the toroidal singularities introduced in \cite{BBGHdJ21}. 

As a special case, when $X$ is a smooth toric variety and $L$ is a toric invariant line bundle. Then any toric invariant psh metric on $L$ is $\mathcal{I}$-good. This is also an unpublished result of Yi Yao. In fact, this result also follows from a direct but somewhat lengthy computation.

There are also a large number of $\mathcal{I}$-good singularities with geometric origins. For example, one can construct them using \cite[Theorem~4.7]{DX22}.

A number of pluripotential-theoretic operations preserve the class of $\mathcal{I}$-good singularities, as explained in \cite[Lemma~2.9]{Xia23Operations}, \cref{prop:Igoodsum}. These give abundance of examples of $\mathcal{I}$-good singularities.

Another class of nice singularities on a line bundle is the so-called full mass singularities:
We say $\varphi$, $\hat{L}$ or $h$ has \emph{full mass} if $\int_X (\ddc h)^n=\int_X \theta_{V_{\theta}}^n$. See \cite{DDNL18fullmass}.

\section{Singular metrics on vector bundles}
Let $X$ be a complex manifold of pure dimension $n$.

\subsection{Singular Hermitian forms}
Let $V$, $V'$ be finite-dimensional complex linear spaces. 

We write $\Herm(V)$ for the set of semi-positive definite Hermitian forms on $V$. By definition, $h\in \Herm(V)$ is a sesquilinear form $h:V\times V\rightarrow \mathbb{C}$ such that
\begin{enumerate}
    \item $h(x,y)=\overline{h(y,x)}$ for all $x,y\in V$.
    \item $h(x,x)\geq 0$ for all $x\in V$.
\end{enumerate}
We can equivalently view $h$ as a map $V\rightarrow [0,\infty)$ by sending $x\in V$ to $h(x,x)$ satisfying 
\[
h(x+y)+h(x-y)=2h(x)+2h(y)
\]
for all $x,y\in V$ and
\[
h(cx)=|c|^2h(x)
\]
for all $x\in V$, $c\in \mathbb{C}$.

\begin{definition}
A \emph{singular Hermitian form} on $V$ is a map $h:V\rightarrow [0,\infty]$, such that 
\begin{enumerate}
    \item $V_{\fin}\coloneqq \{x\in V:h(x)<\infty\}$ is a linear subspace.
    \item $h|_{V_{\fin}}\in \Herm(V_{\fin})$.
\end{enumerate}
We write $\Herm^{\infty}(V)$ for the set of singular Hermitian forms on $V$.

We say $h$ is \emph{finite} if $h$ does not take the value $\infty$ and 
\emph{non-degenerate} if $h|_{V_{\fin}}$ is positive definite.
\end{definition}

Let $h\in \Herm^{\infty}(V)$. Write $N=h^{-1}(0)$. Observe that $N$ is a linear subspace of $V_{\fin}$. Then $h$ induces a non-degenerate Hermitian form $\tilde{h}$ on $V_{\fin}/N$. Let $\tilde{h}^{\vee}:(V_{\fin}/N)^{\vee}\rightarrow [0,\infty)$ denote the dual Hermitian form of $\tilde{h}$.

Write $V^{\vee}_{\fin}\coloneqq \{\ell\in V^{\vee}:\ell|_N=0\}$. Given $\ell\in V^{\vee}_{\fin}$, $\ell|_{V_{\fin}}$ therefore induces a linear form $\tilde{\ell}\in (V_{\fin}/N)^{\vee}$. We define
\[
h^{\vee}(\ell)=\tilde{h}^{\vee}(\tilde{\ell}).
\]
We extend $h^{\vee}$ to be $\infty$ outside $V^{\vee}_{\fin}$. It is easy to see that $h^{\vee}\in \Herm^{\infty}(V^{\vee})$.
\begin{definition}
Given $h\in \Herm^{\infty}(V)$, we call $h^{\vee}\in \Herm^{\infty}(V^{\vee})$ defined above the \emph{dual Hermitian form} of $h$.
\end{definition}
\begin{proposition}[{\cite[Lemma~3.1]{LRRS18}}]
Let $h\in \Herm^{\infty}(V)$, under the canonical identification $V\cong V^{\vee\vee}$, we have $h^{\vee\vee}=h$.
\end{proposition}

\begin{definition}\label{def:tensor}
Let $h\in \Herm^{\infty}(V)$, $h'\in \Herm^{\infty}(V')$. Assume one of the following conditions hold
\begin{enumerate}
    \item $h$, $h'$ are both non-degenerate or both finite.
    \item $h$ or $h'$ is both non-degenerate and finite.
\end{enumerate}
We define $h\otimes h'\in \Herm^{\infty}(V\otimes V')$ as follows: the set $(V\otimes V')_{\fin}$ is defined as 
\[
(V\otimes V')_{\fin}\coloneqq V_{\fin}\otimes V'_{\fin}.
\]
We define $(h\otimes h')_{(V\otimes V')_{\fin}}$ as the usual tensor product.
\end{definition}
The two conditions are to ensure that we do not get a product like $0\cdot \infty$. In fact, without these assumptions, \cref{prop:tens} fails. If one of these conditions are satisfied, we say $h\otimes h'$ \emph{is defined}.

By inspection, we find:
\begin{proposition}\label{prop:tens}
Let $h\in \Herm^{\infty}(V)$, $h'\in \Herm^{\infty}(V')$. Assume that $h\otimes h'$ is defined.
Then under the canonical identification $(V\otimes V')^{\vee}\rightarrow V^{\vee}\otimes V'^{\vee}$, we have $(h\otimes h')^{\vee}=h^{\vee}\otimes h'^{\vee}$.
\end{proposition}

\subsection{Singular metrics on vector bundles}\label{subsec:singmonvb}

\begin{definition}\label{def:singHerm}
Let $E$ be a holomorphic vector bundle on $X$. A \emph{singular Hermitian metric} $h$ on $X$ is an assignment
\[
X\in x\mapsto h_x:E_x\rightarrow [0,\infty],
\]
satisfying
\begin{enumerate}
    \item $h_x\in \Herm^{\infty}(E_x)$.
    \item For each local section $\xi$ of $E$, $h_x(\xi)$ is a measurable function in $x$.
\end{enumerate}

We say $h$ is \emph{finite} (resp. \emph{non-degenerate}) if $h_x$ is finite (resp. non-degenerate) for all $x\in X$.
\end{definition}

\begin{definition}\label{def:grif}
Let $\hat{E}=(E,h)$ be a holomorphic vector bundle $E$ on $X$ together with a singular Hermitian metric $h$. We say $\hat{E}$ is \emph{Griffiths negative} or $h$ is \emph{Griffiths negative} if 
\[
\chi_h(x,\xi)\coloneqq \log h_x(\xi) \quad (x\in X, \xi\in E_x)
\]
is psh on the total space of $E$. We say $\hat{E}$ is \emph{Griffiths positive} or $h$ is \emph{Griffiths positive} if the dual $\hat{E}^{\vee}$ is Griffiths negative.
\end{definition}
See \cite{Rau15} for details. We only mention that when $h$ is smooth, these notions reduce to the usual notion of Griffiths negativity and Griffiths positivity in terms of the curvature.

Observe that if $\hat{E}$ is Griffiths negative (resp. Griffiths positive), then $h$ is finite (resp. non-degenerate).

Let $\Vect(X)$ denote the category of vector bundles on $X$. Let $\widehat{\Vect}(X)$ denote the category of vector bundles endowed with a \emph{Griffiths positive} metric. A morphism between $\hat{E}=(E,h_E)$ and $\hat{F}=(F,h_F)$ is a morphism $E\rightarrow F$ in $\Vect(X)$.
Write $\widehat{\Pic}(X)$ for the full subcategory of $\widehat{\Vect}(X)$ consisting of pairs $(L,h_L)$ with $L$ of rank $1$.

We observe that the $\widehat{\Vect}(X)$'s for various $X$ (adding the constant singular metric $\infty$) is fibered over the category of connected complex manifolds in the sense of \cite[Exposé~VI]{SGA1}:
\begin{lemma}\label{lma:pullGp}
Let $f:Y\rightarrow X$ be a morphism of connected complex manifolds, $\hat{E}=(E,h_E)\in \widehat{\Vect}(X)$. Then the pull-back $f^*\hat{E}=(f^*E,f^*h_E)\in \widehat{\Vect}(Y)$ unless $f^*h_E$ is constant $\infty$.
\end{lemma}
\begin{proof}
It is clear that $(f^*\hat{E})^{\vee}=f^*\hat{E}^{\vee}$, so the problem is equivalent to the corresponding problem with negative curvature instead of positive curvature. Argue as in \cite[Lemma~2.3.2]{PT18}.
\end{proof}

A basic fact about Griffiths positive vector bundles is
\begin{proposition}[{\cite[Proposition~3.1]{BD08}}]\label{prop:appgri}
Let $\hat{E}=(E,h_E)\in \widehat{\Vect}(\Delta^n)$. Then up to replacing $\Delta^n$ by a smaller polydisk, there is a sequence of smooth Griffiths positive metrics $h^i$ on $E$ increasing pointwisely to $h_E$.
\end{proposition}
The sequence $h^i$ is increasing in the sense that for any vector $v\in E$, $h^i(v,v)$ is increasing in $i$.

\begin{corollary}
The tensor product (resp. direct sum) of $\hat{E},\hat{F}\in \widehat{\Vect}(X)$ is in $\widehat{\Vect}(X)$.
\end{corollary}
Observe that the tensor product of $h_E$ and $h_F$ is always defined as $h_E$, $h_F$ are both non-degenerate.

\subsection{The projective bundle}\label{subsec:projbundle}

Let $\hat{E}=(E,h_E)\in \widehat{\Vect}(X)$. Write $r+1=\rank E$. Let $p\colon \mathbb{P}E^{\vee}\rightarrow X$ be the projection. There is a natural injection
\[
\mathcal{O}_{\mathbb{P}E^{\vee}}(-1)\hookrightarrow p^*E^{\vee}.
\]
We remind the readers that our convention of $\mathbb{P}$ is such that $\mathbb{P}E^{\vee}$ consists of lines in $E^{\vee}$, see \cref{subsec:conv}.
We endow $\mathcal{O}_{\mathbb{P}E^{\vee}}(-1)$ with the induced subspace metric and write $\hO_{\mathbb{P}E^{\vee}}(-1)$ for the corresponding Hermitian line bundle.

Let us compute the local potential of this metric. Locally we may choose a basis $e_1,\ldots,e_{r+1}$ of $E$ and identify $E=X\times \mathbb{C}^{r+1}$. Consider a coordinate chart $U_i=\{(x,[\xi])\in \mathbb{P}E^{\vee}:\xi_i\neq 0\}$. In this chart, there is a canonical isomorphism
\[
\mathcal{O}_{\mathbb{P}E^{\vee}}(-1)\cong U_i\times \mathbb{C},\quad (x,[\xi];v)\mapsto (x,[\xi];v_i).
\]
Write $e^{\varphi_i}$ for the local potential corresponding to the metric on $\mathcal{O}(-1)|_{U_i}$ with respect to this coordinate chart.
We take the section $s_i$ of $\mathcal{O}_{\mathbb{P}E^{\vee}}(-1)$ corresponding to the $1$-section of $U_i\times \mathbb{C}$.
Now let $(x,[\xi])\in U_i$ and $v\in \mathcal{O}(-1)_{(x,[\xi])}=\mathbb{C}\xi$. Then
\[
h^{\vee}_x(v)=p^*h^{\vee}_{(x,[\xi])}(v)=|v/s_i(x,[\xi])|^2 e^{\varphi_i(x,[\xi])}=|v_i|^2e^{\varphi_i(x,[\xi])}.
\]
We thus find
\begin{equation}\label{eq:curv}
\varphi_i(x,[\xi])=\chi_{h^{\vee}}(x,\xi_0/\xi_i,\ldots,\xi_i/\xi_i,\ldots,\xi_r/\xi_i).
\end{equation}
In particular, $\hO(-1)$ is negatively curved. Write $\hO(1)$ for the dual Hermitian line bundle, we find that $\hO(1)$ is a positively curved line bundle on $\mathbb{P}E^{\vee}$. We remind the readers that in the whole paper, we only use the hat notation when the corresponding metric is positively-curved or negatively-curved. This differs from the common use in the literature.

We recall that we have the relative Segre embedding:
\begin{equation}\label{eq:iisoseg}
    i\colon \mathbb{P}E^{\vee}\times_X \mathbb{P}F^{\vee}\rightarrow \mathbb{P}(E\otimes F)^{\vee}.
\end{equation}
Under this embedding, we have
\[
i^*\mathcal{O}_{\mathbb{P}(E\otimes F)^{\vee}}(1)=\mathcal{O}_{\mathbb{P}E^{\vee}}(1)\boxtimes \mathcal{O}_{\mathbb{P}F^{\vee}}(1).
\]
By definition, 
\begin{equation}\label{eq:irestsegemb}
i^*\hO_{\mathbb{P}(E\otimes F)^{\vee}}(1)=\hO_{\mathbb{P}E^{\vee}}(1)\boxtimes \hO_{\mathbb{P}F^{\vee}}(1).
\end{equation}
When $\rank F=1$, $i$ is in fact an isomorphism.

\subsection{Finsler metrics}
Motivated by Griffiths' conjecture on ample vector bundles, Kobayashi \cite{Kob75} studied the Finsler metrics on a vector bundle, as a generalization of Hermitian metrics introduced above. By a simple observation of Kobayashi, Finsler metrics on a vector bundle $E$ on $X$ are in bijective correspondence with Hermitian metrics on $\mathcal{O}_{\mathbb{P}E^{\vee}}(1)$, so we will make use of the following convenient definition:
\begin{definition}
A \emph{Finsler metric} on $E\in \Vect(X)$ is a singular Hermitian metric $h$ on $\mathcal{O}_{\mathbb{P}E^{\vee}}(1)$. 
We say $h$ is \emph{Griffiths positive} if $h$ is positively curved as a metric on $\mathcal{O}_{\mathbb{P}E^{\vee}}(1)$.

We will write $\Fins(X)$ for the category of $\hat{E}=(E,h)$ consisting of a holomorphic vector bundle $E$ on $X$ and a Griffiths positive Finsler metric $h$ on $E$: a morphism from $(E,h)$ to $(F,h')$ is just a morphism from $E$ to $F$.
\end{definition}
\begin{remark}
When $\rank E=1$, a Finsler metric on $E$ is the same as a singular Hermitian metric on $E$.
\end{remark}

\begin{remark}
Note that each Hermitian metric induces canonically a Finsler metric, as we explained in \cref{subsec:projbundle}.
The notions of Griffiths positivity coincide in these two cases by the explicit formula \eqref{eq:curv}, this is also proved in \cite[Proposition~5.2]{LRSW18}. 

On the other hand, by \cite[Theorem~7.1]{LSY13}, given a \emph{smooth non-degenerate} Hermitian metric $h$ on $E$, we can recover the metric $h$ from the induced Finsler metric together with the induced metric on $K_{\mathbb{P}E^{\vee}/X}$. So we do not lose too much information when replacing the Hermitian metric by the corresponding Finsler metric.
\end{remark}

There are several motivations for the use of Finsler metrics: usually natural constructions in potential theory only lead to metrics on $\mathcal{O}(1)$. In general, there is no effective way of inducing a Griffiths positive metric on $E$ from a metric on $\mathcal{O}(1)$, so we are forced to consider $\mathcal{O}(1)$ instead. This is related to the difficulty in Griffiths conjecture.
On the other hand, Finsler metrics do occur naturally in many problems, see \cite{DW22} for example.

Observe that $\Fins(X)$ (include the singular metric $\infty$) is fibered over the category of connected complex manifolds: given a morphism of connected complex manifolds $f:Y\rightarrow X$, we can define $f^*\hat{E}$ for all $\hat{E}=(E,h_E)\in \Fins(X)$: the underlying vector bundle of $f^*\hat{E}$ is just $f^*E$; in order to define the Finsler metric, consider the Cartesian square
\[
\begin{tikzcd}
\mathbb{P}(f^*E)^{\vee} \arrow[r,"f'"] \arrow[d] \arrow[rd, "\square", phantom] & \mathbb{P}E^{\vee} \arrow[d] \\
Y \arrow[r, "f"]                      & X          
\end{tikzcd}.
\]
It is easy to see that $\mathcal{O}_{\mathbb{P}(f^*E)^{\vee}}(1)=f'^*\mathcal{O}_{\mathbb{P}E^{\vee}}(1)$ and we just define the metric on $f^*\hat{E}$ as the pull-back of the Finsler metric $h_E$, which is a Finsler metric on $f^*E$ as long as it is not identically $\infty$. 
When $\hat{E}\in \widehat{\Vect}(X)$, this construction coincides with the construction in \cref{lma:pullGp}.

Next, let us define the tensor product between $\hat{E}=(E,h_E)\in \Fins(X)$ and $\hat{L}=(L,h_L)\in \widehat{\Pic}(X)$. By definition, the underlying vector bundle of $\hat{E}\otimes \hat{L}$ is just $E\otimes L$. The Hermitian metric on on $\mathcal{O}_{\mathbb{P}(E\otimes L)^{\vee}}(1)$ is given by the tensor product between the induced metric on $\mathcal{O}_{\mathbb{P}E^{\vee}}(1)$ and $h_L$ under the canonical isomorphism \eqref{eq:iisoseg}. More generally, we can define the tensor product between $\hat{E}\in \Fins(X)$ and $\hat{L}\in \widehat{\Pic}(\mathbb{P}E^{\vee})$ in the same way.

\subsection{Special singularities on vector bundles}

\begin{definition}\label{def:goodmetric}
Assume that $X$ is projective and $D$ is a snc divisor in $X$. Let $E$ be a vector bundle on $X$. A smooth Hermitian metric $h$ on $E|_{X\setminus D}$ is \emph{good} with respect to $D$ if for any $x\in D$, we can find a coordinate chart $U\cong\Delta^n$ containing $x$ on which $E$ is trivialized by sections $e_1,\ldots,e_{r+1}$ and such that $D\cap U=(z_1\cdots z_k=0)$, such that if we set $h_{ij}=h(e_i,e_j)$, then
\begin{enumerate}
    \item $|h_{ij}|$, $(\det h)^{-1}$ are both bounded from above by $C \sum_{i=1}^k (-\log |z_i|)^m$ for some $C$ and $m$.
    \item The 1-forms $(\partial h\cdot h^{-1})_{ij}$ are good forms. 
\end{enumerate}
We also say $(E,h)$ is good with respect to $D$.
\end{definition}
Recall that a form $\alpha$ on $X\setminus D$ is good if $\alpha$ and $\mathrm{d}\alpha$ both have at worst Poincar\'e growth at the boundary $D$. See \cite{Mum77} for details.

\begin{proposition}
Let $(E,h_E)$ be a good vector bundle on $X\setminus D$ (with respect to $D$) and $L$ be a line bundle on $X$ with a smooth non-degenerate metric $h_L$. Then $(E,h_E)\otimes (L,h_L)|_{X\setminus D}$ is good with respect to $D$. 
\end{proposition}
\begin{proof}
The problem is local, we can fix $U$, $e_i$, $h_{ij}$ as in \cref{def:goodmetric}. We trivialize $L$ on $U$ by a holomorphic section $e_0$ and write $\rho=h_L(e_0,e_0)$. Then $\rho$ is a bounded smooth function bounded away from $0$. Let $h'_{ij}=\rho h_{ij}$.  We will verify the two conditions. The first condition is obvious by now. As for the second, let us compute
\[
\sum_j\partial h'_{ij}\cdot h'^{jk}=\rho^{-1}\sum_j\partial(\rho h_{ij})\cdot h^{jk}=\sum_j \partial h_{ij}\cdot h^{jk}+\rho^{-1}\partial \rho \cdot \delta_{ik}. 
\]
Both parts are obviously good forms.
\end{proof}
In particular, in order to determine the goodness, we can always make a twist of the original bundle. In most cases, we can therefore assume that $E$ has some positivity properties.
A more general twist is as follows:
\begin{lemma}\label{lma:extend}Assume that $X$ is projective.
Assume that
$\hat{E}=(E,h_E)$ is a good Hermitian vector bundle with respect to a snc divisor $D$ in $X$. Let $\rank E=r+1$.
Let $\hat{L}=(L,h_L)\in \widehat{\Pic}(X)$ be an ample
line bundle together with a psh metric $h_L$ with log singularities along some snc $\mathbb{Q}$-divisor $D'$ with $|D'|=|D|$ such that $\ddc h_L-[D']$ is a smooth form.
Let $p\colon \mathbb{P}E^{\vee}\rightarrow X$ be the natural projection.
Assume that $\hO(1)\otimes p^*\hat{L}\in \widehat{\Pic}(\mathbb{P}E^{\vee}|_{X\setminus D})$ (namely, the tensor product metric is Griffiths positive), then
\begin{equation}\label{eq:temp12}
(\mathcal{O}(1)+p^*(L-D'))^{n+r}=\int_{\mathbb{P}E^{\vee}} \left( c_1(\hO(1))+p^*c_1(\hat{L})\right)^{n+r}.
\end{equation}
\end{lemma}
The product on the right-hand side is the non-pluripolar product. The notion of log singularities is introduced in \cref{def:anaD}.
\begin{proof}
First observe that the metric
$\hO(1)\otimes p^*\hat{L}\in \widehat{\Pic}(\mathbb{P}E^{\vee}|_{X\setminus D})$ tends to $\infty$ everywhere along $p^*D$: in fact, by Property~(1) in \cref{def:goodmetric} and \eqref{eq:curv}, up to a bounded term, the potential of the singular metric on $\hO(1)$ is bounded from above by $\log \sum_i (-\log |z_i|)^m$ for some $m>0$ and the $z_i$'s are the local defining equations for the components of $D$. By definition of the log singularity, we find that the product metric is singular everywhere along $D$.

It follows from Grauert--Remmert's extension theorem \cite{GR56} that $\hO(1)\otimes p^*\hat{L}$ admits a unique extension to $\widehat{\Pic}(\mathbb{P}E^{\vee})$, which we denote by the same notation. In particular, the right-hand side of \eqref{eq:temp12} makes sense. Expand both sides of \eqref{eq:temp12} using the binomial formula, we find that it suffices to prove the following: for any $i\geq 0$, and smooth closed form $T$ on $X$ representing a cohomology class $\alpha$, $s_i(E|_{X\setminus D},h_E)\wedge T$
represents $s_i(E)\wedge \alpha$.
This follows from the same argument as \cite[Theorem~1.4]{Mum77}.

\end{proof}

\begin{definition} \label{def:anasing}
Consider $\hat{E}\in \widehat{\Vect}(X)$. We say $\hat{E}$ has \emph{analytic singularities} if $\hO(1)$ has analytic singularities.

Similarly, we say $\hat{E}=(E,h_E) \in \Fins(X)$ has \emph{analytic singularities} if the metric $h_E$ has analytic singularities as a metric on $\hO(1)$.
\end{definition}
This is the type of singularities studied in \cite{LRSW18}.

\begin{definition}\label{def:tor}
Let $\Delta^n$ be the standard polydisk of dimension $n$. Consider the divisor $D=(z_1\cdots z_k=0)$. Let $(E\cong \Delta^n\times \mathbb{C}^{r+1},h_E)$ be a vector bundle on $\Delta^n$ together with a singular Griffiths positive metric. We assume that $h_E$ is locally bounded on $\Delta^n\setminus D$. We say $(E,h_E)$ has \emph{toroidal singularities} along $D$ (at $0$) if for each $i=0,\ldots,r$, the function
\[
\chi_{h_E^{\vee}}(x_1,\ldots,x_n,\xi_0/\xi_i,\ldots,\xi_i/\xi_i,\ldots,\xi_{r+1}/\xi_i)
\]
has the form
\[
\gamma(-\log|x_1|,\ldots,-\log|x_k|)+\text{bounded term}
\]
on $\Delta'\times U$, where $\Delta'\subseteq \Delta^n$ is a smaller polydisk centered at $0$,  $U\subseteq \mathbb{C}^{r}$ is any small disk in $\mathbb{C}^{r}$, $\gamma$ is a convex bounded from above Lipschitz function defined on $\{y\in \mathbb{R}^n:y_1\geq M,\ldots, y_n\geq M\}$ for some large enough $M$.

Globally, given a snc divisor $D$ in $X$ and $\hat{E}\in \widehat{\Vect}(X)$ such that $h_E$ is locally bounded on $X\setminus D$, we say $\hat{E}$ has \emph{toroidal singularities} along $D$ if the restriction of $\hat{E}$ each small enough local coordinate chart has toroidal singularities.
\end{definition}
This is a straightforward extension of the definition in \cite{BBGHdJ21}. Equivalently, $\hat{E}$ has toroidal singularities if $\hO(1)$ has toroidal singularities in the sense of \cite[Definition~3.10]{BBGHdJ21}.

\section{Pull-backs of currents}\label{sec:pull}

Let $X$ be a complex manifold of pure dimension $n$. A positive current on $X$ could have several different meanings. In this paper, we say a closed positive $(p,p)$-current $T$ on $X$ is positive if for all smooth $(1,0)$-forms $\alpha_1,\ldots,\alpha_{n-p}$, the distribution 
\[
T\wedge \mathrm{i}\alpha_1\wedge \overline{\alpha_1}\wedge \cdots \wedge \mathrm{i}\alpha_{n-p}\wedge \overline{\alpha_{n-p}}
\]
is positive.
\begin{definition}\label{def:dsh}
A \emph{closed dsh current} of bi-dimension $(p,p)$ on $X$ is a current $T$ of bi-dimension $(p,p)$ of the form $T=S_1-S_2$, where $S_1$, $S_2$ are closed positive currents of bi-dimension $(p,p)$ on $X$. 
The set of such currents is denoted by $\hat{Z}_p(X)$.

Given $T\in \hat{Z}_p(X)$, we will call any expression $T=S_1-S_2$ as above a \emph{decomposition} of $T$.
\end{definition}
Observe that $\hat{Z}_p(X)$ is a real vector space.
We endow $\hat{Z}_p(X)$ with the weak topology of currents. Thus $\hat{Z}_p(X)$ becomes a locally convex topological vector space.

\begin{lemma}\label{lma:nomassnonpp}
Let $T\in \hat{Z}_p(X)$. Suppose that $T$ puts no mass on a complete pluripolar set $A\subseteq X$, then there is a decomposition $T=S_1-S_2$ such that $S_1$, $S_2$ put no masses on $A$. 
\end{lemma}
In the whole paper, a \emph{complete pluripolar set} on a compact K\"ahler manifold $X$ means a subset $Z\subseteq X$ such that for any $x\in X$, we can find a neighbourhood $U\subseteq X$ of $x$ and a plurisubharmonic function $\varphi$ on $U$ such that $Z\cap U=\{x\in U:\varphi(x)=-\infty\}$. A complete pluripolar set is sometimes known as a \emph{locally complete pluripolar set} in the literature. 
\begin{proof}

In fact, let $T=S_1-S_2$ be an arbitrary decomposition, then we consider
\begin{equation}\label{eq:temp13}
T=\mathds{1}_{X\setminus A}S_1-\mathds{1}_{X\setminus A}S_2.
\end{equation}
It follows from \cite[Remark~1.10]{BEGZ10} that $\mathds{1}_{X\setminus A}S_i$ ($i=1,2$) are both closed and positive. Thus, \eqref{eq:temp13} is the desired decomposition.
\end{proof}

We say a morphism $f:Y\rightarrow X$ between complex analytic spaces has \emph{pure relative dimension} $d$ if each of the fibers has pure dimension $d$ (or equidimensional of dimension $d$). In the literature, this condition is also known as has \emph{relative dimension} $d$. When $X$ is smooth and $Y$ is Cohen--Macaulay, both $X$ and $Y$ are equidimensional and $\dim Y-\dim X=d$, it follows from miracle flatness that $f$ is flat. But we still prefer to say $f$ is flat of pure relative dimension $d$ in this case, with non-smooth extensions of the results below in mind.

\begin{theorem}[Dinh--Sibony]\label{thm:DS}
Let $f:Y\rightarrow X$ be a flat morphism  of pure relative dimension $r$ between complex manifolds $Y$ and $X$ of pure dimensions $n+r$ and $n$.
Then there is a unique continuous linear map $f^*\colon \widehat{Z}_a(X)\rightarrow \widehat{Z}_{a+r}(Y)$ such that the followings hold:
\begin{enumerate}
    \item When the current is represented by a form, $f^*$ is the usual pull-back. 
    \item When $f$ is an open immersion, $f^*$ is the usual restriction.
    \item The pull-back is local on $X$: consider $T\in \widehat{Z}_a(X)$ and an open subset $U\subseteq X$, 
    \[
        f^* (T|_U)=f^*T|_{f^{-1}U}.
    \]
    \item In the case of $(1,1)$-currents, this pull-back is the usual one (namely, pulling back the local K\"ahler potentials).
    \item If $u$ is a locally bounded psh function on $X$, $T\in \widehat{Z}_a(X)$, then
    \[
    f^*(\ddc u \wedge T)=\ddc f^*u \wedge f^*T.
    \]
    \item If $T\in \widehat{Z}_a(X)$ is closed positive, then so is $f^*T$.    
\end{enumerate}
\end{theorem}
\begin{remark}
As pointed out by the referee, for the purpose of the current paper, it suffices to know pull-back currents along smooth morphisms. In this case, the pull-back is easy to construct directly, as the dual linear map of the fibral integration. The properties listed in this theorem can be easily verified directly.
\end{remark}

Here in (5), the product is taken in the sense of Bedford--Taylor. Recall that between complex manifolds, a smooth morphism is the same as a submersive morphism.

We briefly recall the construction of $f^*$. Let $\Gamma_f\subseteq Y\times X$ be the the graph of $f$. Write $p_1\colon \Gamma_f\rightarrow Y$ and $p_2\colon \Gamma_f\rightarrow X$ the two natural projections.
We wish to define
\[
f^*T\coloneqq p_{1*} (p_2^*T\wedge [\Gamma_f]).
\]
Of course, we need to make sense of $p_2^*T\wedge [\Gamma_f]$ as currents. Locally approximate $T$ by smooth forms $T_i$, then we define $p_2^*T\wedge [\Gamma_f]$ as the weak limit of $p_2^*T_i\wedge [\Gamma_f]$. The existence of the the limit and its independence of the choice of $T_i$ are non-trivial facts proved in \cite{DS07}.

\begin{proof}
See \cite{DS07} for the proof of the existence and continuity of $f^*$ and (1), (2), (3), (4), (6). These are not explicit in \cite{DS07}, however, they are all clear from the construction of $f^*$.
In order to prove (5), we may assume that $u$ is smooth, $T$ is closed and positive and $X=\Delta^n$. By approximation, we may assume that $T$ is a form. In this case, (5) is clear.
\end{proof}

\begin{corollary}
Let $X,Y,Z$ be complex manifolds of pure dimensions $n$, $n+r$, $n+r+r'$.
Let $f:Y\rightarrow X$ and $g:Z\rightarrow Y$ be flat morphisms of pure relative dimensions $r$ and $r'$. Then
\[
(fg)^*=g^*f^*\colon \widehat{Z}_a(X)\rightarrow \widehat{Z}_{a+r+r'}(X).
\]
\end{corollary}
\begin{proof}
Take $T\in \widehat{Z}_a(X)$, we want to show that
\begin{equation}\label{eq:temp1}
    (fg)^*T=g^*f^*T.
\end{equation}
As both sides of \eqref{eq:temp1} are linear in $T$, we may assume that $T$ is a closed positive current. As both sides of \eqref{eq:temp1} are local on $X$, we may assume that $X=\Delta^n$. By continuity of pull-back, we may assume that $T$ is represented by a form. In this case, \eqref{eq:temp1} is obvious.
\end{proof}

We write $\mathcal{A}^{a,a}(X)$ (resp. $\mathcal{A}^{a,a}_c(X)$) for the set of smooth real-valued $(a,a)$-forms (resp. smooth real-valued $(a,a)$-forms with compact supports) on $X$. 
\begin{corollary}\label{cor:adj1}
Let $X,Y$ be complex manifolds of pure dimension $n$, $m$. Let $f:Y\rightarrow X$ be a smooth morphism  of pure relative dimension $m-n$. Consider $T\in \widehat{Z}_{a}(X)$ and $\alpha \in \mathcal{A}^{a+m-n,a+m-n}_c(Y)$, then
\begin{equation}\label{eq:adjpullbackcur}
\int_Y \alpha\wedge f^*T=\int_X f_*\alpha\wedge T.
\end{equation}
\end{corollary}
\begin{proof}
As $f$ is smooth, $f_*\alpha\in \mathcal{A}_c^{a,a}(X)$, so the right-hand side of \eqref{eq:adjpullbackcur} makes sense. The problem \eqref{eq:adjpullbackcur} is local on $X$, so we may assume that $X$ is the unit polydisk $\Delta^n$. As both sides of \eqref{eq:adjpullbackcur} are linear in $T$, we may further assume that $T$ is closed and positive in $\Delta^n$. By approximation, we may further assume that $T$ is represented by a form, in which case, \eqref{eq:adjpullbackcur} is clear.
\end{proof}

\begin{corollary}\label{cor:adjunction}
Let $X,Y$ be complex manifolds of pure dimension $n$, $m$. Let $f:Y\rightarrow X$ be a smooth morphism  of pure relative dimension $m-n$. Consider $T\in \widehat{Z}_{a}(X)$ and $\alpha \in \mathcal{A}^{b,b}_c(Y)$. Then
\begin{equation}
    f_*(\alpha\wedge f^*T)=f_*\alpha\wedge T.
\end{equation}
\end{corollary}
\begin{proof}
Let $\beta\in \mathcal{A}_c^{a-b+m-n,a-b+m-n}(X)$. We need to show that
\begin{equation}\label{eq:temp2}
\int_Y f^*\beta\wedge \alpha\wedge f^*T=\int_X \beta\wedge f_*\alpha\wedge T.
\end{equation}
By \cref{cor:adj1}, we can rewrite the left-hand side of \eqref{eq:temp2} as
\[
\int_X f_*(f^*\beta\wedge \alpha)\wedge T.
\]
Thus \eqref{eq:temp2} follows from the adjunction formula of forms.
\end{proof}

\begin{corollary}\label{cor:comppullbackpush}
Let $X,Y,X'$ be complex manifolds of pure dimension $n$, $m$, $k$. Let $f:Y\rightarrow X$ be a proper map and $g\colon X'\rightarrow X$ be a smooth morphism of pure relative dimension $k-n$.
Consider the following Cartesian square 
\[
\begin{tikzcd}
Y' \arrow[r,"f'"] \arrow[d,"g'"] \arrow[rd, "\square", phantom] & X' \arrow[d,"g"] \\
Y \arrow[r,"f"]                                    & X 
\end{tikzcd}.
\]
Then
\[
f'_* g'^*=g^*f_*\colon \widehat{Z}_a(Y)\rightarrow \widehat{Z}_{k-n+a}(X').
\]
\end{corollary}
\begin{proof}
Take $T\in \widehat{Z}_a(Y)$. We need to prove
\[
    f'_* g'^*T=g^*f_*T.
\]
We may assume that $T$ is closed positive current. Take $\alpha\in \mathcal{A}_c^{k-n+a,k-n+a}(X')$, then we are reduced to show
\[
    \int_{Y'} f'^{*}\alpha\wedge g'^*T=\int_{X'} \alpha\wedge g^*f_*T.
\]
By \cref{cor:adj1}, this is equivalent to
\[
    \int_{Y} g'_*f'^{*}\alpha\wedge T=\int_Y f^*g_*\alpha\wedge T.
\]
So we are reduced to show
\[
g'_*f'^{*}\alpha=f^*g_*\alpha,
\]
which is nothing but the naturality of fiber integration.
\end{proof}

\section{Relative non-pluripolar products}\label{sec:relnpp}
We fix a complex manifold $X$ of pure dimension $n$. We will extend Vu's theory \cite{Vu20} of relative non-pluripolar products in this section and prove a few functoriality results.

\begin{definition}
Suppose $T_1,\ldots,T_m$ are closed positive $(1,1)$-currents on $X$ and $T$ is a closed positive current of bidimension $(p,p)$ on $X$. Then we say that the relative non-pluripolar product $T_1\wedge\cdots\wedge T_m\cap T$ is \emph{well-defined} if for each $x\in X$, we can take a local chart $U\cong \Delta^n$ containing $x$, on which $T_i=\ddc u_i$ for some psh functions $u_i$ on $U$, so that if we define 
\begin{equation}\label{eq:defRk}
R_k=\ddc \max\{u_1,-k\}\wedge \cdots \wedge \ddc \max\{u_m,-k\}\wedge T,\quad k\in \mathbb{N}
\end{equation}
using Bedford--Taylor theory, then
\begin{equation}\label{eq:supnormKfinite}
\sup_{k\in \mathbb{N}} \|\mathds{1}_{\{u_1>-k,\ldots, u_m>-k\}}R_k\|_K<\infty
\end{equation}
for each compact subset $K\subseteq U$. 
Here after choosing a strictly positive smooth real $(1,1)$-form $\omega$ on $X$, we let
\[
\|S\|_K\coloneqq \|(S\wedge \omega^{n-a})|_K\|_{\text{total variation}}
\]
for any $(a,a)$-current $S$ on $X$. The condition \eqref{eq:supnormKfinite} does not depend on the choice of $\omega$.

In this case, we define the \emph{relative non-pluripolar product}
\begin{equation}\label{eq:relativenppdef}
T_1\wedge\cdots\wedge T_m\cap T\coloneqq \lim_{k\to\infty}R_k,
\end{equation}
where the limit is a weak limit of currents.
\end{definition}

When $T=[X]$ is the current of integration along $X$, $T_1\wedge\cdots\wedge T_m\cap T$ is nothing but the non-pluripolar product studied in \cite{BEGZ10}. The general notion is due to \cite{Vu20}.

\begin{lemma}[{\cite[Lemma~3.1]{Vu20}}]
Suppose $T_1,\ldots,T_m$ are closed positive $(1,1)$-currents on $X$ and $T$ is a closed positive current of bidimension $(p,p)$ on $X$. Suppose that $T_1\wedge\cdots\wedge T_m\cap T$ is well-defined.
Then the limit in \eqref{eq:relativenppdef} exists. Moreover, for each Borel measurable $(p-m,p-m)$-form $\Phi$ with locally bounded coefficients on $X$ satisfying $\Supp \Phi\Subset X$, we have
\[
\int_X(\Phi,T_1\wedge\cdots\wedge T_m\cap T)=\lim_{k\to\infty}\int_X (\Phi,R_k),
\]
where $R_k$ is defined as in \eqref{eq:defRk}.
\end{lemma}
Here the pairing is that between a test form and a current.

\begin{lemma}\label{lma:linearinT}
Suppose $T_1,\ldots,T_m$ are closed positive $(1,1)$-currents on $X$ and $T_1$, $T_2$ are closed positive currents of bidimension $(p,p)$ on $X$. Take $\lambda_1,\lambda_2\geq 0$. Assume that $T_1\wedge\cdots\wedge T_m\cap T_1$ and $T_1\wedge\cdots\wedge T_m\cap T_2$ are both well-defined, then so is $T_1\wedge\cdots\wedge T_m\cap (\lambda_1 T_1+\lambda_2 T_2)$ and
\[
T_1\wedge\cdots\wedge T_m\cap (\lambda_1 T_1+\lambda_2 T_2)=\lambda_1 (T_1\wedge\cdots\wedge T_m\cap T_1)+\lambda_2 (T_1\wedge\cdots\wedge T_m\cap T_2).
\]
\end{lemma}
This is obvious from the definition.

\begin{definition}
Suppose $T_1,\ldots,T_m$ are closed positive $(1,1)$-currents on $X$ and $T\in  \hat{Z}_p(X)$.
We say the relative non-pluripolar product $T_1\wedge\cdots\wedge T_m\cap T$ is \emph{well-defined} if there is a decomposition $T=S_1-S_2$ such that $T_1\wedge\cdots\wedge T_m\cap S_i$ ($i=1,2$) are both well-defined.

In this case, we define
\[
T_1\wedge\cdots\wedge T_m\cap T\coloneqq T_1\wedge\cdots\wedge T_m\cap S_1-T_1\wedge\cdots\wedge T_m\cap S_2.
\]
\end{definition}
Observe that the product $T_1\wedge\cdots\wedge T_m\cap T$ is symmetric in $T_i$.

\begin{lemma}
Suppose $T_1,\ldots,T_m$ are closed positive $(1,1)$-currents on $X$ and $T\in  \hat{Z}_p(X)$.
Assume that $T_1\wedge\cdots\wedge T_m\cap T$ is well-defined, then $T_1\wedge\cdots\wedge T_m\cap T$ does not depend on the choice of the decomposition $T=S_1-S_2$.
\end{lemma}
\begin{proof}
This follows immediately from \cref{lma:linearinT}.
\end{proof}

\begin{proposition}
Suppose $T_1,\ldots,T_m$ are closed positive $(1,1)$-currents on $X$ and $T_1, T_2\in \hat{Z}_p(X)$. Take $\lambda_1,\lambda_2 \in \mathbb{R}$. Assume that $T_1\wedge\cdots\wedge T_m\cap T_1$ and $T_1\wedge\cdots\wedge T_m\cap T_2$ are both well-defined, then so is $T_1\wedge\cdots\wedge T_m\cap (\lambda_1 T_1+\lambda_2 T_2)$ and
\begin{equation}\label{eq:npprellinear}
T_1\wedge\cdots\wedge T_m\cap  (\lambda_1 T_1+\lambda_2 T_2)=\lambda_1 (T_1\wedge\cdots\wedge T_m\cap T_1)+\lambda_2 (T_1\wedge\cdots\wedge T_m\cap T_2).
\end{equation}
\end{proposition}
\begin{proof}
We may assume that $\lambda_1,\lambda_2\geq 0$.

By definition, we can find closed positive currents $S_1^1$, $S_1^2$, $S_2^1$, $S_2^2$ of bidimension $(p,p)$ such that $T_1=S_1^1-S_1^2$, $T_2=S_2^1-S_2^2$ and $T_1\wedge \cdots \wedge T_m\cap S_i^j$ ($i=1,2$, $j=1,2$) are all well-defined. Then $T_1\wedge \cdots \wedge T_m\cap (\lambda_1 S_1^j+\lambda_2 S_2^j)$ ($j=1,2$) are both well-defined. Hence so is $T_1\wedge\cdots\wedge T_m\cap  (\lambda_1 T_1+\lambda_2 T_2)$ and \eqref{eq:npprellinear} follows.
\end{proof}

On the other hand, the product $T_1\wedge\cdots\wedge T_m\cap T$ is not additive in $T_i$ in general. 
\begin{example}\label{ex:c1notadd}
    When $T$ is the current of integration of a complex submanifold $V$ of $X$ such that $V$ is contained in the polar locus of $T_1$. Take another current $T_1'$ such that $T_1'\wedge T_2\wedge \cdots\wedge T_m\cap T>0$. In this case, by \cite[Remark~3.8]{Vu20}, we have
    \[
        0=(T_1+T_1')\wedge T_2\wedge \cdots\wedge T_m\cap T< T_1\wedge T_2\wedge \cdots\wedge T_m\cap T+T_1'\wedge T_2\wedge \cdots\wedge T_m\cap T.
    \]
\end{example}

\begin{proposition}\label{prop:relnpplinearinTi}
Suppose $T_1,\ldots,T_m$ and $T_1'$ are closed positive $(1,1)$-currents on $X$ and $T\in \hat{Z}_p(X)$. Assume that $T_1\wedge T_2\wedge\cdots\wedge T_m\cap T$ and $T_1'\wedge T_2\wedge\cdots\wedge T_m\cap T$ are both well-defined, then so is $(T_1+T_1')\wedge T_2\wedge\cdots\wedge T_m\cap T$.

Furthermore, if $T$ puts no mass on the polar loci of $T_1$ and $T_1'$, then
\[
(T_1+T_1')\wedge T_2\wedge\cdots\wedge T_m\cap T=T_1\wedge T_2\wedge\cdots\wedge T_m\cap T+T_1'\wedge T_2\wedge\cdots\wedge T_m\cap T.
\]
\end{proposition}
Recall that the polar locus of a closed positive $(1,1)$-current $T$ is locally defined as $\{\varphi=-\infty\}$ when the current $T$ is written as $\ddc\varphi$ locally.
\begin{proof}
This follows from \cite[Proposition~3.5(iv)]{Vu20} and \cref{lma:nomassnonpp}.
\end{proof}

\begin{proposition}\label{prop:relnppclopos}
Suppose $T_1,\ldots,T_m$ are closed positive $(1,1)$-currents on $X$ and $T\in \hat{Z}_p(X)$.
Assume that $T_1\wedge\cdots\wedge T_m\cap T$ is well-defined, then $T_1\wedge\cdots\wedge T_m\cap T\in \hat{Z}_{p-m}(X)$.
\end{proposition}
\begin{proof}
This follows immediately from \cite[Theorem~3.7, Lemma~3.2(ii)]{Vu20}.
\end{proof}

\begin{proposition}
Suppose $T_1,\ldots,T_m$ are closed positive $(1,1)$-currents on $X$ and $T\in \hat{Z}_p(X)$ and $T_1\wedge\cdots\wedge T_m\cap T$ is well-defined. Assume that $T$ puts no mass on a complete pluripolar set $A\subseteq X$, then so is $T_1\wedge\cdots\wedge T_m\cap T$.
\end{proposition}
\begin{proof}
This follows from \cref{lma:nomassnonpp}.
\end{proof}

\begin{proposition}\label{prop:tower}
Suppose $T_1,\ldots,T_m$ are closed positive $(1,1)$-currents on $X$ and $T\in \hat{Z}_p(X)$. Fix an integer $b$ with $1\leq b\leq m$. Assume that $R=T_{b+1}\wedge \cdots\wedge T_m\cap T$ is well-defined and $T_1\wedge \cdots\wedge T_b\cap R$ is well-defined. Then $T_1\wedge\cdots\wedge T_m\cap T$ is well-defined and
\[
T_1\wedge\cdots\wedge T_m\cap T=T_1\wedge \cdots\wedge T_b\cap R=T_1\wedge \cdots\wedge T_b\cap (T_{b+1}\wedge \cdots\wedge T_m\cap T).
\]
\end{proposition}
\begin{proof}
This follows from \cite[Proposition~3.5(vi)]{Vu20}.
\end{proof}

\begin{proposition}
Suppose $T_1,\ldots,T_m$ are closed positive $(1,1)$-currents on $X$ and $T\in \hat{Z}_p(X)$.
Assume that $X$ is a compact K\"ahler manifold, then $T_1\wedge\cdots\wedge T_m\cap T$ is well-defined.
\end{proposition}
\begin{proof}
This follows from \cite[Lemma~3.4]{Vu20}.
\end{proof}

\begin{proposition}\label{prop:projTiterms}
Let $f:Y\rightarrow X$ be a proper morphism of complex manifolds. 
Suppose $T_1,\ldots,T_m$ are closed positive $(1,1)$-currents on $X$ and $T\in \hat{Z}_p(Y)$. 
Assume that the $f^*T_i$'s are defined in the sense that $f$ does not map any of the connected components of $Y$ into the polar loci of any $T_i$. Suppose that $f^*T_1\wedge \cdots \wedge f^*T_m\cap T$ is well-defined, then $T_1\wedge \cdots \wedge T_m\cap f_*T$ is well-defined and
\begin{equation}\label{eq:projtiterms}
f_* (f^*T_1\wedge \cdots \wedge f^*T_m\cap T)=T_1\wedge \cdots \wedge T_m\cap f_*T.
\end{equation}
\end{proposition}
Here the pull-back $f^*T$ of a closed positive $(1,1)$-current on $X$ is defined as follows: locally write $T=\ddc \varphi$, then we set $f^*T=\ddc f^*\varphi$.
This is a closed positive $(1,1)$-current unless $f$ maps some connected component of $Y$ completely into the polar locus of $T$. 
\begin{proof}
We may assume that $T$ is a closed positive current.

The problem is local on $X$, we may assume that $X\cong \Delta^n$ and $T_i=\ddc u_i$ for some psh functions $u_i$ on $\Delta^n$. 

Observe that for any $k\in \mathbb{Z}$,
\[
\begin{split}
\ddc \max\{u_1,-k\}\wedge \cdots \wedge \ddc \{u_m,-k\}\wedge f_*T=\\
f_*\left(\ddc \max\{f^*u_1,-k\}\wedge \cdots \wedge \ddc \max\{f^*u_m,-k\}\wedge T \right).
\end{split}
\]
In particular,
\[
\begin{split}
\mathds{1}_{\{u_1>-k,\dots,u_m>-k\}}\ddc \max\{u_1,-k\}\wedge \cdots \wedge \ddc \{u_m,-k\}\wedge f_*T=\\
f_*\left(\mathds{1}_{\{f^*u_1>-k,\dots,f^*u_m>-k\}}\ddc \max\{f^*u_1,-k\}\wedge \cdots \wedge \ddc \max\{f^*u_m,-k\}\wedge T \right).
\end{split}
\]
Let $k\to\infty$, \eqref{eq:projtiterms} follows from the continuity of $f_*$.
\end{proof}

\begin{proposition}\label{prop:relnppbdd}
Suppose $T_1,\ldots,T_m$ are closed positive $(1,1)$-currents on $X$ with locally bounded potentials and $T\in \hat{Z}_p(X)$. Then $T_1\wedge \cdots\wedge T_m\cap T$ is well-defined and
\[
T_1\wedge \cdots\wedge T_m\cap T=T_1\wedge \cdots\wedge T_m\wedge T.
\]
\end{proposition}
Here on the right-hand side, we make use of Bedford--Taylor theory.
\begin{proof}
This follows from \cite[Proposition3.6]{Vu20}.
\end{proof}
\begin{proposition}\label{prop:flatpullnpprel}
Let $Y$ be a complex manifold of pure dimension $m$. Let $f:Y\rightarrow X$ be a flat morphism of pure relative dimension $m-n$. Suppose $T_1,\ldots,T_m$ are closed positive $(1,1)$-currents on $X$ and $T\in \hat{Z}_p(X)$. Assume that $T_1\wedge \cdots \wedge T_m\cap T$ is well-defined, then so is $f^*T_1\wedge \cdots \wedge f^*T_m\cap f^*T$ and
\begin{equation}\label{eq:relnppflatpull}
f^*T_1\wedge \cdots \wedge f^*T_m\cap f^*T=f^*\left(T_1\wedge \cdots \wedge T_m\cap T \right)
\end{equation}
\end{proposition}
\begin{proof}
\textbf{Step 1}.
We first show that \eqref{eq:relnppflatpull} holds if $T_1,\ldots,T_m$ all have locally bounded potentials. In this case, by \cref{prop:relnppbdd}, \eqref{eq:relnppflatpull} is equivalent to
\begin{equation}\label{eq:temp3}
    f^*T_1\wedge \cdots \wedge f^*T_m\wedge f^*T=f^*\left(T_1\wedge \cdots \wedge T_m\wedge T \right).
\end{equation}
By induction, we reduce immediately to the case $m=1$. The problem is local on $X$, so we are reduced to \cref{thm:DS}(5).

\textbf{Step 2}. We handle the general case.
We may assume that $T$ is a closed positive current.

The problem is local on $X$, we may assume that $X\cong \Delta^n$ and $T_i=\ddc u_i$ for some psh functions $u_i$ on $\Delta^n$. In this case, by Step~1,
\[
\ddc \max\{f^* u_1,-k\}\wedge \cdots \wedge \ddc \{f^* u_m,-k\}\wedge f^*T=f^*(\ddc \max\{u_1,-k\}\wedge \cdots \wedge \ddc \{u_m,-k\}\wedge T)
\]
for any $k$. In particular,
\[
\begin{split}
    \mathds{1}_{\{f^*u_1>-k,\dots,f^*u_m>-k\}}\ddc \max\{f^* u_1,-k\}\wedge \cdots \wedge \ddc \{f^* u_m,-k\}\wedge f^*T=\\
    f^*\left(\mathds{1}_{\{u_1>-k,\dots,u_m>-k\}}\ddc \max\{u_1,-k\}\wedge \cdots \wedge \ddc \{u_m,-k\}\wedge T\right)
\end{split}
\]

Since $f^*$ is continuous by \cref{thm:DS}, letting $k\to\infty$, we conclude \eqref{eq:relnppflatpull}.
\end{proof}

\section{Segre currents and Chern currents}\label{sec:SegChern}
In this section, let $X$ be a compact K\"ahler manifold of pure dimension $n$.
We will define Segre currents and Chern currents following the approach of \cite[Chapter~3]{Ful}.

\subsection{First Chern class}\label{subsec:firstChern}
Recall that we have introduced the notation $\widehat{\Pic}(X)$ in \cref{subsec:singmonvb}: an object $\hat{L}=(L,h_L)\in \widehat{\Pic}(X)$ consists of a holomorphic line bundle $L$ on $X$ together with a \emph{positive} (possibly singular) Hermitian metric.

\begin{definition}
For any $T\in \widehat{Z}_a(X)$ and any $\hat{L}=(L,h_L)\in \widehat{\Pic}(X)$,
we define 
\begin{equation}
c_1(\hat{L})\cap T= \ddc h_{L} \cap T\in \widehat{Z}_{a-1}(X).
\end{equation}
\end{definition}
Note that $c_1(\hat{L})\cap T\in \widehat{Z}_{a-1}(X)$ as a consequence of \cref{prop:relnppclopos}.

We remind the readers that $c_1(\hat{L})$ is defined as an operator on the space of currents instead of a class in certain Chow group. This is a key difference with the usual intersection theory and is the main innovation of this paper. 
This point of view will be helpful when we consider the intersection theory on the Riemann--Zariski spaces as well.

We can translate the results in \cref{sec:relnpp} in terms of $c_1$. Firstly we have the commutativity of $c_1$.
\begin{proposition}\label{prop:c1comm}
Let $T\in \hat{Z}_a(X)$, $\hat{L}_i\in \widehat{\Pic}(X)$ ($i=1,2$). Then
\[
c_1(\hat{L}_1)\cap \left(c_1(\hat{L}_2)\cap T\right)=c_1(\hat{L}_2)\cap \left(c_1(\hat{L}_1)\cap T\right).
\]
\end{proposition}
\begin{proof}
This follows from \cref{prop:tower}.
\end{proof}

The functorality results in  \cref{sec:relnpp} can be put in the familiar form now.
\begin{proposition}\label{prop:firstChernformpush}Assume that $Y$ is a compact K\"ahler manifold of pure dimension $m$. Let $\hat{L}\in \widehat{\Pic}(X)$.
\begin{enumerate}
    \item 
    Let $f:Y\rightarrow X$ be a proper morphism, $T\in \widehat{Z}_a(Y)$, then if the metric on  $f^*\hat{L}$ is not identically $\infty$ on each connected component of $Y$,
    \[
    f_*(c_1(f^*\hat{L})\cap T)=c_1(\hat{L})\cap f_*T.
    \]
    \item 
    Let $f:Y\rightarrow X$ be a flat morphism of pure relative dimension $m-n$. Let $\hat{L}\in \widehat{\Pic}(X)$, $T\in \widehat{Z}_a(X)$, then
    \[
    f^*(c_1(\hat{L})\cap T)=c_1(f^*\hat{L})\cap f^*T.
    \]
\end{enumerate}
\end{proposition}
\begin{proof}
(1) follows from \cref{prop:projTiterms} and (2) follows from \cref{prop:flatpullnpprel}.
\end{proof}

\begin{proposition}
Assume that $\hat{L}=(L,h_L)\in \widehat{\Pic}(X)$ and $T\in \widehat{Z}_a(X)$. If $h_L$ is bounded, then
\[
c_1(\hat{L})\cap T=\ddc h_L\wedge T.
\]
\end{proposition}
On the right-hand side, we are using the classical Bedford--Taylor product.
\begin{proof}
This follows from \cref{prop:relnppbdd}.
\end{proof}

In general, $c_1(\hat{L})\cap T$ is not linear in $\hat{L}$, as can be easily seen from \cref{ex:c1notadd} in dimension $2$. We need a technical assumption:
\begin{definition}
We say $\hat{E}\in \widehat{\Vect}(X)$  or $\Fins(X)$ is \emph{transversal} to $T\in \widehat{Z}_a(X)$ (or $T$ is \emph{transversal} to $E$) if $p^*T$ (with $p\colon \mathbb{P}E^{\vee}\rightarrow X$ being the natural map) does not have mass on the polar locus of $\hO(1)$.
\end{definition}
In the special case $\hat{L}\in \widehat{\Pic}(X)$, $\hat{L}$ is transversal to $T$ if and only if $T$ does not have mass on the polar locus of $\hat{L}$. The following simple observation will be quite useful:
\begin{lemma}\label{lma:tranmasspr}
Let $T\in \widehat{Z}_a(X)$ and $A$ be a measurable subset of $X$. Let $p:Y\rightarrow X$ be a smooth morphism of pure relative dimension $b$ from a compact K\"ahler manifold $Y$ of pure dimension $n+b$.
Then $p^*T$ does not have mass on $p^{-1}A$ if and only if $T$ does not have mass on $A$. 

In particular, if $\hat{E}\in \widehat{\Vect}(X)$ is transversal to $T$ if and only if $T$ does not have mass on the polar locus of $\hat{E}$.
\end{lemma}
\begin{proof}
This follows from \cref{cor:adjunction}.
\end{proof}

\begin{lemma}\label{lma:flatpullbacktrans}
Let $\hat{E}\in \widehat{\Vect}(X)$  or $\Fins(X)$, $f:Y\rightarrow X$ be a smooth morphism of pure relative dimension $d$ from a compact K\"ahler manifold $Y$ of pure dimension $n+d$. Assume that $T\in\widehat{Z}_a(X)$ is transversal to $\hat{E}$, then $f^*T$ is transversal to $f^*\hat{E}$.
\end{lemma}
\begin{proof}
Consider the Cartesian square
\[
\begin{tikzcd}
\mathbb{P}(f^*E^{\vee}) \arrow[r,"f'"] \arrow[d,"p'"] \arrow[rd, "\square", phantom] & \mathbb{P}E^{\vee} \arrow[d,"p"] \\
Y \arrow[r,"f"]                                    & X          
\end{tikzcd}
\]
We need to show that $p'^*f^*T=f'^*p^*T$ does not have mass on the polar locus of $\hO_{\mathbb{P}(f^*\hat{E})^{\vee}}(1)$, which is equal to the inverse image of the polar locus of $\hO_{\mathbb{P}E^{\vee}}(1)$. By \cref{lma:tranmasspr}, it suffices to show that $p^*T$ does not have mass on the polar locus of $\hO_{\mathbb{P}E^{\vee}}(1)$, which holds by our assumption.
\end{proof}

\begin{proposition}\label{lma:c1linear}
Let $T\in \hat{Z}_a(X)$, $\hat{L}_i\in \widehat{\Pic}(X)$. Assume that $T$ is transversal to both $\hat{L}_1$ and $\hat{L}_2$. 
Then
\[
c_1(\hat{L}_1\otimes \hat{L}_2)\cap T= c_1(\hat{L}_1)\cap T+c_1(\hat{L}_2)\cap T.
\]
\end{proposition}
\begin{proof}
This follows from \cref{prop:relnpplinearinTi}.
\end{proof}

\begin{proposition}\label{prop:massless}
Assume that $\hat{L}\in \widehat{\Pic}(X)$ and $T\in \widehat{Z}_a(X)$.
Let $A\subseteq X$ be a complete pluripolar set such that $T$ does not have mass on $A$. Then $c_1(\hat{L})\cap T$ also does not have mass on $A$.
\end{proposition}
\begin{proof}
This follows from \cite[Proposition~3.5(iii)]{Vu20} and \cref{lma:nomassnonpp}.
\end{proof}
\begin{corollary}\label{cor:transtrans}
Let $T\in \hat{Z}_a(X)$, $\hat{L}_i\in \widehat{\Pic}(X)$ ($i=1,2$). Suppose that $T$ is transversal to both $\hat{L}_1$ and $\hat{L}_2$, then $c_1(\hat{L}_2)\cap T$ is transversal to $\hat{L}_1$.
\end{corollary}

\subsection{Segre classes}

Consider $\hat{E}=(E,h_E)\in \widehat{\Vect}(X)$ of rank $r+1$. 
Let $p=p_E\colon \mathbb{P}E^{\vee}\rightarrow X$ be the projective bundle associated with $E^{\vee}$. 
We defined $\hO(1)\in\widehat{\Pic}(\mathbb{P}E^{\vee})$ in \cref{subsec:projbundle}. 
Every result proved in this section works equally well for $\hat{E}\in \Fins(X)$. We will state the results in both settings at the same time.

\begin{definition}
Consider $\hat{E}=(E,h_E)\in \widehat{\Vect}(X)$ or $\Fins(X)$ of rank $r+1$. Define the \emph{$i$-th Serge class} $s_i(\hat{E})\cap \colon \widehat{Z}_{a}(X)\rightarrow \widehat{Z}_{a-i}(X)$ as follows:
let $T\in \widehat{Z}_a(X)$, then we set
\begin{equation}
    s_i(\hat{E})\cap T= (-1)^i p_*\left( c_1(\widehat{\mathcal{O}}(1))^{r+i}\cap p^*T \right).    
\end{equation}
Here $c_1(\widehat{\mathcal{O}}(1))^{r+i}\cap $ is short for iterated application of $c_1(\widehat{\mathcal{O}}(1))\cap \bullet$ for $(r+i)$ times.
\end{definition}

We show that the Segre classes satisfy the usual functoriality.
\begin{proposition}\label{prop:projectionform}
 Let $Y$ be a compact K\"ahler manifold of pure dimension $m$.
 \begin{enumerate}
     \item Let $f:Y\rightarrow X$ be a proper morphism, $\hat{E}\in \widehat{\Vect}(X)$ or $\Fins(X)$, $T\in \widehat{Z}_a(Y)$. Assume that the metric on $f^*\hat{E}$ is not identically $\infty$ on each connected component of $Y$.
     Then for all $i$,
\[
f_*(s_i(f^*\hat{E})\cap T)=s_i(\hat{E})\cap f_*T
\]
\item 
Let $f:Y\rightarrow X$ be a flat morphism of pure relative dimension $m-n$, $\hat{E}\in \widehat{\Vect}(X)$ or $\Fins(X)$, $T\in \widehat{Z}_a(X)$. Then for all $i$,
\[
f^*(s_i(\hat{E})\cap T)=s_i(f^*\hat{E})\cap f^*T.
\]
 \end{enumerate}
\end{proposition}

\begin{proof}
Let $\rank E=r+1$.

In both cases, consider the Cartesian square
\[
\begin{tikzcd}
\mathbb{P}(f^*E^{\vee}) \arrow[r,"f'"] \arrow[d,"p'"] \arrow[rd, "\square", phantom] & \mathbb{P}E^{\vee} \arrow[d,"p"] \\
Y \arrow[r,"f"]                                    & X          
\end{tikzcd}
\]
It is well-known that $f'^*\mathcal{O}_{\mathbb{P}E^{\vee}}(1)=\mathcal{O}_{\mathbb{P}(f^*E)^{\vee}}(1)$. A direct verification shows that it preserves the metric as well. Namely,
\begin{equation}\label{eq:hOpreserved}
  f'^*\hO_{\mathbb{P}E^{\vee}}(1)=\hO_{\mathbb{P}(f^*E^{\vee})}(1).
\end{equation}

(1)
Consider $T\in \widehat{Z}_a(Y)$, then
\[
\begin{aligned}
    (-1)^if_*(s_i(f^*\hat{E})\cap T)=& f_* p'_* \left(  c_1(\hO_{\mathbb{P}(f^*E^{\vee})}(1))^{r+i}\cap p'^*T \right)&\\
    =& p_* f'_*\left(  c_1(f'^*\hO_{\mathbb{P}E^{\vee}}(1))^{r+i}\cap p'^*T \right)& \eqref{eq:hOpreserved}\\
    =& p_* \left ( c_1(\hO_{\mathbb{P}E^{\vee}}(1))^{r+i}\cap f'_*p'^*T \right)& \cref{prop:firstChernformpush}\\
    =& p_* \left (c_1(\hO_{\mathbb{P}E^{\vee}}(1))^{r+i}\cap p^*f_*T \right)& \cref{cor:comppullbackpush}\\
    =& (-1)^is_i(\hat{E})\cap f_*T&.
\end{aligned}
\]

(2) Consider $T\in \widehat{Z}_a(X)$, then
\[
\begin{aligned}
(-1)^i s_i(f^*\hat{E})\cap f^*T=& p'_*\left(c_1(\hO_{\mathbb{P}(f^*E^{\vee})}(1))^{r+i}\cap p'^*f^*T\right)&\\
=&p'_*\left(f'^*c_1(\hO_{\mathbb{P}E^{\vee}}(1))^{r+i}\cap f'^*p^*T\right) & \eqref{eq:hOpreserved}\\
=&p'_*f'^*\left(c_1(\hO_{\mathbb{P}E^{\vee}}(1))^{r+i}\cap p^*T\right) & \cref{prop:firstChernformpush}\\
=&f^*p_*\left(c_1(\hO_{\mathbb{P}E^{\vee}}(1))^{r+i}\cap p^*T\right)& \cref{cor:comppullbackpush}\\
=&(-1)^i f^*(s_i(\hat{E})\cap T)& .
\end{aligned}
\]
\end{proof}

\begin{proposition}\label{prop:npp}
Consider $\hat{E}\in \widehat{\Vect}(X)$ or $\Fins(X)$ and $T\in\widehat{Z}_a(X)$.  Suppose that $T$ does not have mass on a complete pluripolar set $A\subseteq X$, then so is $s_i(\hat{E})\cap T$.
\end{proposition}
\begin{proof}
First observe that $p^*T$ does not have mass on $p^{-1}(A)$ by \cref{lma:tranmasspr}.
By an iterated application of \cref{prop:massless}, $c_1(\widehat{\mathcal{O}}(1))^{r+i}\cap p^*T$ does not have mass on $p^{-1}A$. By \cref{lma:tranmasspr} again, $s_i(\hat{E})\cap T$ does not have mass on $A$.
\end{proof}

\begin{corollary}\label{cor:transafterseg}
Consider $\hat{E},\hat{F}\in \widehat{\Vect}(X)$ or $\Fins(X)$ and $T\in\widehat{Z}_a(X)$. Suppose that $T$ is transversal to both $\hat{E}$ and $\hat{F}$. Then $s_i(\hat{E})\cap T$ is transversal to $\hat{F}$. 
\end{corollary}
So our intersection theory is indeed a non-pluripolar theory. The transversality condition is preserved by the Segre classes.
\begin{proof}
Let $q\colon \mathbb{P}F^{\vee}\rightarrow X$ be the natural projection. By definition, we need to show that $\hO_{\mathbb{P}F^{\vee}}(1)$ is transversal to $q^*(s_i(\hat{E})\cap T)$. By \cref{prop:projectionform}, the latter is equal to $s_i(q^*\hat{E})\cap q^*T$. By \cref{lma:flatpullbacktrans}, $q^*T$ is transversal to $q^*\hat{E}$. On the other hand, $q^*T$ is transversal to $\hO_{\mathbb{P}F^{\vee}}(1)$ by assumption. So we are reduced to the case where $F$ is a line bundle. In this case, it suffices to apply \cref{prop:npp}.
\end{proof}

\begin{proposition}\label{prop:smT}
Let $\hat{E}=(E,h_E)\in \widehat{\Vect}(X)$ or $\Fins(X)$, $T\in \widehat{Z}_a(X)$. Assume that $h_E$ is locally bounded on an open set $U\subseteq X$, then
\begin{equation}\label{eq:smlocal}
(s_m(\hat{E})\cap T)|_U=s_m(\hat{E}|_U)\wedge T|_U.
\end{equation}
Here $s_m(\hat{E}|_U)$ is defined using the same formula as $s_m(\hat{E})$:
\[
s_m(\hat{E}|_U)\wedge T|_U\coloneqq (-1)^mp_*\left( c_1(\hO(1))^{r+i}\wedge p^*T|_U \right),
\]
where $\rank E=r+1$ and $p\colon \mathbb{P}E|_U^{\vee}\rightarrow U$ is the natural projection.
\end{proposition}
\begin{proof}
This follows from \cref{prop:relnppbdd}.
\end{proof}

\begin{proposition}\label{prop:negsegvan}
Consider $\hat{E}\in \widehat{\Vect}(X)$  or $\Fins(X)$ and $T\in\widehat{Z}_a(X)$. Then
$s_i(\hat{E})\cap T=0$ if $i<0$.
\end{proposition}
\begin{proof}
Let $p\colon \mathbb{P}E^\vee \rightarrow X$ be the projection map. We need to show that
\[
p_* \left(c_1(\hO(1))^{r+i}\cap p^*T\right)=0.
\]
The problem is local on $X$, so we may assume that $E$ is the trivial vector bundle of rank $r+1$ and $X=\Delta^n$. 
Take $\eta\in \mathcal{A}_c^{a-i,a-i}(X)$, then
\[
\begin{split}
\int_{X}\eta \wedge p_* \left(c_1(\hO(1))^{r+i}\cap p^*T\right)=&\int_{\mathbb{P}E^{\vee}} p^*\eta\wedge  \left(c_1(\hO(1))^{r+i}\cap p^*T\right).
\end{split}
\]
Fix a smooth form $\Theta\in c_1(\mathcal{O}(1))$ and identify the metric on $\hO(1)$ with $\Phi\in \PSH(\mathbb{P}E^{\vee},\Theta)$. By \cite[Lemma~3.2]{Vu20},
\[
\int_{\mathbb{P}E^{\vee}} p^*\eta\wedge \left(c_1(\hO(1))^{r+i}\cap p^*T\right)=\lim_{k\to\infty}\int_{\mathbb{P}E^{\vee}} p^*\eta \wedge \mathds{1}_{\{\Phi>-k\}} \Theta_{\max\{\Phi,-k\}}^{r+i}\wedge p^*T .
\]
Assume that $i<0$.
It suffices to prove more generally that for any bounded $\Psi\in \PSH(\mathbb{P}E^{\vee},\Theta)$,
\begin{equation}\label{eq:temp4}
p^*\eta \wedge \Theta_{\Psi}^{r+i}\wedge p^*T=0.
\end{equation}
By regularization and possibly passing to an open subset of $\mathbb{P}E^{\vee}$, we may assume that $\Psi$ is smooth. Then a simple degree counting proves \eqref{eq:temp4}.
\end{proof}

\begin{proposition}\label{prop:Segrecomm}
Let $\hat{E},\hat{F}\in \widehat{\Vect}(X)$ or $\Fins(X)$, $T\in \widehat{Z}_a(X)$. For any $b,c$, we have
\begin{equation}\label{eq:Segreclasscomm}
s_c(\hat{E})\cap s_b(\hat{F})\cap T=s_b(\hat{F})\cap s_c(\hat{E})\cap  T.
\end{equation}
\end{proposition}
\begin{remark}
This result is of course what one should expect from classical intersection theory. However, we want to emphasize that it is by no means trivial, as one would imagine at a first glance. \emph{A priori}, there are not many evidences indicating that non-pluripolar intersection theory is better than the other Hermitian intersection theories. This result marks the big advantage of our theory. In fact, the corresponding results fail in the theory of Chern currents of L\"ark\"ang--Raufi--Ruppenthal--Sera \cite{LRRS18}.
\end{remark}

\begin{proof}
Write $\rank E=r+1$ and $\rank F=s+1$.

Consider the Cartesian square
\[
\begin{tikzcd}
Y \arrow[r,"p'"] \arrow[d,"q'"] \arrow[rd, "\square", phantom] & \mathbb{P}F^{\vee} \arrow[d,"q"] \\
\mathbb{P}E^{\vee} \arrow[r,"p"]                                    & X          
\end{tikzcd}.
\]
We compute the left-hand side of \eqref{eq:Segreclasscomm}:
\[
\begin{aligned}
&(-1)^{c+b}s_c(\hat{E})\cap s_b(\hat{F})\cap T &
\\ =& (-1)^c s_c(\hat{E})\cap  q_* \left( c_1(\hO_{\mathbb{P}F^{\vee}}(1))^{b+s}\cap q^*T\right)&\\
=&(-1)^c q_*\left(s_c(q^*\hat{E})\cap c_1(\hO_{\mathbb{P}F^{\vee}}(1))^{b+s}\cap q^*T \right) & \cref{prop:projectionform}\\
=& q_*p'_*\left(  c_1(q'^*\hO_{\mathbb{P}E^{\vee}}(1))^{c+r} \cap p'^*\left(c_1(\hO_{\mathbb{P}F^{\vee}}(1))^{b+s}\cap q^*T \right) \right)&\\
=& q_*p'_*\left( q'^* c_1(\hO_{\mathbb{P}E^{\vee}}(1))^{c+r} \cap p'^*c_1(\hO_{\mathbb{P}F^{\vee}}(1))^{b+s}\cap p'^*q^*T  \right)& \cref{prop:firstChernformpush}\\
=& q_*p'_*\left( p'^*c_1(\hO_{\mathbb{P}F^{\vee}}(1))^{b+s}\cap q'^* c_1(\hO_{\mathbb{P}E^{\vee}}(1))^{c+r} \cap  p'^*q^*T  \right)&\cref{prop:c1comm}
\end{aligned}.
\]
Now run the same computation again for the right-hand side of \eqref{eq:Segreclasscomm}, we get the same expression.
\end{proof}

From the proof, we get
\begin{corollary}\label{cor:segreintformula}
Let $\hat{E}_i\in \widehat{\Vect}(X)$ or $\Fins(X)$ ($i=1,\ldots,m$), $T\in \widehat{Z}_a(X)$. Write $\rank E_i=r_i+1$. 
Let $Y=\mathbb{P}E_1^{\vee}\times_X\cdots\times_X\mathbb{P}E_m^{\vee}$. Let $\pi\colon Y\rightarrow X$, $\pi_i:Y\rightarrow \mathbb{P}E_i^{\vee}$ be the natural projections.
Then for any integers $b_i$,
\[
\begin{split}
(-1)^{b_1+\cdots+b_m}s_{b_1}(\widehat{E}_1)\cap \cdots\cap s_{b_m}(\widehat{E}_m)\cap T= \\
\pi_*\left(\pi_1^*c_1(\hO_{\mathbb{P}E_1^{\vee}}(1))^{b_1+r_1}\cap \cdots \cap \pi_m^*c_1(\hO_{\mathbb{P}E_m^{\vee}}(1))^{b_m+r_m}\cap \pi^*T\right).
\end{split}
\]
\end{corollary}

\subsection{Chern polynomials}

\begin{definition}
Let $\hat{E}_1,\ldots,\hat{E}_b\in \widehat{\Vect}(X)$ or $\Fins(X)$. Consider $n_1,\ldots,n_b\in \mathbb{N}$. We define
\begin{equation}\label{eq:puresegreprod}
s_{n_1}(\hat{E}_1)\cdots s_{n_b}(\hat{E}_b)\cap \colon \widehat{Z}_a(X)\rightarrow \widehat{Z}_{a-n_1-\cdots-n_b}(X)
\end{equation}
inductively as follows: when $b=0$, this is just the identity map; When $b>0$, we let
\[
s_{n_1}(\hat{E}_1)\cdots s_{n_b}(\hat{E}_b)\cap T\coloneqq s_{n_1}(\hat{E}_1)\cap\left( s_{n_2}(\hat{E}_2)\cdots s_{n_b}(\hat{E}_b)\cap T\right).
\]
It follows from \cref{prop:Segrecomm} that \eqref{eq:puresegreprod} is invariant under permutation of the terms. We call the formal combination $s_{n_1}(\hat{E}_1)\cdots s_{n_b}(\hat{E}_b)$ a \emph{pure Segre polynomial} of degree $n_1+\cdots+n_b$.

We say the pure Segre polynomial $s_{n_1}(\hat{E}_1)\cdots s_{n_b}(\hat{E}_b)$ is \emph{transversal} to a given $T\in \widehat{Z}_a(X)$ if $T$ is transversal to each $\hat{E}_i$.

More generally, a \emph{Segre polynomial} or a \emph{Chern polynomial} of degree $n$ is a finite formal $\mathbb{R}$-linear combination of pure Segre polynomials of degree $n$.

We say a Chern polynomial $P=\sum_i a_i P_i$ of degree $n$ is \emph{transversal} to a given $T\in \widehat{Z}_a(X)$ if each pure Segre polynomial $P_i$ is transversal to $T$. In this case, we define
\[
P\cap T\coloneqq \sum_i a_i (P_i \cap T)\in \widehat{Z}_{a-n}(X).
\]
\end{definition}

When $T$ is the current of integration $[X]$ of $X$, we usually omit $\cap T$ from the notations.
\subsection{Chern classes}
Let $P_m$ be the universal polynomial so that
\[
c_m=P_m(s_0,\ldots,s_m)
\]
for the usual Chern and Segre classes. Namely,
\[
P_m(s_1,\ldots,s_m)=\sum_{t=0}^m \sum_{\alpha=(\alpha_1,\ldots,\alpha_t)\in \mathbb{N}^t,|\alpha|=m}(-1)^t s_{\alpha_1}\cdots s_{\alpha_m}.
\]
\begin{definition}
Let $\hat{E}\in \widehat{\Vect}(X)$ or $\Fins(X)$. Assume that $\hat{E}$ is transversal to $T\in \widehat{Z}_a(X)$, then we define
\[
c_m(\hat{E})\cap T=P_m(s_1(\hat{E}),\ldots,s_m(\hat{E}))\cap T.
\]
\end{definition}
In particular, 
\[
c_1(\hat{E})\cap \bullet=-s_1(\hat{E})\cap \bullet.
\]
We can express each Chern polynomial as a polynomial of the $c_i$'s, hence explaining the name \emph{Chern} polynomial.

\subsection{Small unbounded loci}
\begin{definition}\label{def:small}
Let $\hat{E}\in \widehat{\Vect}(X)$. We say $\hat{E}$ \emph{has small unbounded locus} if there is a closed complete pluripolar set $A\subseteq X$ such that $\hat{E}$ is locally bounded on $X\setminus A$.
\end{definition}

\begin{proposition}\label{prop:smTzeroext}
Assume that $\hat{E}=(E,h_E)\in \widehat{\Vect}(X)$ is transversal to $T\in \widehat{Z}_a(X)$ and $\hat{E}$ has small unbounded locus.
Take a closed complete pluripolar set $A\subseteq X$ such that $h_E$ is locally bounded outside $A$. Then $s_m(\hat{E})\cap T$ is the zero extension of $s_m(\hat{E}|_{X\setminus A})\wedge T|_{X\setminus A}$ to $X$.
\end{proposition}
\begin{proof}
This follows immediately from \cref{prop:npp} and \cref{prop:smT}.
\end{proof}

\begin{corollary}\label{cor:vanishingChern}
Assume that $\hat{E}=(E,h_E)\in \widehat{\Vect}(X)$ is transversal to $T\in \widehat{Z}_a(X)$ and $\hat{E}$ has small unbounded locus. Then $c_s(\hat{E})\cap T=0$ for $s>\rank E$.
\end{corollary}\label{cor:vanhigherChern}
Note that in this corollary, it is essential that $h_E$ is a Hermitian metric instead of a Finsler metric, as we need the vanishing of certain Bott--Chern forms as proved in \cite{Mou04}.

\begin{proof}
Take a closed complete pluripolar set $A\subseteq X$ such that $h_E$ is locally bounded outside $A$.
By \cref{prop:smTzeroext}, if suffices to verify 
\[
c_s(\hat{E}|_{X\setminus A})\wedge T|_{X\setminus A}=0.
\]
The problem is local and we can therefore localize. By taking increasing regularizations as in \cref{prop:appgri}, we may further assume that the metric on $\hat{E}$ is smooth. In this case, it is well-known that $c_s(\hat{E})$ is represented the usual Chern forms  defined using the Chern--Weil theory, see \cite{Mou04}. We conclude as the Chern form $c_s(\hat{E})$ vanishes if $s>\rank E$.
\end{proof}

With essentially the same proof, we find
\begin{corollary}\label{cor:posiChern}
Assume that $\hat{E}=(E,h_E)\in \widehat{\Vect}(X)$ has small unbounded locus. Then $c_s(\hat{E})$ is a closed positive $(s,s)$-current for all $s$.
\end{corollary}

\begin{conjecture}
\cref{cor:vanishingChern} and \cref{cor:posiChern} still hold without assuming that $\hat{E}$ has small unbounded locus.
\end{conjecture}
The difficulty of this conjecture lies in the fact that we do not have nice regularizations of Griffiths positive metrics. It is also of interest to know if $c_s(\hat{E})$ and other Schur polynomials are always positive currents. In fact, we expect some kind of monotonicity theorem extending \cite{WN19} and \cite{DDNL18mono}.

\section{Full mass metrics and \texorpdfstring{$\mathcal{I}$}{I}-good metrics}\label{sec:specmet}
In this section, we will analyze two special classes of nice metrics, corresponding to nice metrics in Shimura setting and mixed Shimura setting respectively.

Let $X$ be a compact K\"ahler manifold of pure dimension $n$.
\subsection{Full mass metrics}
\begin{definition}\label{def:fullmass}
Let $\hat{E}=(E,h_E)\in \widehat{\Vect}(X)$ or $\Fins(X)$. We say $\hat{E}$ (or $h_E$) has \emph{full mass} (resp. \emph{positive mass}) if $\hO(1)$ on $\mathbb{P}E^{\vee}$ has full mass (resp. \emph{positive mass}).

We write $\mathcal{E}(E)$ for the set of full mass Finsler metrics on $E$.

Let $\hat{E}=(E,h_E)\in \Fins(X)$.
We say $\hat{E}$ (or $h_E$) has \emph{minimal singularities} if $h_E$ has minimal singularities as a metric on $\mathcal{O}(1)$.
\end{definition} 
We will write $\widehat{\Vect}(X)_{>0}$ (resp.  $\Fins(X)_{>0}$) for the full subcategory of $\widehat{\Vect}(X)$ (resp. $\Fins(X)$) consisting of $\hat{E}$ with positive mass. 

Recall that a vector bundle $E$ is \emph{nef} (resp. \emph{big}) if $\mathcal{O}_{\mathbb{P}E^{\vee}}(1)$ is nef (resp. big).

\begin{proposition}\label{prop:fullmasscrit} 
Let $\hat{E}\in \widehat{\Vect}(X)$  or $\Fins(X)$. Assume that $E$ is nef.
Then the following are equivalent:
\begin{enumerate}
    \item $\hat{E}$ has full mass.
    \item $\int_X s_n(\hat{E})=\int_X s_n(E)$.
    \item $s_n(\hat{E})$ represents $s_n(E)$.
\end{enumerate} 
\end{proposition}
\begin{proof}
Write $\rank E=r+1$. Let $p\colon \mathbb{P}E^{\vee}\rightarrow X$ be the natural projection. Then by definition, $\hat{E}$ has full mass if and only if
\[
\int_{\mathbb{P}E^{\vee}} c_1(\hO(1))^{n+r}=\int_{\mathbb{P}E^{\vee}}\mathcal{O}(1)^{n+r}=(-1)^n\int_X s_n(E).
\]
But
\[
\int_{\mathbb{P}E^{\vee}} c_1(\hO(1))^{n+r}=(-1)^n \int_X s_n(\hat{E}).
\]
We get the equivalence between (1) and (2).

It is obvious that (3) implies (2). Conversely, suppose that (2) holds, then $c_1(\hO(1))^{n+r}$ represents $c_1(\mathcal{O}(1))^{n+r}$. By push-forward, we get (3).
\end{proof}

\begin{example}
Suppose $\hat{E}=(E,h_E)\in \widehat{\Vect}(X)$ or $\Fins(X)$ and $h_E$ is bounded, then $\hat{E}$ has full mass as $h_{\hO(1)}$ is clearly bounded.
\end{example}

\begin{example}
Suppose $\hat{E}=(E,h_E)\in \Fins(X)$ has minimal singularities, then $\hat{E}$ has full mass.
\end{example}
In relation to this example, we propose the following conjecture, as an analogue of Griffiths conjecture:
\begin{conjecture}
Let $E$ be a pseudo-effective vector bundle on $X$. Then there is always a singular Griffiths positive Hermitian metric $h_E$ on $X$ such that the induced Finsler metric has minimal singularities. In particular, there is always a Hermitian metric $h_E$ of full mass.
\end{conjecture}

\begin{example}
Suppose that $X$ is projective, $\hat{E}\in \widehat{\Vect}(X)$ and $D$ is a snc divisor in $X$. Assume that $h_E|_{X\setminus D}$ is good with respect to $D$ in the sense of Mumford \cite{Mum77}. Then $h_E$ has full mass by \cite[Theorem~1.4]{Mum77} and \cref{prop:fullmasscrit}. Observe that $E$ is nef as $\mathcal{O}(1)$ clearly is.
\end{example}

Unfortunately, in the case of mixed Shimura varieties, the natural metrics are not always good nor of full mass, as shown in \cite{BGKK16}.

\begin{theorem}
Let $\hat{E}_j=(E_j,h_{E_j})\in \widehat{\Vect}(X)$  or $\Fins(X)$ ($j=1,\ldots,m$). Assume that each $\hat{E}_j$ has full and positive mass. Assume that $h_{E_j}^k$ are a sequence of positive metrics decreasing or increasing to $h_{E_j}$ a.e.. Let $P(c_i(\hat{E}_j))$ be a Chern polynomial, then
\[
P(c_i(\hat{E}_j^k))\rightharpoonup P(c_i(\hat{E}_j))
\]
as currents, where $\hat{E}_j^k=(E_j,h_{E_j}^k)$.
\end{theorem}
\begin{proof}
It suffices to prove
\[
s_{a_1}(\hat{E}_1^k)\cdots s_{a_j}(\hat{E}_m^k)\rightharpoonup s_{a_1}(\hat{E}_1) \cdots s_{a_j}(\hat{E}_m).
\]
By \cref{cor:segreintformula}, this reduces immediately to 
\[
c_1(\hO_{\mathbb{P}E_j^{1,\vee}}(1))^{a_1+r_1}\wedge \cdots\wedge c_1(\hO_{\mathbb{P}E_j^{m,\vee}}(1))^{a_m+r_m}\rightharpoonup c_1(\hO_{\mathbb{P}E_1^{\vee}}(1))^{a_1+r_1}\wedge \cdots\wedge c_1(\hO_{\mathbb{P}E_m^{\vee}}(1))^{a_m+r_m}.
\]
After polarization, this follows from  \cref{thm:convdsmeasures}, see also \cite[Theorem~2.3]{DDNL18mono}.
\end{proof}

\begin{theorem}\label{thm:ChernrepChern}
Let $\hat{E}_j\in \widehat{\Vect}(X)$ ($j=1,\ldots,m$). Assume that each $E_j$ is nef and each $\hat{E}_j$ has full and positive mass. Let $P(c_i(\hat{E}_j))$ be a homogeneous Chern polynomial of degree $n$. Then $P(c_i(\hat{E}_j))$ represents $P(c_i(E_j))$. 

\end{theorem}

We will not give the direct proof, instead, we deduce it from a more general result \cref{cor:cptCherncurri} below.

\subsection{Multiplier ideal sheaves and \texorpdfstring{$\mathcal{I}$}{I}-good metrics}
The purpose of this section is to define and study the multiplier ideal sheaves and $\mathcal{I}$-good metrics on a vector bundle. 

Let $\hat{E}=(E,h_E)\in \widehat{\Vect}(X)$ or $\Fins(X)$.
Let $p\colon \mathbb{P}E^{\vee}\rightarrow X$ be the projection and $h$ denotes the induced metric on $\mathcal{O}_{\mathbb{P}E^{\vee}}(1)$.

\begin{definition}\label{def:mis}
As $p$ is proper, for any $k\in \mathbb{N}_{>0}$, we have an inclusion of coherent sheaves on $X$:
\[
p_*\left( \mathcal{O}_{\mathbb{P}E^{\vee}}(k)\otimes \mathcal{I}(k h) \right)\subseteq p_*\mathcal{O}_{\mathbb{P}E^{\vee}}(k)=\Sym^k \mathcal{O}_X(E).
\]
We then define the \emph{$k$-the multiplier sheaf} of $h_E$ as
\[
\mathcal{I}_k(h_E)\coloneqq p_*\left( \mathcal{O}_{\mathbb{P}E^{\vee}}(k)\otimes \mathcal{I}(k h) \right)\subseteq \Sym^k \mathcal{O}_X(E) .
\]
\end{definition}

The author believes the $\mathcal{I}_k$'s are the correct multiplier ideal sheaves for vector bundles.

Recall that the $\mathcal{I}$-goodness of a Hermitian line bundle is defined in \cref{def:modelandgood}.
\begin{definition}\label{def:Igoodvect}
We say $\hat{E}\in \widehat{\Vect}(X)$  or $\Fins(X)$ is \emph{$\mathcal{I}$-good} if $\hO(1)$ on $\mathbb{P}E^{\vee}$ is $\mathcal{I}$-good.
\end{definition}
We write $\widehat{\Vect}_{\mathcal{I}}(X)_{>0}$ (resp. $\Fins_{\mathcal{I}}(X)_{>0}$) for the set of $\mathcal{I}$-good elements in $\widehat{\Vect}(X)_{>0}$ (resp. $\Fins(X)_{>0}$).

\begin{example}
Assume that $\hat{E}\in \widehat{\Vect}(X)$  or $\Fins(X)$ has full and positive mass, then $\hat{E}$ is $\mathcal{I}$-good. To see this, we may assume that $E$ is a line bundle, so we rename it as $L$. We denote the given metric on $L$ as $h_L$. In this case, the main theorem of \cite{DX22, DX21} shows that
\[
\frac{1}{n!}\int_X (\ddc h_L)^n\leq \vol(L,h)\leq \vol (L).
\]
But as $h_L$ has full mass, the two ends of the inequality are equal, so the first inequality is in fact an equality. Hence $\hat{L}$ is $\mathcal{I}$-good by \cref{thm:DXmain}. 
\end{example}

\begin{example}\label{ex:torvect}
All toroidal singularities on vector bundles bundles are $\mathcal{I}$-good. In the line bundle case, this is proved in \cite{BBGHdJ21}. The general case follows from the line bundle case.
\end{example}

\begin{example}
When $\hat{E}\in \widehat{\Vect}(X)$ or $\Fins(X)$ has analytic singularities in the sense of \cref{def:anasing}, $\hat{E}$ is $\mathcal{I}$-good.
\end{example}

Given these examples, $\mathcal{I}$-good singularities seem to be general enough for practice.

Our \cref{thm:DXmain} admits a straightforward extension in this setting.
\begin{theorem}\label{thm:Igoodvect}
Let $\hat{E}=(E,h_E)\in \widehat{\Vect}(X)$ or $\Fins(X)$. Assume that $\hat{E}$ has positive mass.
Write $\rank E=r+1$.
Then the followings are equivalent:
\begin{enumerate}
    \item $\hat{E}\in \widehat{\Vect}_{\mathcal{I}}(X)_{>0}$ or $\Fins_{\mathcal{I}}(X)_{>0}$;
    \item we have
    \[
    \lim_{k\to\infty} \frac{1}{k^{n+r}} h^0(X, \mathcal{I}_k(h_E))=\frac{(-1)^n}{(n+r)!}\int_X s_n(\hat{E}).
    \]
\end{enumerate}
\end{theorem}
\begin{proof}
This follows from \cref{thm:DXmain}.
\end{proof}

\section{Quasi-positive vector bundles}\label{sec:quapos}
We will extend the previous results to not necessarily positively curved vector bundles by perturbation.

Let $X$ be a projective manifold of pure dimension $n$.

\subsection{The case of line bundles}

\begin{definition}Let $L$ be a line bundle on $X$ and $h_L$ be a singular Hermitian metric on $L$.
We say $\hat{L}=(L,h_L)$ is \emph{quasi-positive} if
\begin{enumerate}
    \item $h_L$ is non-degenerate in the sense of \cref{def:singHerm}.
    \item There is $\hat{L}'\in \widehat{\Pic}(X)$ such that $\hat{L}\otimes \hat{L}'\in \widehat{\Pic}(X)$.
\end{enumerate}
For $T\in\widehat{Z}_a(X)$, we say $T$ is \emph{transversal} to $\hat{L}$ or $\hat{L}$ is \emph{transversal} to $T$ if it is possible to choose $\hat{L}'$ so that $T$ is transversal to $\hat{L}'$ and $\hat{L}\otimes \hat{L}'$.
\end{definition}
We want to extend the notion of $\mathcal{I}$-goodness to quasi-positive line bundles.

\begin{proposition}\label{prop:Igoodtensor}
Let $\hat{L},\hat{L}'\in \widehat{\Pic}_{\mathcal{I}}(X)$. Then $\hat{L}\otimes \hat{L}'\in \widehat{\Pic}_{\mathcal{I}}(X)$.
\end{proposition}
This is just a reformulation of \cref{prop:Igoodsum}.

On the other hand, 
\begin{theorem}\label{thm:Igoodcancel}
Let $\hat{L}\in \widehat{\Pic}(X)_{>0}$, $\hat{L}'\in \widehat{\Pic}(X)$. Assume that $\hat{L}\otimes \hat{L}'\in \widehat{\Pic}_{\mathcal{I}}(X)$, then  $\hat{L}\in \widehat{\Pic}_{\mathcal{I}}(X)$.
\end{theorem}
\begin{proof}
We take smooth Hermitian metrics $h_0$, $h_0'$ on $L$ and $L'$ and write their Chern forms as $\theta$ and $\theta'$. Then we may identify $h_L$ with $\varphi\in \PSH(X,\theta)$ and $h_L'$ with $\psi\in \PSH(X,\theta')$.

By \cref{lma:Igoodinsenspert}, we may assume that $\ddc h_L'$ and $\ddc h_L$ are K\"ahler currents.
Let $\varphi_j\in \PSH(X,\theta)_{>0}$ (resp. $\psi_j\in \PSH(X,\theta')_{>0}$) be a quasi-equisingular approximation of $\varphi$ (resp. $\psi$). 
By \cite[Corollary~4.5]{Xia21}, $\varphi_j+\psi_j\xrightarrow{d_{S,\theta+\theta'}} P_{\theta}[\varphi]_{\mathcal{I}}+P_{\theta'}[\psi]$. On the other hand, by our assumption, 
\[
\varphi+\psi\sim_{\mathcal{I}}P_{\theta}[\varphi]_{\mathcal{I}}+P_{\theta'}[\psi]\sim_{\mathcal{I}} P_{\theta+\theta'}[\varphi+\psi]_{\mathcal{I}}.
\]
It follows that
\[
\lim_{j\to\infty}\int_X (\theta_{\varphi_j}+\theta'{}_{\psi_j})^n=\int_X (\theta_{\varphi}+\theta'_{\psi})^n.
\]
We decompose the left-hand side as
\[
\sum_{i=0}^n \binom{n}{i}\lim_{j\to\infty}\int_X \theta_{\varphi_j}^i\wedge \theta'{}_{\psi_j}^{n-i}
\]
and the right-hand side as
\[
\sum_{i=0}^n \binom{n}{i}\lim_{j\to\infty}\int_X \theta_{\varphi}^i\wedge \theta'{}_{\psi}^{n-i}.
\]
From the monotonicity theorem \cite{DDNL18mono}, we find
\[
\lim_{j\to\infty}\int_X \theta_{\varphi_j}^n=\int_X \theta_{\varphi}^n.
\]
It follows that $\varphi_j\xrightarrow{d_S}\varphi$ by \cref{lma:decseqplusvolumeimds}. Hence $\varphi$ is $\mathcal{I}$-good by \cref{thm:DXmain}.
\end{proof}

\begin{definition}\label{def:Igoodnotp}
Let $\hat{L}=(L,h_L)$ be a quasi-positive singular Hermitian line bundle on $X$. 
We say $\hat{L}$ is $\mathcal{I}$-good if 
\begin{enumerate}
    \item $\hat{L}$ is non-degenerate.
    \item There is $\hat{L}'\in \widehat{\Pic}_{\mathcal{I}}(X)$ such that $\hat{L}\otimes \hat{L}'\in \widehat{\Pic}_{\mathcal{I}}(X)$.
\end{enumerate}
\end{definition}

By \cref{prop:Igoodtensor} and \cref{thm:Igoodcancel}, this notion  coincides with the usual notion when $\hat{L}\in \widehat{\Pic}(X)_{>0}$. 

It does not seem possible to define the corresponding notion of full mass Hermitian line bundles, as the full mass property of a Hermitian pseudo-effective line bundle is not preserved by tensor products.

\begin{proposition}\label{prop:pullIgood}
Let $f:Y\rightarrow X$ be a smooth morphism of pure relative dimension $m-n$ between projective manifolds of pure dimension $m$ and $n$. 
Assume that $\hat{L}$ is an $\mathcal{I}$-good line bundle on $X$, then $f^*\hat{L}$ is $\mathcal{I}$-good. 
\end{proposition}
This proposition fails if $f$ is not smooth, as can be easily seen from the case of closed immersions.
\begin{proof}
First observe that we may assume that $\hat{L}\in \widehat{\Pic}_{\mathcal{I}}(X)$.
Take an ample line bundle $H$ on $Y$. Fix a strictly psh smooth metric $h_H$ on $H$. It suffices to show that $f^*\hat{L}\otimes (H,h_H)$ is $\mathcal{I}$-good. For this purpose, let us take a smooth representative $\theta$ (resp. $\Theta$) in $c_1(L)$ (resp. $c_1(H)$). Then we identify $h_L$ with $\varphi\in \PSH(X,\theta)$ and $h_H$ with $\Phi\in \PSH(Y,\Theta)$. By \cref{lma:Igoodinsenspert} and \cref{thm:DXmain}, we may assume that $\theta_{\varphi}$ is a K\"ahler current. Let $\varphi_j$ be a quasi-equisingular approximation of $\varphi$. By \cref{lma:decseqplusvolumeimds}, it suffices to show that
\[
\lim_{j\to\infty}\int_Y (\Theta+\pi^*\theta+\ddc\pi^*\varphi_j+\ddc \Phi)^{m}=\int_Y (\Theta+\pi^*\theta+\ddc\pi^*\varphi+\ddc \Phi)^{m}.
\]
Decomposing both sides, it suffices to show that for all $a=0,\ldots,m$,
\[
\lim_{j\to\infty}\int_X f_*\Theta_{\Phi}^a\wedge \theta_{\varphi_j}^{m-a}=\int_X f_*\Theta_{\Phi}^a\wedge \theta_{\varphi}^{m-a}.
\]
This follows from \cref{thm:convdsmeasures}.
\end{proof}

\subsection{Vector bundle case}

\begin{definition}
Assume that $X$ is projective. Let $E$ be a vector bundle on $X$.
Let $h_E$ be a singular Hermitian metric on $E$ or a Finsler metric on $E$. We say $\hat{E}=(E,h_E)$ is \emph{quasi-positive} if $\hO(1)$ is quasi-positive. 

For $T\in\widehat{Z}_a(X)$, we say $T$ is \emph{transversal} to $\hat{E}$ or $\hat{E}$ is \emph{transversal} to $T$ if $p^*T$ is transversal to $\hO(1)$.
\end{definition}

\begin{definition}\label{def:igoodv}
Assume that $X$ is projective. Let $E$ be a vector bundle on $X$.
Let $h_E$ be a singular Hermitian metric on $E$ or a Finsler metric on $E$. We say $\hat{E}=(E,h_E)$ is \emph{$\mathcal{I}$-good} if $\hO(1)$ is $\mathcal{I}$-good in the sense of \cref{def:Igoodnotp}.
\end{definition}
\begin{corollary}\label{cor:Igoodvectflatpu}
Let $f:Y\rightarrow X$ be a smooth morphism of pure relative dimension $m-n$ between projective manifolds of pure dimension $m$ and $n$. 
Assume that $\hat{E}$ is an $\mathcal{I}$-good vector bundle on $X$, then $f^*\hat{E}$ is $\mathcal{I}$-good. 
\end{corollary}
\begin{proof}
This is an immediate consequence of \cref{prop:pullIgood}.
\end{proof}

We write $\widehat{\Vect}_{\mathcal{I}}(X)$ (resp. $\Fins_{\mathcal{I}}(X)$) for the full subcategory of $\widehat{\Vect}(X)$ (resp. $\Fins(X)$) consisting of $\mathcal{I}$-good objects.

\begin{theorem}\label{thm:vecttens}
Let $\hat{E}\in \widehat{\Vect}_{\mathcal{I}}(X)$ (resp. $\Fins_{\mathcal{I}}(X)$) and $\hat{L}\in \widehat{\Pic}(X)_{>0}$. 
Assume that $\hat{L}$ has analytic singularities. 
Then $\hat{E}\otimes \hat{L}\in \widehat{\Vect}_{\mathcal{I}}(X)_{>0}$ (resp. $\Fins_{\mathcal{I}}(X)_{>0}$).
\end{theorem}
\begin{proof}
Take a smooth metric $h_0$ on $L$ and write $\theta=c_1(L,h_0)$. We identify $h_L$ with $\varphi\in \PSH(X,\theta)$. 
Let $p\colon \mathbb{P}E^{\vee}\rightarrow X$ be the natural projection. Fix a smooth metric $h'$ on $\mathcal{O}(1)$, write $\Theta=c_1(\mathcal{O}(1),h')$ and identify the metric on $\hO(1)$ with $\Phi\in \PSH(\mathbb{P}E^{\vee},\Theta)$.
We want to show that $\Phi+p^*\varphi\in \PSH(\mathbb{P}E^{\vee},\Theta+p^*\theta)$ is $\mathcal{I}$-good. We will prove this more generally for $\mathcal{I}$-good $\Phi$ with positive mass. 

Fix a K\"ahler form $\Omega_0$ on $\mathbb{P}E^{\vee}$. 
Let $\Phi_j\in \PSH(\mathbb{P}E^{\vee},\Theta+\Omega_0)$ be a quasi-equisingular approximation of $\Phi$. By \cref{lma:Igoodinsenspert}, it suffices to show that $\Phi_j+p^*\varphi\xrightarrow{d_{S,\Theta+\Omega_0+p^*\theta}} \Phi+p^*\varphi$. In view of \cref{lma:decseqplusvolumeimds}, we need to show
\[
\lim_{j\to\infty}\int_{\mathbb{P}E^{\vee}} \left(\Theta+\Omega_0+\ddc \Phi_j + p^*(\theta+\ddc\varphi)\right)^{n+r}=\int_{\mathbb{P}E^{\vee}} \left(\Theta+\Omega_0+\ddc \Phi + p^*(\theta+\ddc\varphi)\right)^{n+r}.
\]
Here $\rank E=r+1$. 
This then follows from \cite[Theorem~4.2]{Xia21}.
\end{proof}

\begin{proposition}\label{prop:vectcanc}
Let $\hat{E}\in \widehat{\Vect}(X)_{>0}$ (resp. $\Fins(X)_{>0}$) and $\hat{L}\in \widehat{\Pic}(X)$. Assume that $\hat{E}\otimes \hat{L}\in \widehat{\Vect}_{\mathcal{I}}(X)$ (resp. $\Fins_{\mathcal{I}}(X)$). Then $\hat{E}\in \widehat{\Vect}_{\mathcal{I}}(X)$ (resp. $\Fins_{\mathcal{I}}(X)$).
\end{proposition}
\begin{proof}
This follows from \cref{thm:Igoodcancel}.
\end{proof}

Now we prove that good singularities are $\mathcal{I}$-good as long as the curvature is suitably bounded from below.
\begin{example}
Assume that $X$ is projective.
Assume that
$\hat{E}=(E,h_E)$ is a good Hermitian vector bundle with respect to a snc divisor $D$ in $X$. Let $\rank E=r+1$.
Let $\hat{L}=(L,h_L)\in \widehat{\Pic}(X)$ be an ample
line bundle together with a psh metric $h_L$ with log singularities along some snc $\mathbb{Q}$-divisor $D'$ with $|D'|=|D|$ such that $\ddc h_L-[D']$ is a smooth form.
Let $p\colon \mathbb{P}E^{\vee}\rightarrow X$ be the natural projection.
Assume that $\hO(1)\otimes p^*\hat{L}\in \widehat{\Pic}(\mathbb{P}E^{\vee}|_{X\setminus D})$, then $\hat{E}$ is $\mathcal{I}$-good. This follows from \cref{lma:extend}.

Note that the assumptions are trivially satisfied if $h_E$ is good and positively curved. In particular, this includes the important examples like the Siegel line bundles on Siegel modular varieties.
\end{example}

\subsection{Extension of the Segre currents}

In this section, we extend our previous theory of Segre currents to the quasi-positive setup.

\begin{definition}
Let $\hat{L}$ be a quasi-positive line bundle on $X$. Let $T\in \widehat{Z}_a(X)$ be a current transversal to $\hat{L}$.
Take $\hat{L}'\in \widehat{\Pic}(X)$ such that $\hat{L}\otimes \hat{L}'\in \widehat{\Pic}(X)$ and such that both $\hat{L}'$ and $\hat{L}\otimes \hat{L}'$ are transversal to $T$. 
Then we define $c_1(\hat{L})\cap T\in \widehat{Z}_{a-1}(X)$ as follows:
\[
c_1(\hat{L})\cap T\coloneqq c_1(\hat{L}\otimes \hat{L}')\cap T-c_1(\hat{L}')\cap T.
\]
It follows from \cref{lma:c1linear} that this definition is independent of the choice of $\hat{L}'$.
\end{definition}

Most of the formal properties in \cref{subsec:firstChern} extends to quasi-positive setting without difficulty. 

\begin{proposition}
Let $\hat{L}_1$, $\hat{L}_2$ be quasi-positive line bundles on $X$ and $T\in \widehat{Z}_a(X)$ transversal to both. Then
\begin{enumerate}
    \item $c_1(\hat{L}_i)\cap T$ is transversal to $\hat{L}_{3-i}$ ($i=1,2$) and
    \[
    c_1(\hat{L}_1)\cap \left(c_1(\hat{L}_2)\cap T\right)=c_1(\hat{L}_2)\cap \left(c_1(\hat{L}_1)\cap T\right)
    \]
    \item
    $\hat{L}_1\otimes \hat{L}_2$ is transversal to $T$ and
    \[
    c_1(\hat{L}_1\otimes \hat{L}_2)\cap T= c_1(\hat{L}_1)\cap T+c_1(\hat{L}_2)\cap T.
    \]
\end{enumerate}
\end{proposition}
\begin{proof}
(1)
We first show that $c_1(\hat{L}_1)\cap T$ is transversal to $\hat{L}_2$. Take $\hat{L}'\in\widehat{\Pic}(X)$ transversal to $T$ such that $\hat{L}_2\otimes \hat{L}'\in\widehat{\Pic}(X)$ is also transversal to $T$. Similarly, take $\hat{L}''\in \widehat{\Pic}(X)$ transversal to $T$ such that $\hat{L}_1\otimes \hat{L}''\in \widehat{\Pic}(X)$ is also transversal to $T$. It follows from \cref{cor:transtrans} that $\hat{L}'$ and $\hat{L}_2\otimes \hat{L}'$ are both transversal to $c_1(\hat{L}'')\cap T$ and $c_1(\hat{L}_1\otimes \hat{L}'')\cap T$. It follows that $c_1(\hat{L}_1)\cap T$ is transversal to $\hat{L}'$ and $\hat{L}_2\otimes \hat{L}'$. Hence, $c_1(\hat{L}_1)\cap T$ is transversal to $\hat{L}_2$.
Now (1) follows from \cref{prop:c1comm}.

(2) Take $\hat{L}'$ and $\hat{L}''$ as in (1), then
\[
(\hat{L}_1\otimes\hat{L}_2)\otimes (\hat{L}''\otimes \hat{L}')=(\hat{L}_1\otimes\hat{L}'')\otimes (\hat{L}_2\otimes\hat{L}')\in \widehat{\Pic}(X).
\]
As both parts of the right-hand side are transversal to $T$, so is the tensor product. Similarly, $T$ is transversal to $\hat{L}''\otimes \hat{L}'$, it follows that $T$ is transversal to $\hat{L}_1\otimes\hat{L}_2$. Now (2) follows from \cref{lma:c1linear}.
\end{proof}
The remaining results in this section all follow from a similar argument, we omit the detailed proofs.

\begin{proposition}
Let $Y$ be a projective manifold of pure dimension $m$. Let $\hat{L}$ be a quasi-positive line bundle on $X$.
\begin{enumerate}
    \item Let $f:Y\rightarrow X$ be a proper surjective morphism. Assume that $f_*T$ is transversal to $\hat{L}$ and the metric on  $f^*\hat{L}$ is not identically $\infty$ on each connected component of $Y$, then $T\in \widehat{Z}_a(Y)$ is transversal to $f^*\hat{L}$ and
    \[
    f_*(c_1(f^*\hat{L})\cap T)=c_1(\hat{L})\cap f_*T.
    \]
    \item Let $f:Y\rightarrow X$ be a flat morphism of pure relative dimension $m-n$ and $T\in \widehat{Z}_a(X)$ be transversal to $\hat{L}$, then $f^*\hat{L}$ is transversal to $f^*T$ and
    \[
    f^*(c_1(\hat{L})\cap T)=c_1(f^*\hat{L})\cap f^*T.
    \]
\end{enumerate}
\end{proposition}
\begin{proof}
This follows from \cref{prop:firstChernformpush}.
\end{proof}

\begin{proposition}\label{prop:transtrans2}
Let $T\in \hat{Z}_a(X)$, $\hat{L}_1$, $\hat{L}_2$ be quasi-positive line bundles on $X$. Suppose that $T$ is transversal to both $\hat{L}_1$ and $\hat{L}_2$, then $c_1(\hat{L}_2)\cap T$ is transversal to $\hat{L}_1$.
\end{proposition}
\begin{proof}
This follows from \cref{cor:transtrans}.
\end{proof}

\begin{definition}
Let $\hat{E}$ be a quasi-positive vector bundle on $X$ of rank $r+1$. Let $T\in \widehat{Z}_a(X)$ be a current transversal to $\hat{E}$. Then we define $s_j(\hat{E})\cap T\in \widehat{Z}_{a-j}(X)$ as
\[
s_j(\hat{E})\cap T=(-1)^jp_*\left(c_1(\hO(1))^{r+j}\cap p^*T\right),
\]
where $p\colon \mathbb{P}E^{\vee}\rightarrow X$ is the natural projection.
\end{definition}

\begin{proposition}
 Let $Y$ be a projective manifold of pure dimension $m$.
 \begin{enumerate}
     \item Let $f:Y\rightarrow X$ be a proper surjective morphism, $\hat{E}$ be quasi-positive vector bundle on $X$, $T\in \widehat{Z}_a(Y)$ be a current transversal to $f^*\hat{E}$. Assume that $f_*T$ is transversal to $\hat{E}$ and the metric on  $f^*\hat{E}$ is not identically $\infty$ on each connected component of $Y$,
     then for all $i$,
    \[
    f_*(s_i(f^*\hat{E})\cap T)=s_i(\hat{E})\cap f_*T
    \]
    \item 
    Let $f:Y\rightarrow X$ be a flat morphism of pure relative dimension $m-n$, $\hat{E}$ be a quasi-positive vector bundle on $X$, $T\in \widehat{Z}_a(X)$ be a current transversal to $\hat{E}$. Then $f^*E$ is transversal to $f^*T$ and  for all $i$,
    \[
    f^*(s_i(\hat{E})\cap T)=s_i(f^*\hat{E})\cap f^*T.
    \]
 \end{enumerate}
\end{proposition}
\begin{proof}
This follows from the same proof as \cref{prop:projectionform}.
\end{proof}

\begin{proposition}
Let $\hat{E},\hat{F}$ be two quasi-positive vector bundles on $X$. Let $T\in \widehat{Z}_a(X)$ be a current transversal to both $\hat{E}$ and $\hat{F}$. Then for any $i$, $s_i(\hat{E})\cap T$ is transversal to $\hat{F}$.
\end{proposition}
\begin{proof}
This follows from the argument of \cref{cor:transafterseg}.
\end{proof}

\begin{proposition}
Let $\hat{E}$ be a quasi-positive vector bundle on $X$ and $T\in\widehat{Z}_a(X)$ transversal to $T$. Then
$s_i(\hat{E})\cap T=0$ if $i<0$.
\end{proposition}
\begin{proof}
This follows from the same argument as \cref{prop:negsegvan}.
\end{proof}
\begin{proposition}
Let $\hat{E},\hat{F}$ be quasi-positive vector bundles on $X$, $T\in \widehat{Z}_a(X)$ be transversal to $\hat{E},\hat{F}$. For any $b,c$ we have
\begin{equation}
s_c(\hat{E})\cap s_b(\hat{F})\cap T=s_b(\hat{F})\cap s_c(\hat{E})\cap  T.
\end{equation}
\end{proposition}
\begin{proof}
This follows from the same argument as \cref{prop:Segrecomm}.
\end{proof}

In this setting, one can similarly define the Chern currents using the same formulae, we omit the details.

\part{b-divisor techniques}\label{part:2}

\section{The intersection theory of b-divisors}
In this section, we briefly recall the intersection theory of Dang--Favre \cite{DF20}.

In the whole section, we assume that $k$ is a field of characteristic $0$ and $X$ is a smooth projective variety over $k$ of pure dimension $n$. 

We write $\mathfrak{X}$ for the Riemann--Zariski space associated with $X$. Readers who are not familiar with Riemann--Zariski spaces could simply regard $\mathfrak{X}$ as a formal notation. We will treat the Riemann--Zariski spaces more carefully in \cref{sec:RZ}.

\begin{definition}
A \emph{birational model} of $X$ is a projective birational morphism $\pi\colon Y\rightarrow X$ from a \emph{smooth} variety $Y$. A morphism between two birational models $\pi\colon Y\rightarrow X$ and $\pi':Y'\rightarrow X$ is a morphism $Y\rightarrow Y'$ over $X$. 

We write $\Bir(X)$ for the isomorphism classes of birational models of $X$. It is a directed set under the partial ordering of domination. 
\end{definition}
We will usually be sloppy by omitting $\pi$ and say $Y$ is a birational model of $X$.

We write $\NS^1(X)$ for the Néron--Severi group of $X$ and $\NS^1(X)_K$ for $\NS^1(X)\otimes_{\mathbb{Z}}K$ for any subfield $K$ of $\mathbb{R}$. Given $\alpha,\beta\in \NS^1(X)_K$, we write $\alpha \leq \beta$ if $\beta-\alpha$ is pseudo-effective. More generally, we write $\NS^p(X)$ for the quotient of $\CH^p(X)$ by the numerical equivalence relation. See \cite[Section~19.1]{Ful}. Again, we denote the tensor product with $K$ by a sub-index $K$.

\begin{definition}
A \emph{Weil b-divisor} $\mathbb{D}$ on $X$ (or on $\mathfrak{X}$) is an assignment that associates with each $(\pi\colon Y\rightarrow X)\in \Bir(X)$ a class $\mathbb{D}_Y=\mathbb{D}_{\pi}\in \NS^1(Y)_{\mathbb{R}}$ such that when $\pi':Y'\rightarrow X$ dominates $\pi$ through $p:Y'\rightarrow Y$, we have
\[
p_*\mathbb{D}_{Y'}=\mathbb{D}_Y.
\]
The set of Weil b-divisors on $X$ is denoted by $\bWeil(\mathfrak{X})$.

A Weil b-divisor $\mathbb{D}$ on $X$ is \emph{Cartier} if there is $(\pi\colon Y\rightarrow X)\in \Bir(X)$ such that for any $(\pi':Y'\rightarrow X)\in \Bir(X)$ which dominates $\pi$ through $p:Y'\rightarrow Y$, we have 
\[
\mathbb{D}_{Y'}=p^*\mathbb{D}_Y.
\]
In this case we say $\mathbb{D}$ is \emph{determined} on $Y$ or $\mathbb{D}$ has an \emph{incarnation} $\mathbb{D}_Y$ on $Y$ and write $\mathbb{D}=\mathbb{D}(\mathbb{D}_Y)$.
We also say $\mathbb{D}$ is a Cartier b-divisor. The linear space of Cartier b-divisors is denoted by $\bCart(\mathfrak{X})$.
\end{definition}
Our definition simply means
\begin{equation}\label{eq:bdivprojlim}
\begin{aligned}
\bWeil(\mathfrak{X})=&\varprojlim_{(\pi\colon Y\rightarrow X)\in \Bir(X)} \NS^1(Y)_{\mathbb{R}},\\
\bCart(\mathfrak{X})=&\varinjlim_{(\pi\colon Y\rightarrow X)\in \Bir(X)} \NS^1(Y)_{\mathbb{R}},
\end{aligned}
\end{equation}
in the category of vector spaces. 

Similarly, for all $p\in \mathbb{N}$, we will write
\begin{equation}\label{eq:bdivprojlim2}
\begin{aligned}
\bWeil^p(\mathfrak{X})=&\varprojlim_{(\pi\colon Y\rightarrow X)\in \Bir(X)} \NS^p(Y)_{\mathbb{R}},\\
\bCart^p(\mathfrak{X})=&\varinjlim_{(\pi\colon Y\rightarrow X)\in \Bir(X)} \NS^p(Y)_{\mathbb{R}}.
\end{aligned}
\end{equation}

We endow $\bWeil^p(\mathfrak{X})$ with the projective limit topology, then the first equation in \eqref{eq:bdivprojlim2} becomes a projective limit in the category of locally convex linear spaces. Clearly, $\bCart^p(\mathfrak{X})$ is dense in $\bWeil^p(\mathfrak{X})$.


\begin{definition}\label{def:nef}
A Cartier b-divisor $\mathbb{D}$ on $X$ is \emph{nef} (resp. \emph{big}) if some incarnation is (equivalently all incarnations are) nef (resp. \emph{big}).

A Weil b-divisor $\mathbb{D}$ on $X$ is \emph{nef} if it lies in the closure of the set of nef Cartier b-divisors.

Write $\bWeil_{\nef}(\mathfrak{X})$ for the set of nef Weil b-divisors on $X$.

A Weil b-divisor $\mathbb{D}$ on $X$ is \emph{pseudo-effective} if for all $(\pi\colon Y\rightarrow X)\in \Bir(X)$, $\mathbb{D}_Y\geq 0$.

We introduce a partial ordering on $\bWeil(\mathfrak{X})$: 
\[
\mathbb{D}\leq \mathbb{D}'\text{ if and only if }\mathbb{D}_Y\leq \mathbb{D}'_Y\text{ for all }(\pi\colon Y\rightarrow X)\in \Bir(X).
\]
\end{definition}

We summarise Dang--Favre's results:
\begin{theorem}[{\cite[Theorem~2.1]{DF20}}]\label{thm:DF1}
Let $\mathbb{D}\in \bWeil(\mathfrak{X})$ be a nef Weil b-divisor. Then there is a \emph{decreasing} net $(\mathbb{D}_{i})_{i\in I}$ of nef Cartier b-divisors such that
\[
\mathbb{D}=\lim_{i\in I}\mathbb{D}_{i}.
\]
\end{theorem}

\begin{definition}\label{def:nefint}
Let $\mathbb{D}_i\in\bWeil(\mathfrak{X})$ ($i=1,\ldots,n$) be nef Cartier b-divisors on $X$. We define $(\mathbb{D}_1,\ldots,\mathbb{D}_n)\in \mathbb{R}$ as follows: take $(\pi\colon Y\rightarrow X)\in \Bir(X)$ such that all $\mathbb{D}_i's$ are determined on $Y$. Then define
\begin{equation}
    (\mathbb{D}_1,\ldots,\mathbb{D}_n)\coloneqq (\mathbb{D}_{1,Y},\ldots,\mathbb{D}_{n,Y}).
\end{equation}
The intersection number $(\mathbb{D}_1,\ldots,\mathbb{D}_n)$ does not depend on the choice of $Y$.
\end{definition}

\begin{theorem}[{\cite[Proposition~3.1,Theorem~3.2]{DF20}}]\label{thm:DF2}
There is a unique pairing 
\[
(\bWeil_{\nef}(\mathfrak{X}))^n\rightarrow \mathbb{R}_{\geq 0}
\]
extending the pairing  in \cref{def:nefint} such that
\begin{enumerate}
    \item The pairing is monotonically increasing in each variable.
    \item The pairing is continuous along decreasing nets in each variable.
\end{enumerate}
Moreover, this pairing has the following properties:
\begin{enumerate}
    \item It is symmetric, multilinear.
    \item It is usc in each variable.
\end{enumerate}
\end{theorem}
\begin{remark}
Strictly speaking, this is not the original theorem proved by Dang--Favre, but the argument in \cite{DF20} extends to our setting \emph{verbatim}. Equivalently, our intersection product is the degree of the original intersection product of Dang--Favre.
\end{remark}

\begin{remark}\label{rmk:DFexte}
We observe that this product has a unique linear extension to differences of nef b-divisors:
\[
(\bWeil_{\nef}(\mathfrak{X})-\bWeil_{\nef}(\mathfrak{X}))^n\rightarrow \mathbb{R}.
\]
\end{remark}

\begin{definition}
We define the \emph{volume} of $\mathbb{D}\in \bWeil_{\nef}(\mathfrak{X})$ by
\begin{equation}\label{eq:volbdivdef}
    \vol \mathbb{D}=(\mathbb{D},\ldots,\mathbb{D}).
\end{equation}

We say $\mathbb{D}\in \bWeil_{\nef}(\mathfrak{X})$ is \emph{big} if $\vol \mathbb{D}>0$. 
\end{definition}
Note that the definition of bigness is compatible with the definition in \cref{def:nef} in the case of Cartier b-divisors.

The following simple lemma does not seem to have appeared in the literature:
\begin{lemma}\label{lma:volbdivaslim}
Let $\mathbb{D}\in \bWeil_{\nef}(\mathfrak{X})$, then
\[
\vol \mathbb{D}=\inf_{(Y\rightarrow X)\in \Bir(X)} \vol \mathbb{D}_Y=\lim_{(Y\rightarrow X)\in \Bir(X)} \vol \mathbb{D}_Y.
\]
\end{lemma}
\begin{proof}
By \cref{thm:DF1}, we can find a decreasing net $\mathbb{D}^{\alpha}$ of nef Cartier b-divisors on $X$ converging to $\mathbb{D}$. Clearly,
\[
\vol \mathbb{D}^{\alpha}=\inf_{Y\rightarrow X} \vol \mathbb{D}^{\alpha}_Y.
\]
It follows from \cref{thm:DF2} and the continuity of the volume functional \cite[Corollary~2.6]{ELMNP05} that
\[
\vol \mathbb{D}=\inf_{\alpha}\inf_{Y\rightarrow X} \vol \mathbb{D}^{\alpha}_Y=\inf_{Y\rightarrow X} \vol \mathbb{D}_Y.
\]
On the other hand, as in general push-forward will increase the volume, we see that $\vol \mathbb{D}_Y$ is decreasing in $Y$, so we conclude.
\end{proof}

\section{b-divisor techniques on projective manifolds}\label{sec:bdiv1}

\subsection{b-divisors of psh metrics}
Let $X$ be a smooth projective variety over $\mathbb{C}$ of pure dimension $n$. Consider $\hat{L}=(L,h_L)\in \widehat{\Pic}(X)_{>0}$.
Fix a smooth Hermitian metric $h_0$ on $L$ and write $\theta=c_1(L,h_0)$. We could identify $h_L$ with $\varphi\in \PSH(X,\theta)$.

\begin{definition}
Define the \emph{singularity divisor} $\Sing_X \hat{L}$ of $\hat{L}$ as the formal sum 
\begin{equation}\label{eq:singhatL}
\Sing_X h_L=\Sing_X \hat{L}\coloneqq \sum_{E} \nu(h_L,E)E,
\end{equation}
where $E$ runs over all prime divisors contained in $X$ and $\nu(h_L,E)$ is the generic Lelong number of $h_L$ along $E$. 
The singularity divisor is \emph{not} a Weil divisor in general.
\end{definition}
Note that this is a countable sum by Siu's semicontinuity theorem. Although $\Sing_X \hat{L}$ is not a divisor in general, it does define a class in $\NS^1(X)_{\mathbb{R}}$ as follows from \cite[Proposition~1.3]{BFJ09}. We will be sloppy in the notations by writing $\Sing_X \hat{L}$ for this numerical class.

In \cite{XiaPPT}, we introduced the following b-divisor associated with $\hat{L}$:
\begin{definition}
The \emph{singularity b-divisor} $\Sing \hat{L}$ of $\hat{L}$ is the b-divisor over $X$ defined by
\[
(\Sing \hat{L})_Y\coloneqq \Sing_Y \pi^*\hat{L},
\]
where $(\pi\colon Y\rightarrow X)\in \Bir(X)$.

Define
\[
\mathbb{D}(\hat{L})\coloneqq \mathbb{D}(L)-\Sing \hat{L}.
\]
Here $\mathbb{D}(L)$ is the Cartier b-divisor determined by $L$ on $X$.

We also write $\mathbb{D}^L(\varphi)=\mathbb{D}(\theta,\varphi)$ for $\mathbb{D}(\hat{L})$.
\end{definition}
Recall the notation $\varphi$ is introduced in the beginning of this section.

By definition, we have
\begin{lemma}\label{lma:bdivadd}
Let $\hat{L}_1,\hat{L}_2\in \widehat{\Pic}_{\mathcal{I}}(X)$. Then
\[
\mathbb{D}(\hat{L}_1\otimes \hat{L}_2)=\mathbb{D}(\hat{L}_1)+\mathbb{D}(\hat{L}_2).
\]
\end{lemma}

\begin{definition}
Let $\hat{L}$ be an $\mathcal{I}$-good quasi-positive line bundle on $X$. Take $\hat{L}'\in \widehat{\Pic}_{\mathcal{I}}(X)$ such that $\hat{L}\otimes \hat{L}'\in \widehat{\Pic}_{\mathcal{I}}(X)$. Then we define
\[
\mathbb{D}(\hat{L})\coloneqq \mathbb{D}(\hat{L}\otimes \hat{L}')-\mathbb{D}(\hat{L}').
\]
By \cref{lma:bdivadd}, this definition is independent of the choice of $\hat{L}'$.
\end{definition}

\begin{example}\label{ex:fullmassnef}
Assume that $\hat{L}$ has full mass, then $\mathbb{D}(\hat{L})=\mathbb{D}(L,h_{\min})$, where $h_{\min}$ is a psh metric with minimal singularities on $L$. See \cite[Theorem~1.1]{DDNL18fullmass} for example.
In particular, if $L$ is furthermore nef, we have
\[
\mathbb{D}(\hat{L})=\mathbb{D}(L).
\]
\end{example}

We are ready to derive the first version of the Chern--Weil formula.
\begin{theorem}\label{thm:nefbvolume}
The b-divisor $\mathbb{D}(\hat{L})$ is a nef b-divisor and if in addition $\int_X c_1(\hat{L})^n>0$,
\begin{equation}\label{eq:volbandline}
\frac{1}{n!}\vol \mathbb{D}(\hat{L})=\vol \hat{L}.
\end{equation}
\end{theorem}
This is essentially \cite[Theorem~5.2]{XiaPPT}. At the time when \cite{XiaPPT} was written, neither \cite{DF20} nor \cite{DX21} was available, so our formulation was different there.
\begin{proof}

\textbf{Step 1}. We first handle the case where $\varphi$ has analytic singularities. Take a resolution $\pi\colon Y\rightarrow X$ so that $\pi^*\varphi$ has log singularities along a snc $\mathbb{Q}$-divisor $E$ on $Y$.
By \cref{prop:notionsbirational}, $\vol \pi^*\hat{L}=\vol \hat{L}$. Similarly, by definition, $\vol \mathbb{D}(\hat{L})=\vol \mathbb{D}(\pi^*\hat{L})$.
Replacing $X$ by $Y$, we may assume that $\varphi$ has log singularities along a snc $\mathbb{Q}$-divisor $E$ on $X$.
In this case, $\mathbb{D}(\hat{L})=\mathbb{D}(L-E)$, which is nef by \cite[Lemma~2.4]{XiaPPT}. 
We are reduced to show that
\begin{equation}\label{eq:temp14}
\vol \hat{L}=\frac{1}{n!}((L-E)^n).    
\end{equation}
The volume of $\hat{L}$ is computed in this case by a theorem of Bonavero \cite{Bon98}, giving \eqref{eq:temp14}.

\textbf{Step 2}. Assume that $\ddc h_L$ is a K\"ahler current. Take a quasi-equisingular approximation $\varphi^j\in \PSH(X,\theta)$ of $\varphi$.
Write $h^j$ for the corresponding metrics on $L$. By \cite[Lemma~3.7]{DX21}, $\vol (L,h^j)\to \vol (L,h)$. Observe that $\mathbb{D}(L,h^j)$ is decreasing in $j$. By Step~1 and \cref{thm:DF2}, it therefore suffices to show that $\mathbb{D}(L,h^j)\to \mathbb{D}(L,h)$. In more concrete terms, this means that for any $(\pi\colon Y\rightarrow X)\in \Bir(X)$, 
\[
\Sing (\pi^*L,\pi^*h^j)\to \Sing (\pi^*L,\pi^*h)
\]
in $\NS^1(Y)_{\mathbb{R}}$.
This obviously follows from \cite[Lemma~2.2]{XiaPPT} if $\Sing (\pi^*L,\pi^*h)$ has only finitely many components. In general, fix an ample class $\omega$ in $\NS^1(Y)$. We want to show that for any $\epsilon>0$, we can find $j_0>0$ so that when $j\geq j_0$,
\begin{equation}\label{eq:temp5}
\Sing (\pi^*L,\pi^*h^j)\geq \Sing (\pi^*L,\pi^*h)-\epsilon \omega.
\end{equation}
Write 
\[
\Sing (\pi^*L,\pi^*h)=\sum_{i=1}^{\infty}a_i E_i,\quad \Sing (\pi^*L,\pi^*h^j)=\sum_{i=1}^{\infty}a_i^j E_i.
\]
Then $a_i^j\leq a_i$.
We can find $N>0$ large enough, so that 
\[
\Sing (\pi^*L,\pi^*h)\leq \sum_{i=1}^{N}a_i E_i+\frac{\epsilon}{2}\omega.
\]
By \cite[Lemma~2.2]{XiaPPT}, we can take $j_0$ large enough so that for $j>j_0$,
\[
(a_i-a_i^j)E_i\leq \frac{\epsilon}{2N}\omega,\quad i=1,\ldots,N.
\]
Then \eqref{eq:temp5} follows.

\textbf{Step 3}. Assume that $\int_X c_1(\hat{L})^n>0$.

By \cref{lma:kcapp}, take $\psi\in \PSH(X,\theta)$ such that $\theta_{\psi}$ is a K\"ahler current and $\varphi\geq \psi$. Then $(1-j^{-1})\varphi+j^{-1}\psi$ is an increasing sequence in $\PSH(X,\theta)$ converging to $\varphi$ pointwisely and hence with respect to $d_S$ as $j\to\infty$. It follows that
\[
\lim_{j\to\infty}\vol(\theta,(1-j^{-1})\varphi+j^{-1}\psi)=\vol(\theta,\varphi).
\]
Write $h_1$ for the metric on $L$ induced by $\psi$.
It is obvious that 
\[
\vol \mathbb{D}(L,(1-j^{-1})h_L+j^{-1} h_1)\to \vol \mathbb{D}(L,h_L)
\]
as $j\to\infty$. So we conclude by Step~2.

\textbf{Step~4}. We handle the general case.

Take an ample line bundle $S$ on $X$. 
From Step~3, we know that for any rational $\epsilon>0$, $\mathbb{D}(\hat{L})+\epsilon \mathbb{D}(S)$ is a nef b-divisor. It follows immediately that $\mathbb{D}(\hat{L})$ is nef. 
\end{proof}

\begin{corollary}\label{cor:Imodcharbdiv}
Assume that $\int_X c_1(\hat{L})^n>0$, then 
$\hat{L}$ is $\mathcal{I}$-good if and only if 
\[
\vol \mathbb{D}(\hat{L})=\int_X c_1(\hat{L})^n.
\]
\end{corollary}
\begin{proof}
This follows from \cref{thm:nefbvolume} and \cref{thm:DXmain}.
\end{proof}

\begin{theorem}\label{thm:pshbdivcont}
The map $\mathbb{D}\colon \PSH(X,\theta)\rightarrow \bWeil(\mathfrak{X})$ is continuous. Here on $\PSH(X,\theta)$ we take the $d_S$-pseudometric.
\end{theorem}
\begin{proof}
Let $\varphi_i\in \PSH(X,\theta)$ be a sequence converging to $\varphi\in \PSH(X,\theta)$ with respect to $d_S$. We want to show that
\[
\mathbb{D}(\theta,\varphi_i)\rightarrow \mathbb{D}(\theta,\varphi).
\]
As $\varphi_i\xrightarrow{d_S}\varphi$ implies that $\pi^*\varphi_i\xrightarrow{d_S}\pi^*\varphi$ for any $(\pi\colon Y\rightarrow X)\in \Bir(X)$, it suffices to prove
\begin{equation}\label{eq:temp7}
\Sing_X \varphi_i\rightarrow \Sing_X \varphi\quad \text{in }\NS^1(X)_{\mathbb{R}}.
\end{equation}
Write 
\[
\Sing_X \varphi_i= \sum_{E}a_i^E E,\quad \Sing_X \varphi= \sum_{E}a^E E,
\]
where $E$ runs over all prime divisors on $X$. By \cref{thm:NAdatacontdS} below, $a_i^E\to a^E$ as $i\to\infty$. When the number of $E$'s is finite, \eqref{eq:temp7} follows trivially. Otherwise, we write the prime divisors on $X$ having positive coefficients in either $\Sing_X \varphi_i$ or $\Sing_X \varphi$ as $E_1,E_2,\dots$. 

We fix a basis $e_1,\ldots,e_N$ of the finite-dimensional vector space $\NS^1(X)_{\mathbb{R}}$, so that the pseudo-effective cone is contained in the cone $\sum_d \mathbb{R}_{\geq 0} e_d$. Write
\[
E_i=\sum_{d=1}^N f_i^d e_d,\quad i=1,2,\dots.
\]
Then we need to show that for any $d=1,\ldots,N$, 
\[
\lim_{i\to\infty}\sum_{j=1}^{\infty}a_i^{E_j} f_j^d = \sum_{j=1}^{\infty} a^{E_j} f_j^d.
\]
This follows from the dominated convergence theorem, since 
\[
\sum_{j=1}^{\infty}a_i^{E_j}E_j\leq c_1(L),\quad \sum_{j=1}^{\infty}a^{E_j}E_j\leq c_1(L).
\]

\end{proof}

\begin{lemma}\label{lma:Imodeldeclelong}
Let $X$ be a compact K\"ahler manifold and $\theta$ be a closed real positive $(1,1)$-form representing a big cohomology class.
Let $\varphi^j\in \PSH(X,\theta)$  be a decreasing sequence of model potentials. Let $\varphi$ be the limit of $\varphi^j$. Assume that $\varphi$ has positive mass. Then for any prime divisor $E$ over $X$,
\[
\lim_{j\to\infty}\nu(\varphi^j,E)=\nu(\varphi,E).
\]
\end{lemma}
The following proof is due to Darvas, which greatly simplifies the author's original proof.
\begin{proof}
Since $\varphi \coloneqq  \lim_j \varphi^j$ and the $\varphi^j$'s are model, we obtain that $\int_Y \theta_\varphi^n = \lim_j \int_Y \theta_{\varphi^j}^n>0$ \cite[Proposition 4.8]{DDNLmetric}. By \cite[Lemma 4.3]{DDNLmetric}, for any $\varepsilon \in (0,1)$, for $j$ big enough there exists $\psi^j \in \PSH(X,\theta)$ such that $(1-\varepsilon) \varphi^j + \varepsilon \psi^j \leq \varphi$. This implies that for $j$ big enough we have 
\[
(1-\varepsilon)\nu(\varphi^j,E) + \varepsilon\nu(\psi^j,E) \geq \nu(\varphi,E) \geq \nu(\varphi^j,E).
\]

However $\nu(\chi,E)$ is uniformly bounded (by some Seshadri type constant) for any $\chi \in \PSH(X,\theta)$ and $E$ fixed. So letting $\varepsilon \searrow 0$ we conclude. 
\end{proof}

\begin{theorem}\label{thm:NAdatacontdS}
Let $X$ be a compact K\"ahler manifold and $\theta$ be a closed real positive $(1,1)$-form representing a big cohomology class.
Let $\varphi_i\in \PSH(X,\theta)$ ($i\in \mathbb{N}$) be a sequence and $\varphi\in \PSH(X,\theta)$. Assume that $\varphi_i\xrightarrow{d_S}\varphi$, then for any prime divisor $E$ over $X$,
\begin{equation}\label{eq:convnu}
\lim_{i\to\infty}\nu(\varphi_i,E)=\nu(\varphi,E).
\end{equation}
\end{theorem}

\begin{proof}
By \cite[Corollary~4.3]{Xia21}, we may replace $\theta$ by $\theta+\omega$, where $\omega$ is a K\"ahler form on $X$ and reduce to the case where $\varphi_i,\varphi\in \PSH(X,\theta)_{>0}$. Then we may further assume that $\varphi_i$ and $\varphi$ are both model potentials.

When proving \eqref{eq:convnu}, we are free to pass to subsequences. By \cite[Theorem~5.6]{DDNLmetric}, we may assume that the sequence $\varphi_i$ is either increasing or decreasing. In the increasing case, \eqref{eq:convnu} is trivially true. In the decreasing case, it follows from \cref{lma:Imodeldeclelong}.

\end{proof}

A mixed version of \cref{thm:nefbvolume} is also true:
\begin{theorem}\label{thm:nefbvolume2}
Let $\hat{L}_1,\ldots,\hat{L}_n\in \widehat{\Pic}(X)_{>0}$.
Then
\begin{equation}\label{eq:bdivmixint}
\frac{1}{n!}\left(\mathbb{D}(\hat{L}_1),\ldots,\mathbb{D}(\hat{L}_n)\right)=\vol(\hat{L}_1,\ldots,\hat{L}_n)\geq \frac{1}{n!}\int_X  c_1(\hat{L}_1)\wedge \cdots\wedge c_1(\hat{L}_n).
\end{equation}
If each $\hat{L}_i$ is $\mathcal{I}$-good, then equality holds.
\end{theorem}
In other words, for $\mathcal{I}$-good potentials, Chern numbers are intersection numbers of b-divisors. One can also prove that the intersection number on the left-hand side is equal to the mixed mass in the sense of Cao \cite{Cao14}, as proved in Corollary~4.5 in the first version of \cite{Xia21}.

\begin{proof}
The inequality part of \eqref{eq:bdivmixint} is obvious. It suffices to establish the equality part.

\textbf{Step 1}.
We first handle the case of when each $\hat{L}_i$ has analytic singularities. We may clearly reduce to the case of log singularities along a snc $\mathbb{Q}$-divisor $D_i$ on $X$. In this case, the left-hand side of \eqref{eq:bdivmixint} is just $(L_1-D_1,\ldots,L_n-D_n)$. The middle term is $\int_X c_1(\hat{L}_1)\wedge \cdots \wedge c_1(\hat{L}_n)$. By polarization, this follows from \cref{thm:nefbvolume}.

\textbf{Step 2}.
Assume that the $\ddc h_{L_i}$'s are K\"ahler currents.
Let $(h^j_i)_{j}$ be a quasi-equisingular approximation of $h_{L_i}$.
By \cref{thm:DF2}, the left-hand side of \eqref{eq:bdivmixint} is continuous along these approximations:
\[
\lim_{j\to\infty} \left(\mathbb{D}(L_1,h_1^j),\ldots,\mathbb{D}(L_n,h_n^j)\right)=\left(\mathbb{D}(\hat{L}_1),\ldots,\mathbb{D}(\hat{L}_n)\right).
\]
On the other hand, by \cite[Theorem~4.2]{Xia21}, the middle part of  \eqref{eq:bdivmixint} is also continuous:
\[
\lim_{j\to\infty}\vol\left((L_1,h_1^j),\dots,(L_n,h_n^j)\right)=\vol(\hat{L}_1,\ldots,\hat{L}_n).
\]
So we reduce to Step~1.

\textbf{Step 3}.
The general case follows from the same argument as Step~3 in the proof \cref{thm:nefbvolume}. 
\end{proof}

In the full mass case, we get the following result:
\begin{corollary}\label{cor:intnefrep}
Let $\hat{L}_1,\ldots,\hat{L}_n\in \widehat{\Pic}(X)$ be Hermitian line bundles of positive and full masses. Then 
\[
\left(\mathbb{D}(\hat{L}_1),\ldots,\mathbb{D}(\hat{L}_n)\right)=\langle L_1,\ldots,L_n\rangle=\int_X  c_1(\hat{L}_1)\wedge \cdots\wedge c_1(\hat{L}_n),
\]
where the product in the middle is the movable intersection \cite{BDPP13, BFJ09}.

In particular, if furthermore the $L_i$'s are nef, then
\[
( L_1,\ldots,L_n )=\int_X  c_1(\hat{L}_1)\wedge \cdots\wedge c_1(\hat{L}_n).
\]
\end{corollary}

Now it is easy to extend to the case of quasi-positive line bundles.
\begin{corollary}\label{cor:igdff}
Let $\hat{L}_1,\ldots,\hat{L}_n$ be $\mathcal{I}$-good quasi-positive line bundles on $X$.
Then
\begin{equation}\label{eq:qpie}
\left(\mathbb{D}(\hat{L}_1),\ldots,\mathbb{D}(\hat{L}_n)\right)= \int_X  c_1(\hat{L}_1)\wedge \cdots\wedge c_1(\hat{L}_n).
\end{equation}
\end{corollary}
Here we remind the readers that the product of b-divisors is defined as in \cref{rmk:DFexte}.
\begin{proof}
We make an induction on the number $j$ of $i$ such that $\hat{L}_i\not \in \widehat{\Pic}_{\mathcal{I}}(X)$. When $j=0$, \eqref{eq:qpie} is just \eqref{eq:bdivmixint}. Assume that this theorem has been proved up to $j-1\leq n-1$, we prove it for $j$. We may assume that $\hat{L}_1,\ldots,\hat{L}_j\not\in \widehat{\Pic}_{\mathcal{I}}(X)$. Take $\hat{L}_i'\in \widehat{\Pic}_{\mathcal{I}}(X)$ such that $\hat{L}_i\otimes \hat{L}_i'\in \widehat{\Pic}_{\mathcal{I}}(X)$ for all $i\leq j$. Then by definition,
\[
\mathbb{D}(\hat{L}_i\otimes \hat{L}_i')=\mathbb{D}(\hat{L}_i)+\mathbb{D}(\hat{L}_i').
\]
So
\[
\begin{aligned}
&\left(\mathbb{D}(\hat{L}_1\otimes \hat{L}_1'),\ldots,\mathbb{D}(\hat{L}_j\otimes \hat{L}_j'),\mathbb{D}(\hat{L}_{j+1}),\ldots,\mathbb{D}(\hat{L}_{n})\right)\\
=& \sum_{I\subseteq \{1,\ldots,j\}} \left(\mathbb{D}(\hat{L}_I),\mathbb{D}(\hat{L}'_{\{1,\ldots,j\}\setminus I}),\mathbb{D}(\hat{L}_{j+1}),\ldots,\mathbb{D}(\hat{L}_{n})\right)
\end{aligned}
\]
Here $\mathbb{D}(\hat{L}_I)$ is short for $\mathbb{D}_{i_1},\ldots,\mathds{D}_{i_k}$ if $I=\{i_1,\ldots,i_k\}$. The notation $\hat{L}'_{\bullet}$ is similar. We similarly write $c_1(\hat{L})=c_1(\hat{L}_{i_1})\wedge \cdots \wedge c_1(\hat{L}_{i_k})$.
By inductive hypothesis, we can rewrite this equation as
\[
\begin{aligned}
&\int_X c_1(\hat{L}_1\otimes \hat{L}_1')\wedge \cdots \wedge c_1(\hat{L}_j\otimes \hat{L}_j')\wedge c_1(\hat{L}_{j+1})\wedge \cdots \wedge c_1(\hat{L}_n) \\
=& \left(\mathbb{D}(\hat{L}_1),\ldots,\mathbb{D}(\hat{L}_n)\right)+\sum_{I\subsetneq \{1,\ldots,j\}}
\int_X c_1(\hat{L}_I)\wedge c_1(\hat{L}'_{\{1,\ldots,j\}\setminus I})\wedge c_1(\hat{L}_{j+1})\wedge \cdots \wedge c_1(\hat{L}_n).
\end{aligned}
\]
Expand the left-hand side, we conclude \eqref{eq:qpie}.
\end{proof}

\begin{corollary}\label{cor:cptCherncurri}
Let $\hat{E}_i\in \widehat{\Vect}_{\mathcal{I}}(X)$ or $\Fins_{\mathcal{I}}(X)$ ($i=1,\ldots,m$). 
Consider a monomial in the Segre classes $s_{a_1}(\hat{E}_1)\cdots s_{a_m}(\hat{E}_m)$ with $\sum_i a_i=n$. Let $p:Y\rightarrow X$ be the projection from $Y=\mathbb{P}E_1^{\vee}\times_X \cdots\times_X \mathbb{P}E_m^{\vee}$ to $X$. Let $p_i:Y\rightarrow \mathbb{P}E_i^{\vee}$ be the natural projections.
Then we have
\begin{equation}\label{eq:DDDequal}
(-1)^{a_1+\cdots+a_m}\left(\mathbb{D}(p_1^*\hO_{\mathbb{P}E_1^{\vee}}(1))^{a_1+r_1}, \ldots , \mathbb{D}(\hO_{p_m^*\mathbb{P}E_m^{\vee}}(1))^{a_m+r_m}\right)=\int_X s_{a_1}(\hat{E}_1)\cdots s_{a_m}(\hat{E}_m).
\end{equation}
Here $\rank E_i=r_i+1$.
\end{corollary}
We emphasis that $p_i^*\hO_{\mathbb{P}E_i^{\vee}}$ does not have positive mass in general even if $\hat{E}_i$ does, so we do need the extension of $\mathcal{I}$-goodness in the quasi-positive setting to make sense of this corollary.

We will later on develop the intersection theory on the Riemann--Zariski space $\mathfrak{X}$ and reformulate this result more elegantly in \cref{cor:ref1} and \cref{cor:ref2}.
\begin{proof}
By \cref{prop:pullIgood}, $p_i^*\hO_{\mathbb{P}E_1^{\vee}}(1)$ is $\mathcal{I}$-good.
Thus \eqref{eq:DDDequal} follows from \cref{cor:igdff} and \cref{cor:segreintformula}.
\end{proof}

Now we can prove \cref{thm:ChernrepChern}.
\begin{proof}[Proof of \cref{thm:ChernrepChern}]
We still use the same notations as in the proof of \cref{cor:cptCherncurri}. In this case, $\mathbb{D}(p_i^*\hO_{\mathbb{P}E_i^{\vee}}(1))$ is nothing but $\mathbb{D}(p_i^*\mathcal{O}_{\mathbb{P}E_i^{\vee}}(1))$ by \cref{ex:fullmassnef}. Then \eqref{eq:DDDequal} reduces to
\[
(-1)^{a_1+\cdots+a_m}\int_X s_{a_1}(\hat{E}_1)\cdots s_{a_m}(\hat{E}_m)=(p_1^*\mathcal{O}_{\mathbb{P}E_1^{\vee}}(1)^{a_1+r_1},\ldots,p_m^*\mathcal{O}_{\mathbb{P}E_m^{\vee}}(1)^{a_m+r_m}).
\]
In other words, $s_{a_1}(\hat{E}_1)\cdots s_{a_m}(\hat{E}_m)$ represents $s_{a_1}(E_1)\cdots s_{a_m}(E_m)$.

Now let $P=c_{a_1}(\hat{E}_1)\cdots c_{a_m}(\hat{E}_m)$ be a Chern monomial with $\sum_i a_i=n$. By definition, $P$ is a linear combination of Segre monomials of degree $n$. It follows that  $P$ represents $c_{a_1}(E_1)\cdots c_{a_m}(E_m)$. The general case follows.
\end{proof}

\subsection{Enhanced b-divisors}
We observe that the process from $\mathcal{I}$-model potentials to b-divisors loses some information. To be more precise, given any $\mathcal{I}$-model $\varphi\in \PSH(X,\theta)$, we constructed a b-divisor $\mathbb{D}(\varphi)$, whose components on $(\pi\colon Y\rightarrow X)\in \Bir(X)$ are numerical classes. However, the (formal) divisor $\Sing \pi^*X$ is also canonical. So it deserves its own name.

\begin{definition}
A \emph{formal divisor} on a smooth projective variety $X$ is an assignment $E\mapsto a_E$, where $E$ runs over the set of prime divisors on $X$ and $a_E\in \mathbb{R}$ such that at most countably many of $a_E$ are non-zero. We write a formal divisor as
\[
\sum_E a_E E.
\]
We write $\divf(X)$ for the set of formal divisors on $X$. It has the obvious linear space structure. A formal divisor is effective if $a_E\geq 0$ for all $E$.

Given any $(\pi\colon Y\rightarrow X)\in \Bir(X)$, we define a  pushforward $\pi_*\colon \divf(Y)\rightarrow \divf(X)$ by sending $a_E E$ to $0$ if $\pi(E)$ is not a divisor and to $a_E \pi(E)$ otherwise.

Given an effective formal divisor $D$ on $X$ and a class $\alpha\in \NS^1(X)_{\mathbb{R}}$, we say $D\leq \alpha$ if for any finite collection of prime divisors $\{E_i\}$,
\[
\sum_i a_{E_i}E_i\leq \alpha
\]
as classes in $\NS^1(X)_{\mathbb{R}}$. In this case, we define
\[
[D]=\sum_E a_{E}E\in \NS^1(X)_{\mathbb{R}},
\]
where $E$ runs over all prime divisors on $X$. The sum converges as explained in \cite[Proposition~1.3]{BFJ09} or \cite[Remark~2.2]{XiaPPT}.
\end{definition}
Recall that given $\alpha,\beta\in \NS^1(X)_{\mathbb{R}}$, we say $\alpha\leq \beta$ if $\beta-\alpha$ is pseudo-effective.

\begin{definition}
An \emph{enhanced b-divisor} $\mathbb{D}$ is a formal assignment $\Bir(X)\ni (\pi\colon Y\rightarrow X)\mapsto \mathbb{D}_Y\in \divf(Y)$, compatible under pushforwards between models.

An enhanced b-divisor $\mathbb{D}$ is effective if each $\mathbb{D}_Y$ is effective.
\end{definition}

\begin{definition}
Let $\mathbb{D}$ be an effective enhanced b-divisor on $X$ and $\mathbb{D}'$ be a b-divisor on $X$. We say $\mathbb{D}\leq \mathbb{D}'$ if for each $(\pi\colon Y\rightarrow X)\in \Bir(X)$, $\mathbb{D}_Y\leq \mathbb{D}'_Y$.

If an effective enhanced b-divisor $\mathbb{D}$ on $X$ is bounded from above by some b-divisor on $X$. We can define an associated b-divisor as $([\mathbb{D}_Y])_Y$.
\end{definition}

The main example we have in mind is $(\Sing \pi^*\psi)_{\pi}$. One can recover $\psi$ modulo $\mathcal{I}$-equivalence from this associated enhanced b-divisor since we can read all Lelong numbers of $\psi$ from this enhanced b-divisor.

We also observe that enhanced b-divisors are introduced in \cite{BBGHdJ21} under the name of b-divisors. The enhanced b-divisor $(\Sing \pi^*\psi)_{\pi}$ is introduced in \cite{BBGHdJ21} as well.

\section{b-divisor techniques on quasi-projective manifolds}\label{sec:quasproj}
Let $X$ be a smooth quasi-projective variety over $\mathbb{C}$ of pure dimension $n$.

\subsection{b-divisors on quasi-projective manifolds}

\begin{definition}
A \emph{smooth compactification} is a smooth projective variety $\bar{X}$ over $\mathbb{C}$ and an open immersion (of algebraic varieties) $i\colon X\rightarrow \bar{X}$ with Zariski dense image. Two smooth compactifications $(\bar{X},i)$ and $(\bar{X}',i')$ are identified if there is an isomorphism $\bar{X}\rightarrow \bar{X}'$ sending $i$ to $i'$. We will omit $i$ from the notations.

Let $\Cpt(X)$ be the directed set of smooth compactifications of $X$ with respect to the partial order of domination.
\end{definition}
We remind the readers that we are considering the compactifications of $\bar{X}$ as algebraic varieties because two analytic compactifications of $X$ are not necessarily bimeromorphically equivalent, as shown by a famous example of Serre. See \cite[Example~1.3.1]{Fujino22}.
\begin{remark}
Strictly speaking, there is a subtle set-theoretic issue here, as $\Cpt(X)$ is not really a set. But one can solve this issue easily by fixing a Grothendieck universe for example. Otherwise, one could employ the language of filtered categories instead of directed sets.
\end{remark}

\begin{definition}\label{def:compa}
Let $\bar{X}$ be a smooth compactification of $X$. A \emph{compactification} of $\hat{E}\in \widehat{\Vect}(X)$ on $\bar{X}$ is a Hermitian pseudo-effective vector bundle $(\bar{E},h_{\bar{E}})$ on $\bar{X}$ and an isomorphism $j:(\bar{E}|_X,h_{\bar{E}}|_X)\rightarrow (E,h_E)$. Two compactifications $(\bar{E},h_{\bar{E}},j)$ and $(\bar{E}',h_{\bar{E}'},j')$ are identified if there is an isomorphism $\bar{E}\rightarrow \bar{E}'$ sending $h_{\bar{E}}$ to $h_{\bar{E}'}$ and $j$ to $j'$. We will omit $j$ from our notations.

We write $\Cpt_{\bar{X}}(\hat{E})$ for the set of compactifications of $\hat{E}$ on $\bar{X}$.
We say $\hat{E}$ is \emph{compactifiable} if $\Cpt_{\bar{X}}(\hat{E})$ is non-empty for some $\bar{X}\in \Cpt(X)$.
\end{definition}
Of course, here we are a bit sloppy by identifying two different realizations of a compactification of $X$. Note that if $\hat{E}$ is compactifiable, then $E$ is necessarily an algebraic vector bundle by the GAGA principle.

\begin{remark}
Comparing with the case of line bundles (especially the Lear extension of line bundles), it seems that one should consider some sort of \emph{$\mathbb{Q}$-vector bundles} as the correct object to compactify vector bundles. However, this notion does not seem to be well-defined. Moreover, even in the case of line bundles, by lifting to a certain power, we could avoid the use of $\mathbb{Q}$-line bundles.
\end{remark}

\begin{definition}
A \emph{(Weil) b-divisor} on $X$ is a b-divisor on some $\bar{X}\in \Cpt(X)$. We identify two b-divisors $\mathbb{D}$ on $\bar{X}\in \Cpt(X)$ and $\mathbb{D}'$ on $\bar{X}'\in \Cpt(X)$ if there is $\bar{X}''\in \Cpt(X)$ dominating both $\bar{X}$ and $\bar{X}'$ such that the restrictions of $\mathbb{D}$ and $\mathbb{D}'$ to b-divisors on $\bar{X}''$ are equal. We will call any $\mathbb{D}$ on any $\bar{X}\in \Cpt(X)$ a \emph{realization} of the b-divisor on $X$.

We say a b-divisor $\mathbb{D}$ on $X$ is \emph{nef} if one realization (hence equivalently all realizations) of $\mathbb{D}$ is nef.

When $\mathbb{D}$ is a nef b-divisor on $X$, we define $\vol \mathbb{D}$ as the volume of a realization $\mathbb{D}$. Similarly, we define the intersection product of nef b-divisors on $X$ as the intersection product of their realizations.
\end{definition}

Fix a positively-curved Hermitian line bundle $\hat{L}=(L,h_L)$ on $X$.

Both of the following examples fail with holomorphic line bundles instead of algebraic line bundles.
\begin{example}
Assume that $L$ is algebraic.
If $X$ admits a compactfication (smooth by assumption) such that $\bar{X}\setminus X$ has codimension at least $2$, then $\hat{L}$ is compactfiable by Grauert--Remmert's extension theorem \cite[Page~176, II]{GR56}.
\end{example}

\begin{example}
Assume that $L$ is algebraic.
Assume that $\bar{X}\in \Cpt(X)$. There is always a line bundle $\bar{L}$ on $\bar{X}$ extending $L$. Assume that at any point $x\in \bar{X}\setminus X$, there is a local nowhere-vanishing section $s$ of $\bar{L}$ on $U\subseteq \bar{X}$ such that $h_L(s,s)$ is bounded away from $0$ on $U\cap X$. Then $h_L$ extends to $\bar{L}$ and hence, $\Cpt_{\bar{X}}(\hat{L})$ is non-empty, see \cite[Page~175, I]{GR56}.
\end{example}

\begin{lemma}\label{lma:assbdiv}
Take $\bar{X}\in \Cpt(X)$ and $(\bar{L},h_{\bar{L}})\in \Cpt_{\bar{X}}(\hat{L})$. The nef b-divisor $\mathbb{D}(\hat{L})\coloneqq \mathbb{D}(\bar{L},h_{\bar{L}})$ on $X$ does not depend on the choices of $\bar{X}$ and $(\bar{L},h_{\bar{L}})$.
\end{lemma}
\begin{proof}
Let $\bar{X}'\in \Cpt(X)$ and $(\bar{L}',h_{\bar{L}'})\in \Cpt_{\bar{X}'}(\hat{L})$. We want to show that
\begin{equation}\label{eq:temp6}
\mathbb{D}(\bar{L},h_{\bar{L}})=\mathbb{D}(\bar{L}',h_{\bar{L}'}).
\end{equation}
First observe that we can assume that $\bar{X}=\bar{X}'$ by pulling-back to a common model. Now \eqref{eq:temp6} as an equality of b-divisors on $\bar{X}$ is a direct consequence of Lelong--Poincar\'e formula.
\end{proof}

\begin{definition}\label{def:bdivasso}
Assume that $\hat{L}$ is compactifiable. 
We call the b-divisor $\mathbb{D}(\hat{L})$ on $X$ introduced in \cref{lma:assbdiv} \emph{the b-divisor associated with} $\hat{L}$.
\end{definition}

\begin{definition}\label{def:min}
Let $\bar{X}\in \Cpt(X)$. We say an extension $\hat{\bar{L}}$ of $\hat{L}$ to $\bar{X}$ is \emph{minimal} if $\Sing \hat{\bar{L}}=0$. 
\end{definition}

\begin{lemma}
Assume that $\hat{L}$ is compactifiable and $h_L$ is locally bounded. Let $\bar{X}\in \Cpt(X)$ such that $\Cpt_{\bar{X}}(\hat{L})$ is non-empty. Then there is a unique minimal extension of $\hat{L}$ to $\bar{X}$.  
\end{lemma}
\begin{proof}
Take any compactification $\hat{\bar{L}}=(\bar{L},h_{\bar{L}})$ of $\hat{L}$ to $\bar{X}$. Set 
\[
\bar{L}'\coloneqq \bar{L}-\mathcal{O}_{\bar{X}}(\Sing h_{\bar{L}}).
\]
Define $h_{\bar{L}'}$ as the metric on $\bar{L}'$ such that
\[
(\bar{L}',h_{\bar{L}'})\otimes \hO_{\bar{X}}(\Sing h_{\bar{L}})\cong \hat{\bar{L}}.
\]
Here $\hO_{\bar{X}}(\Sing h_{\bar{L}})$ is $\mathcal{O}_{\bar{X}}(\Sing h_{\bar{L}})$ with the canonical singular metric.
Then by Poincar\'e--Lelong formula,
\[
\ddc h_{\bar{L}'}=\ddc h_{\bar{L}}-[\Sing h_{\bar{L}}].
\]
It follows that $(\bar{L}',h_{\bar{L}'})$ is a minimal extension. Similar reasoning shows that the minimal extension is unique.
\end{proof}

\subsection{Full mass extensions and \texorpdfstring{$\mathcal{I}$}{I}-good extensions}

We will write $\widehat{\Pic}(X)$ for the set of  positively-curved Hermitian holomorphic line bundles on $X$. 

\begin{definition}
We say $\hat{L}\in \widehat{\Pic}(X)$ is \emph{$\mathcal{I}$-good} if $\hat{L}$ is compactifiable and if it admits an $\mathcal{I}$-good compactification.
\end{definition}
We will write $\widehat{\Pic}_{\mathcal{I}}(X)$ for the set of $\mathcal{I}$-good elements in $\widehat{\Pic}(X)$. Similarly, we write $\widehat{\Pic}(X)_{>0}$ for the compactifiable elements in $\widehat{\Pic}(X)$ with positive masses. 

\begin{theorem}\label{thm:mixedDDD}
Assume that $\hat{L}_1,\ldots,\hat{L}_n\in \widehat{\Pic}(X)_{>0}$. Then
\[
\left(\mathbb{D}(\hat{L}_1),\ldots, \mathbb{D}(\hat{L}_n)\right)\geq \int_X c_1(\hat{L}_1)\wedge \cdots \wedge c_1(\hat{L}_n).
\]
Equality holds if each $\hat{L}_i$ is $\mathcal{I}$-good.
\end{theorem}
Of course the intersection number on the left-hand side is just the intersection number of the realizations of the b-divisors.
\begin{proof}
This follows from \cref{thm:nefbvolume2}.
\end{proof}

\begin{proposition}\label{prop:Igoodchar}
Assume that $\hat{L}$ is compactifiable.
The followings are equivalent:
\begin{enumerate}
    \item $\hat{L}$ is $\mathcal{I}$-good;
    \item we have
    \begin{equation}\label{eq:voltoc1}
            \vol \mathbb{D}(\hat{L})=
            \int_X c_1(\hat{L})^n;
          \end{equation}
    \item all extensions of $\hat{L}$ to any $\bar{X}\in \Cpt(X)$ are $\mathcal{I}$-good.
\end{enumerate}
\end{proposition}
\begin{proof}
(1) $\implies$ (2): This follows from \cref{thm:mixedDDD}.

(2) $\implies$ (3): Let $\bar{X}\in \Cpt(X)$ and $\hat{\bar{L}}\in \Cpt_{\bar{X}}(\hat{L})$. It follows from \cref{cor:Imodcharbdiv} that $\hat{\bar{L}}$ is $\mathcal{I}$-good.

(3) $\implies$ (1): This is obvious.
\end{proof}

\begin{proposition}\label{prop:Igoodqp}
Let $\hat{L}\in \widehat{\Pic}(X)$, $\hat{L}'\in \widehat{\Pic}(X)_{>0}$.
Assume that $L$ is compactifiable and $\hat{L}\otimes \hat{L}'\in \widehat{\Pic}_{\mathcal{I}}(X)$, then $\hat{L}\in \widehat{\Pic}_{\mathcal{I}}(X)$. Conversely, if $\hat{L},\hat{L}'\in \widehat{\Pic}_{\mathcal{I}}(X)$, then $\hat{L}\otimes \hat{L}'\in \widehat{\Pic}_{\mathcal{I}}(X)$.
\end{proposition}
\begin{proof}
This follows from \cref{prop:Igoodtensor} and \cref{thm:Igoodcancel}.
\end{proof}

\begin{proposition}Assume that $\hat{L}=(L,h_L)\in \widehat{\Pic}(X)$ is compactifiable and $h_L$ is locally bounded.
We have
\begin{equation}\label{eq:volcp}
\vol \mathbb{D}(\hat{L})=\lim_{\bar{X}\in \Cpt(X)} \vol (\mathbb{D}(\hat{L}))_{\bar{X}}.
\end{equation}
\end{proposition}
\begin{remark}
This result shows that it is natural to consider relative Riemann--Zariski spaces in the sense of \cite{Tem11} for compactification problems.
\end{remark}

\begin{proof}
By \cref{lma:volbdivaslim}, the left-hand side is bounded from above by the right-hand side. We prove the reverse inequality.

Fix a compactification $\bar{X}\in \Cpt(X)$ and $\hat{\bar{L}}\in \Cpt_{\bar{X}}(\hat{L})$. Then 
\[
\vol \mathbb{D}(\hat{L})=\vol \mathbb{D}(\hat{\bar{L}})
\]
by definition.
Consider any birational model $\pi\colon Y\rightarrow \bar{X}$. Then $\Sing \pi^*\hat{\bar{L}}$ is a finite combination of exceptional divisors:
\[
\Sing \pi^*\hat{\bar{L}}=\sum_{i=1}^N a_i E_i.
\]
Observe that the $E_i$'s are supported on $Y\setminus \pi^{-1}(X)$. It follows from a theorem of Zariski \cite[Lemma~2.45]{KM08} that we can find $\bar{X}'\in \Cpt(X)$ dominating $\bar{X}$ such that the $E_i$'s are exceptional for $\bar{X}'\rightarrow \bar{X}$. Then
\[
\vol \mathbb{D}(\hat{L})_{Y}=\vol \left(L-\sum_{i=1}^N  a_iE_i\right)\geq \vol \mathbb{D}(\hat{L})_{X'}.
\]
It follows from \cref{lma:volbdivaslim} that equality holds in \eqref{eq:volcp}.
\end{proof}

\begin{definition}
We say $\hat{L}\in \widehat{\Pic}(X)$ has \emph{full mass} with respect to a compactification $\bar{X}\in \Cpt(X)$ of $X$ if
\[
\vol \mathbb{D}(\hat{L})_{\bar{X}}=\int_X c_1(\hat{L})^n.
\]
\end{definition}

\begin{example}
If $\hat{L}$ has full mass with respect to a compactification $\bar{X}\in \Cpt(X)$, then $\hat{L}$ is $\mathcal{I}$-good by \cref{prop:Igoodchar}.
\end{example}
Intuitively, one should think of $\mathcal{I}$-goodness as having full mass with respect to the collection of all compactifications.

We may also extend the notion of $\mathcal{I}$-goodness to not necessarily positively curved line bundles.
\begin{definition}
Let $\hat{L}=(L,h_L)$ be a Hermitian line bundle on $X$. We say $\hat{L}$ is \emph{$\mathcal{I}$-good} if it is non-degenerate and there is $\hat{L}'\in \widehat{\Pic}_{\mathcal{I}}(X)$ such that $\hat{L}\otimes \hat{L}'\in \widehat{\Pic}_{\mathcal{I}}(X)$.
\end{definition}
This coincides with the already defined notion when $\hat{L}\in \widehat{\Pic}(X)_{>0}$, as follows from \cref{prop:Igoodqp}. 
We can then introduce the associated b-divisor in this case:
\[
\mathbb{D}(\hat{L})\coloneqq \mathbb{D}(\hat{L}\otimes \hat{L}')-\mathbb{D}(\hat{L}').
\]
This is independent of the choice of $\hat{L}'$ by \cref{lma:bdivadd}.

\begin{corollary}\label{cor:Igoodgeneralintersec}
Assume that $\hat{L}_1,\ldots,\hat{L}_n$ are $\mathcal{I}$-good Hermitian line bundles on $X$. Then
\[
\left(\mathbb{D}(\hat{L}_1),\ldots, \mathbb{D}(\hat{L}_n)\right)= \int_X c_1(\hat{L}_1)\wedge \cdots \wedge c_1(\hat{L}_n).
\]
\end{corollary}
\begin{proof}
This follows from \cref{thm:mixedDDD} and the proof of \cref{cor:igdff}.
\end{proof}

In other words, the Chern numbers of $\mathcal{I}$-good Hermitian line bundles on $X$ are equal to the Chern numbers of the corresponding Chern numbers on the Riemann--Zariski space. We regard this as a Chern--Weil formula on $X$.

\subsection{Vector bundles}

Let us consider the case of vector bundles.

We write $\widehat{\Vect}(X)_{>0}$ (resp. $\Fins(X)_{>0}$) for the full subcategory of compactifiable $\hat{E}\in \widehat{\Vect}(X)$ (resp. $\Fins(X)$) with positive mass.
\begin{definition}\label{def:Igoodquasiproj}
Assume that $\hat{E}\in \widehat{\Vect}(X)_{>0}$ or $\Fins(X)_{>0}$. We say $\hat{E}$ is \emph{$\mathcal{I}$-good} if it admits an $\mathcal{I}$-good compactification. We will write $\widehat{\Vect}_{\mathcal{I}}(X)$ or $\Fins_{\mathcal{I}}(X)$ for the set of $\mathcal{I}$-good vector bundles on $X$.
\end{definition}
Observe that for $\hat{E}=(E,h_E)\in \widehat{\Vect}(X)_{>0}$, given a compactification $\hat{E}$ of $E$ on some $\bar{X}\in \Cpt(X)$, the metric $h_E$ admits an $\mathcal{I}$-good extension to $\bar{E}$ if and only if the induced metric on $\hO_{\mathbb{P}E^{\vee}}(1)$ admits an $\mathcal{I}$-good extension to $\hO_{\mathbb{P}\bar{E}^{\vee}}(1)$, see \cite[Lemma~2.3.2]{PT18}.

\begin{proposition}
Assume that $\hat{E}\in \widehat{\Vect}(X)_{>0}$ or $\Fins(X)_{>0}$. Then the followings are equivalent:
\begin{enumerate}
    \item $\hat{E}$ is $\mathcal{I}$-good.
    \item 
    \[
    \vol \mathbb{D}(\hO_{\mathbb{P}E^{\vee}}(1))=\frac{(-1)^n}{(n+r)!}\int_X s_n(\hat{E})^n.
    \]
    \item All extensions of $\hat{E}$ to any $\bar{X}\in \Cpt(X)$ are $\mathcal{I}$-good.
\end{enumerate}
\end{proposition}
\begin{proof}
This follows from \cref{prop:Igoodchar}.
\end{proof}

We extend the Chern--Weil formula to the vector bundle case.
\begin{theorem}\label{thm:CWgeneral}
Let $\hat{E}_i\in \widehat{\Vect}_{\mathcal{I}}(X)$ or $\Fins_{\mathcal{I}}(X)$ ($i=1,\ldots,m$). 
Consider a monomial in the Segre classes $s_{a_1}(\hat{E}_1)\cdots s_{a_m}(\hat{E}_m)$ with $\sum_i a_i=n$. Let $p:Y\rightarrow X$ be the projection from $Y=\mathbb{P}E_1^{\vee}\times_X \cdots\times_X \mathbb{P}E_m^{\vee}$ to $X$. Then we have
\[
\left(\mathbb{D}(\hO_{\mathbb{P}E_1^{\vee}}(1))^{a_1+r_1}, \ldots , \mathbb{D}(\hO_{\mathbb{P}E_m^{\vee}}(1))^{a_m+r_m}\right)=(-1)^{a_1+\cdots+a_m}\int_X s_{a_1}(\hat{E}_1)\cdots s_{a_m}(\hat{E}_m).
\]
Here $\rank E_i=r_i+1$.
\end{theorem}
Here by abuse of notations, $\hO_{\mathbb{P}E_i^{\vee}}(1)$ is the pull-back of $\hO_{\mathbb{P}E_i^{\vee}}(1)$ to $Y$.
\begin{proof}
This follows from the same argument as in the projective case in \cref{cor:cptCherncurri}.
\end{proof}

Finally, let us extend the notion of $\mathcal{I}$-goodness to the general case.

\begin{definition}
Let $E$ be an algebraic vector bundle on $X$. Let $h_E$ be either a singular Hermitian metric on $E$ or a Finsler metric on $E$.
We say $\hat{E}=(E,h_E)$ is \emph{$\mathcal{I}$-good} if $\hO(1)$ is $\mathcal{I}$-good.
\end{definition}
It is straightforward to generalize \cref{thm:CWgeneral} to general quasi-positive $\mathcal{I}$-good cases. The proof remains identical.
We leave the straightforward generalization to the readers.

\section{Intersection theory on Riemann--Zariski spaces}\label{sec:RZ}

In this section, we develop the intersection theory on Riemann--Zariski spaces and reformulate our Chern--Weil formula.

Let $X$ be an irreducible quasi-projective variety over $\mathbb{C}$ of dimension $n$. 
\subsection{Riemann--Zariski spaces}

\begin{definition}
The \emph{Riemann--Zariski space} of $X$ is the filtered limit:
\begin{equation}\label{eq:defRZ}
(\mathfrak{X},\mathcal{O}_{\mathfrak{X}})\coloneqq \varprojlim_{(Y\rightarrow X)\in \Bir(X)} (Y,\mathcal{O}_Y)
\end{equation}
in the category of locally ringed spaces. Here $\mathcal{O}_Y$ denotes the algebraic structure sheaf.
We will write $Y\rightarrow X$ instead of $(Y\rightarrow X)\in \Bir(X)$ if there is no risk of confusion.

When there is no risk of confusion, we will refer to $\mathfrak{X}$ as the Riemann--Zariski of $X$. The notation $\RZ(X)$ is also used if we want to emphasize the role of $X$.
\end{definition}
\begin{remark}
There is no obvious reason for using the algebraic structure sheaf $\mathcal{O}_Y$ rather than the analytic version. Also in view of applications, one should really consider Riemann--Zariski spaces of Deligne--Mumford stacks, in which case, the natural topology is the étale topology. 

It seems that one should follow Temkin--Tyomkin's idea by viewing the Riemann--Zariski space as a strictly Henselian ringed topos
and the limit in \eqref{eq:defRZ} in the category of  strictly Henselian ringed topoi. See \cite{TT18}.
\end{remark}

\begin{theorem}
The sheaf $\mathcal{O}_{\mathfrak{X}}$ is coherent.
\end{theorem}
In other words, $\mathfrak{X}$ is an Oka space. For a proof, we refer to \cite[Proof of Proposition~6.4]{KST18}. As a consequence, an $\mathcal{O}_{\mathfrak{X}}$-module is coherent if and only if it is finitely presented. We denote this category by $\CohCat(\mathfrak{X})$.
By \cite[Theorem~0.4.2.1]{FK18}, we have
\begin{equation}\label{eq:cohRZ}
\CohCat(\mathfrak{X})=\varinjlim_{Y\rightarrow X}\CohCat(Y).    
\end{equation}
Observe that this is only a filtered colimit of categories, not of exact categories, as the pull-back is not exact in general in the category of coherent sheaves. In particular, it is a non-trivial result that the K-theory on $\mathfrak{X}$ coincides with its G-theory, as we recall in the next section.

\subsection{K-theory}

 We will only consider $K_0$, although most of the sequel can be extended to higher $K$-groups as well. We follow the usual conventions by defining $K(X)$ as the Grothendieck group of the exact category of locally free sheaves on $X$ and $G(X)$ as the Grothendieck group of the exact category of coherent sheaves on $X$.
 
 Similarly, we define $K(\mathfrak{X})$ as the Grothendieck group of the exact category of locally free sheaves on $\mathfrak{X}$ and $G(\mathfrak{X})$ as the Grothendieck group of the exact category of coherent sheaves on $\mathfrak{X}$. Note that $K(\mathfrak{X})$ is a commutative ring while $G(\mathfrak{X})$ is an Abelian group by its definition.

However, $K(\mathfrak{X})\cong G(\mathfrak{X})$ is an isomorphism by the proof of \cite[Proof of Proposition~6.4]{KST18}. See also \cite{Dah21}.
We briefly recall the argument for the convenience of the reader. By resolution theorem, it suffices to show that each coherent sheaf $\mathcal{F}$ on $\mathfrak{X}$ admits a finite resolution by locally free sheaves. We may represent $\mathcal{F}$ as the pull-back of a coherent sheaf $\mathcal{F}'$ on some model $Y$ of $X$. By \cite[Lemma~6.5]{KST18}, we may assume that $\mathcal{F}'$ has Tor dimension $\leq 1$. We can find a finite locally free resolution of $\mathcal{F}'$.
Then by the same lemma, the pull-back of this resolution to $\mathfrak{X}$ is a locally free resolution of $\mathcal{F}$.

 Next we recall the notion of $\lambda$-rings \cite[Exposé~0]{SGA6}.
 \begin{definition}
 A $\lambda$-ring is a commutative ring $K$ together with a family of applications $\lambda^i:K\rightarrow K$ ($i\in \mathbb{N}$) satisfying
 \begin{enumerate}
     \item $\lambda^0 x=1$;
     \item $\lambda^1x=x$;
     \item $\lambda^m(x+y)=\sum_{i=0}^m \lambda^ix\cdot \lambda^{m-i}y$
 \end{enumerate}
 for all $x,y\in K$.
 \end{definition}
 It is well-known that $K(X)$ admits a $\lambda$-ring structure: $\lambda^i$ is the usual exterior power $\Lambda^i$. These $\lambda$-ring structures are clear compatible under pull-back. So they define a $\lambda$-ring structure on $K(\mathfrak{X})$. In other words,
 \begin{equation}
 K(\mathfrak{X})=\varinjlim_{Y\rightarrow X}K(Y)
 \end{equation}
 in the category of $\lambda$-rings. 
 
 There is a general procedure of constructing $\gamma$-operators of a $\lambda$-ring. For each $X$, we define $\gamma^i:K(X)\rightarrow K(X)$ as follows: $\gamma^i(x)$ is the coefficient of $t^i$ in 
\begin{equation}\label{eq:lambdatogamma}
\gamma_t(x)\coloneqq \sum_{j=0}^{\infty}\lambda^j(x)\left(\frac{t}{1-t}\right)^j.
\end{equation}
The functoriality of $\lambda^j$ implies immediately that the $\gamma^i$'s are functorial and we get maps
\[
\gamma^i:K(\mathfrak{X})\rightarrow K(\mathfrak{X}).
\]
It is easy to see that $\gamma^i$ can also be constructed from the $\lambda$-ring $(K(\mathfrak{X}),\lambda^i)$ using the same formula \eqref{eq:lambdatogamma}.

On each $K(X)$, we have a homomorphism $\rk:K(X)\rightarrow \mathbb{Z}$ sending a vector bundle to its rank. By the obvious invariance of the rank under pull-back, we have a map
\[
\rk:K(\mathfrak{X})\rightarrow \mathbb{Z}.
\]
Endowed with $\rk$, $K(\mathfrak{X})$ becomes an augmented $\lambda$-ring in the sense of \cite[Exposé~0]{SGA6}.

Again, from the general theory, we can construct a $\gamma$-filtration:
We set $F^0_{\gamma}K(\mathfrak{X})=K(\mathfrak{X})$ and $F^1_{\gamma}\coloneqq \ker (\rk:K(\mathfrak{X})\rightarrow \mathbb{Z})$. For $m\geq 2$, we define  $F^{m}_{\gamma}K(\mathfrak{X})$ as the ideal of $K(\mathfrak{X})$ generated by the products $\gamma^{k_1}(x_1)\dots \gamma^{k_a}(x_a)$ for any $a\geq 0$ with $x_i\in F^{1}_{\gamma}K(\mathfrak{X})$ and $\sum_j k_j\geq m$. The same process works for $X$ as well and we get a filtration $F^{\bullet}_{\gamma}$ on $K(X)$. It is easy to see that
\[
F^{p}_{\gamma}K(\mathfrak{X})=\varinjlim_{Y\rightarrow X} F^p_{\gamma}K(Y)
\]
for all $p\in \mathbb{N}$. In particular, 
\begin{equation}\label{eq:GrK}
\Gr^p_{\gamma}K(\mathfrak{X})=\varinjlim_{Y\rightarrow X} \Gr^p_{\gamma}K(Y).
\end{equation}
\begin{definition}
We define the $p$-th \emph{Chow group} of $\mathfrak{X}$ with $\mathbb{Q}$-coefficients as 
\[
\CH^p(\mathfrak{X})_{\mathbb{Q}}\coloneqq \Gr^p_{\gamma}K(\mathfrak{X})\otimes_{\mathbb{Z}}\mathbb{Q}.
\]
Similarly, 
\[
\CH^p(\mathfrak{X})_{\mathbb{R}}\coloneqq \Gr^p_{\gamma}K(\mathfrak{X})\otimes_{\mathbb{Z}}\mathbb{R}.
\]
We also write $\CH(\mathfrak{X})_{k}$ with $k=\mathbb{Q}$ or $\mathbb{R}$ for the direct sum of $\CH^p(\mathfrak{X})_{k}$ for all $p\in \mathbb{N}$.
\end{definition}
We can now reformulate \eqref{eq:GrK} as
\begin{equation}\label{eq:CHRZcol}
\CH(\mathfrak{X})_{\mathbb{Q}}=\varinjlim_{Y\rightarrow X}\CH(Y)_{\mathbb{Q}}
\end{equation}
in the category of graded Abelian groups,
where $Y$ runs over only regular models. Note that the transition map here is just the Gysin map.
As the $\gamma$-filtration is compatible with the ring structure on both $K(Y)$ and $K(\mathfrak{X})$, we get natural ring structures on both $\CH(\mathfrak{X})_{\mathbb{Q}}$ and  $\CH(Y)_{\mathbb{Q}}$ and the colimit in \eqref{eq:CHRZcol} can be regarded as a colimit of graded rings.

One can make sense of the Chern classes using the usual formula:
\[
c_p:K(\mathfrak{X})\rightarrow \CH^p(\mathfrak{X})_{\mathbb{Q}},\quad \alpha\mapsto \gamma^p(\alpha-\rk \alpha).
\]
The usual formula gives the Chern character homomorphism $\ch:K(\mathfrak{X})\rightarrow \CH(\mathfrak{X})_{\mathbb{Q}}$.

\begin{remark}
As the $\gamma$-filtration is compatible with derived push-forward modulo torsion, it is natural to consider the $\gamma$-filtration on $\varprojlim_{Y\rightarrow X} G(Y)_{\mathbb{Q}}$ as well. It is not clear to the author how to interpret the Chow groups obtained in this way. 

\end{remark}

When $X$ is projective, then all birational models $Y$ are also projective by our convention. Thus the degree of an element $\alpha\in \CH^n(Y)_{\mathbb{Q}}$ is defined, we will denote the degree by $\int_Y \alpha$. See \cite[Definition~1.4]{Ful} for details. When $f:Z\rightarrow Y$ is a morphism over $X$ between regular models of $X$, by projection formula, $\int_Z f^{!}\alpha=\int_Y \alpha$. Thus the degree maps are compatible with each other and induces a map
\[
\int_{\mathfrak{X}}\colon \CH^n(\mathfrak{X})_{\mathbb{R}}\rightarrow \mathbb{R}.
\]

Next we observe that there are natural maps
\[
\CH^1(\mathfrak{X})_{\mathbb{R}}=\varinjlim_{Y\rightarrow X}\Pic(Y)_{\mathbb{R}}\rightarrow \varinjlim_{Y\rightarrow X}\Pic(Y)_{\mathbb{R}}/\varinjlim_{Y\rightarrow X}\Pic^0(Y)_{\mathbb{R}}\rightarrow \varinjlim_{Y\rightarrow X} \NS^1(Y)_{\mathbb{R}}=\bCart(\mathfrak{X}).
\]
We denote the image of $\alpha$ in $\bCart(\mathfrak{X})$ as $[\alpha]$.
\begin{proposition}
The natural map $\CH^1(\mathfrak{X})_{\mathbb{R}}\rightarrow \bCart(\mathfrak{X})$ is a vector space homomorphism. Moreover, if $\alpha_1,\ldots,\alpha_n\in \CH^1(\mathfrak{X})_{\mathbb{R}}$, then
\begin{equation}\label{eq:intRZequal}
    \int_{\mathfrak{X}} \alpha_1\cdots\alpha_n=\left([\alpha_1],\ldots,[\alpha_n] \right).
\end{equation}
\end{proposition}
The right-hand side is the Dang--Favre intersection product.
\begin{proof}
The first claim is obvious by construction. Let us consider the second. We take a model $Y\rightarrow X$ so that $\alpha_1,\ldots,\alpha_n$ are determined by classes $\beta_1,\ldots,\beta_n\in \CH^1(Y)_{\mathbb{R}}$. Then \eqref{eq:intRZequal} becomes
\[
\int_Y \beta_1\cdots\beta_n=\left([\alpha_1],\ldots,[\alpha_n] \right),
\]
which is exactly the definition of the right-hand side.
\end{proof}

\subsection{Functorialities}
Assume that $X$ and $X'$ are both smooth projective varieties over $\mathbb{C}$.
We fix a flat morphism $f\colon X'\rightarrow X$ of pure relative dimension $d$. Observe that $f$ is always proper.

\subsubsection{Pull-back}

We have an obvious homomorphism $f^*:K(X)\rightarrow K(X')$ sending a vector bundle $E$ on $Y\rightarrow X$ to $f^*E$. By the obvious compatibility, $f^*$ extends to a homomorphism 
\[
f^*:K(\mathfrak{X})\rightarrow K(\mathfrak{X}').
\]
Observe that $f^*$ does not depend on the choice of $X$ and $X'$ in the following sense: if we take birational models $p':Y'\rightarrow X'$, $p:Y\rightarrow X$ and a morphism $f':Y'\rightarrow Y$ making the following diagram commute:
\[
\begin{tikzcd}
Y' \arrow[r,"f'"] \arrow[d,"p'"] & Y \arrow[d,"p"] \\
X' \arrow[r,"f"]           & X          
\end{tikzcd},
\]
then there is a natural isomorphism $f'{}^*p^*E\cong p'{}^*f^*E$.

As $f^*$ clearly preserves the $\gamma$-filtration, we get induced maps
\[
f^*\colon \CH(\mathfrak{X})_{\mathbb{Q}}\rightarrow\CH(\mathfrak{X}')_{\mathbb{Q}}
\]
of graded rings.

\subsubsection{Push-forward}

Recall that we have a natural push-forward map of K-groups $f_*:K(X')\rightarrow K(X)$ given by the composition of the following maps:
\[
K(X')\cn K(\PerfCat(X'))\xrightarrow{Rf_*} K(\PerfCat(X))\cn K(X), 
\]
where $\PerfCat(X)$ denotes the Waldhausen category of perfect complexes of $\mathcal{O}_X$-modules and $\PerfCat(X')$ is defined similarly. We refer to \cite[Exposé~I, Section~4]{SGA6} or \cite[\href{https://stacks.math.columbia.edu/tag/08CL}{Tag 08CL}]{stacks-project} for the precise definitions of perfect complexes. For the perspective of Waldhausen category, we refer to \cite[Section~II.9 and Chapter~V]{Wei13}.
The morphism $Rf_*$ preserves perfect complexes by \cite{LN07}.
We will no longer distinguish $K(\PerfCat(X))$ and $K(X)$ in the sequel, always with the canonical isomorphism understood.

Given $\alpha\in K(\mathfrak{X}')$, we want to understand its push-forward to $K(\mathfrak{X})$. Take a birational model $Y_1\rightarrow X'$ such that $\alpha$ is determined by a class $\alpha_{Y_1}\in K(Y_1)$. For any birational model $Y_2\rightarrow X'$, we then find a model $Y_3\rightarrow X'$ dominating both $Y_1$ and $Y_2$ through $p_1:Y_3\rightarrow Y_1$ and $p_2:Y_3\rightarrow Y_2$. Then we can set $\alpha_{Y_2}\coloneqq Rp_{2*}Lp_{1}^*\alpha_{Y_1}$ if $\alpha_{Y_1}$ is represented by a perfect complex. It is a simple consequence of \cite{CR15} that $\alpha_{Y_2}$ does not depend on the choices we made.

Consider a regular birational model $p:Y\rightarrow X$, we form the Cartesian square
\[
\begin{tikzcd}
Y' \arrow[r,"f'"] \arrow[d,"p'"]\arrow[rd,"\square",phantom] & Y \arrow[d,"p"] \\
X' \arrow[r,"f"]           & X          
\end{tikzcd},
\]
We can set $(f_*\alpha)_{Y}\coloneqq Rf'_*\alpha_{Y'}$ if $\alpha$ is represented by a perfect complex. 
More generally, for any birational model $Z'$ of $X'$ dominating $Y'$ through a map $q:Z'\rightarrow Y'$, we also have
\[
(f_*\alpha)_{Y}=R(f'\circ q)_*\alpha_{Z'}.
\]
We observe that $f_*\alpha$ indeed lies in $K(\mathfrak{X})$ if $\alpha$ is determined on a birational model of $X'$ which descends to a birational model of $X$: it suffices to show that if $\alpha$ is determined on $X'$, say by a class $\beta\in K(X')$, then $Rf'_* Lp'^*\beta=Lp^*Rf_*\beta$ if $\beta$ is represented by a perfect complex. This follows from the fact that $X'$ and $Y$ are Tor independent over $X$, see \cite[\href{https://stacks.math.columbia.edu/tag/08IB}{Tag 08IB}]{stacks-project}. 
In this case, we say $f_*\alpha$ is \emph{defined}.
It is easy to see that $f_*$ is independent of the choice of the models $X$ and $X'$ in the same sense as above.

It is well-known that derived push-forward is compatible with the $\gamma$-filtration modulo torsion, so we get a homomorphism $f_*$ maps the subset of $\CH^p(\mathfrak{X}')_{\mathbb{Q}}$ consisting of elements $\alpha$ determined on a birational model of $X'$ that descends to a model of $X$ to $\CH^{p-d}(\mathfrak{X})_{\mathbb{Q}}$. In this case, we say $f_*\alpha$ is  \emph{defined}.

These maps further induce push-forward $f_*$ sending the subset of $\bCart^p(\mathfrak{X}')$ consisting of elements $\alpha$ determined on birational model of $X'$ that descends to a model of $X$ to $\bCart^{p-d}(\mathfrak{X})$. In this case, we say $f_*\alpha$ is  \emph{defined}.

\begin{proposition}\label{prop:cbc}
Assume that $X'$, $X$, $Y$ are smooth projective varieties over $\mathbb{C}$. Consider a proper morphism $f\colon X'\rightarrow X$ and a proper flat morphism $p:Y\rightarrow X$  of pure relative dimension $d$ and form the Cartesian square
\[
\begin{tikzcd}
Y' \arrow[r,"f'"] \arrow[d,"p'"]\arrow[rd,"\square",phantom] & Y \arrow[d,"p"] \\
X' \arrow[r,"f"]           & X          
\end{tikzcd}.
\]
Then $f^*p_*\alpha=p'_*f'^*\alpha$ for any $\alpha\in K(\mathfrak{Y})$ such that $p_*\alpha$ is defined. In particular, $f^*p_*=p'_*f'^*$ as homomorphisms $\CH^a(\mathfrak{Y})_{\mathbb{R}}\rightarrow \CH^{a-d}(\mathfrak{X}')_{\mathbb{R}}$.
\end{proposition}
\begin{proof}
Again, this is a consequence of \cite[\href{https://stacks.math.columbia.edu/tag/08IB}{Tag 08IB}]{stacks-project}.
\end{proof}

Also we have a projection formula:
\begin{proposition}\label{prop:projperf}
Let $X'$ be a smooth projective variety over $\mathbb{C}$. 
Assume that $f\colon X'\rightarrow X$ is a proper flat morphism of pure relative dimension $d$. Then for any $\alpha\in \CH(\mathfrak{X}')_{\mathbb{R}}$ such that $f_*\alpha$ is defined and $\beta\in \CH(\mathfrak{X})_{\mathbb{R}}$, we have
\[
f_*(\alpha f^*\beta)=(f_*\alpha)\beta.
\]
\end{proposition}
\begin{proof}
This is a consequence of \cite[\href{https://stacks.math.columbia.edu/tag/0B54}{Tag 0B54}]{stacks-project}.
\end{proof}
The author does not know if one can define $f_*$ for the whole $K(\mathfrak{X'})$. 

\subsection{Relation with analytic singularities}
Assume that $X$ is projective. Given a Hermitian pseudo-effective line bundle $\hat{L}=(L,h)$ with analytic singularities on some regular birational model $Y\rightarrow X$, we have seen how to associate a class $\mathbb{D}(\hat{L})\in \bCart(\mathfrak{X})$ to $\hat{L}$. We will explain how this construction can be lifted to to $\mathbf{c}_1(\hat{L})\in \CH^1(\mathfrak{X})_{\mathbb{R}}$.

Take a regular birational model $\pi\colon Z\rightarrow Y$ so that $\pi^*h$ has log singularities along a snc $\mathbb{Q}$-divisor $D$. Then we simply set $\mathbf{c}_1(\hat{L})$ as the image of $\pi^*L-D$ in $\CH^1(\mathfrak{X})_{\mathbb{R}}$. It is clear that  $\mathbf{c}_1(\hat{L})$ does not depend on the choice of $\pi$ and $\mathbf{c}_1(\hat{L})$ lifts $\mathbb{D}(\hat{L})\in \bCart(\mathfrak{X})$.

More generally, if $\hat{E}\in \widehat{\Vect}^F(X)$ has analytic singularities and the rank of $E$ is $r+1$, we can define $\mathbf{s}_a(\hat{E})\in \CH^{r+a}(\RZ(\mathbb{P}E^{\vee}))_{\mathbb{R}}$ as
\[
\mathbf{s}_a(\hat{E})\coloneqq (-1)^a  \mathbf{c}_1(\hat{O}_{\mathbb{P}E^{\vee}}(1))^{r+a},
\]
where $p\colon \mathbb{P}E^{\vee}\rightarrow X$ is the natural projection.

We need to verify that the Segre classes are functorial.
\begin{proposition}\label{prop:segfuntrz}
Let $f\colon X'\rightarrow X$ be a flat morphism of pure relative dimension $d$ from a smooth projective variety $X'$ and $\hat{E}\in \widehat{\Vect}^F(X)$ has analytic singularities. Consider the Cartesian diagram
\[
\begin{tikzcd}
\mathbb{P}(f^*E)^{\vee} \arrow[r,"f'"] \arrow[d,"p'"]\arrow[rd,"\square",phantom] & \mathbb{P}E^{\vee} \arrow[d,"p"] \\
X' \arrow[r,"f"]           & X          
\end{tikzcd}.
\]
Then $\mathbf{s}_a(f^*\hat{E})=f'^*\mathbf{s}_a(\hat{E})$.
\end{proposition}
\begin{proof}
Denote the rank of $E$ by $r+1$.
By definition, it suffices to show that 
\[
f'^*(\mathbf{c}_1(\hat{O}_{\mathbb{P}E^{\vee}}(1))^{r+a})=\mathbf{c}_1(\hat{O}_{\mathbb{P}(f^*E)^{\vee}}(1))^{r+a}.
\]
As $f^*$ is a homomorphism of rings, it suffices to show that
\[
f'^*(\mathbf{c}_1(\hat{O}_{\mathbb{P}E^{\vee}}(1)))=\mathbf{c}_1(\hat{O}_{\mathbb{P}(f^*E)^{\vee}}(1)).
\]
This is clear by definition.
\end{proof}

\begin{lemma}\label{lma:segcommrz}
Let $\hat{E}_1, \hat{E}_2 \in  \widehat{\Vect}^F(X)$ having analytic singularities and of rank $r_1+1,r_2+1$ respectively. Then for any  $a,b\in \mathbb{N}$,
\begin{equation}\label{eq:segcomm}
    \mathbf{s}_a(\hat{E}_1) \mathbf{s}_b(\hat{E}_2)=\mathbf{s}_b(\hat{E}_2)\mathbf{s}_a(\hat{E}_1)
\end{equation}
in $\CH^{r_1+r_2+a+b}(\RZ(\mathbb{P}E_1^{\vee}\times_X \mathbb{P}E_2^{\vee}))_{\mathbb{R}}$.
\end{lemma}
Of course, we omitted the obvious pull-backs.
\begin{proof}
Let $p_i\colon \mathbb{P}E_i^{\vee}\rightarrow X$ denote the natural projections.

Let $q_i:Y=\mathbb{P}E_1^{\vee}\times_X \mathbb{P}E_2^{\vee}\rightarrow \mathbb{P}E_i^{\vee}$ be the natural projections.
We have a Cartesian diagram
\[
\begin{tikzcd}
Y \arrow[r,"q_1"] \arrow[d,"q_2"] \arrow[rd, "\square", phantom] & \mathbb{P}E_1^{\vee} \arrow[d,"p_1"] \\
\mathbb{P}E_2^{\vee} \arrow[r,"p_2"]                                    & X  \end{tikzcd}.
\]
With these notations, \eqref{eq:segcomm} can be written more precisely as
\[
\mathbf{s}_a(p_2^*\hat{E}_1) q_2^*\mathbf{s}_b(\hat{E}_2)=\mathbf{s}_b(p_1^*\hat{E}_2)q_1^*\mathbf{s}_a(\hat{E}_1),
\]
which follows easily from \cref{prop:segfuntrz}.
\end{proof}

We can now reformulate our Chern--Weil formula \cref{cor:cptCherncurri} as
\begin{corollary}\label{cor:ref1}
Assume that $X$ is projective.
Let $\hat{E}_i\in \widehat{\Vect}_{\mathcal{I}}(X)$ or $\Fins_{\mathcal{I}}(X)$ ($i=1,\ldots,m$). Assume that each $\hat{E}_i$ has analytic singularities. 
Consider a monomial in the Segre classes $s_{a_1}(\hat{E}_1)\cdots s_{a_m}(\hat{E}_m)$ with $\sum_i a_i=n$. 
Then we have
\begin{equation}
\int_{\mathfrak{X}}\mathbf{s}_{a_1}(\hat{E}_1)\cdots \mathbf{s}_{a_m}(\hat{E}_m)
=\int_X s_{a_1}(\hat{E}_1)\cdots s_{a_m}(\hat{E}_m).
\end{equation}
\end{corollary}
Of course, here we omitted the obvious pull-backs. Setting $Y=\mathds{P}E_1^{\vee}\times_X\cdots\times_X \mathbb{P}E_m^{\vee}$, we also write $\int_{\mathfrak{X}}$ instead of  $\int_{\mathfrak{Y}}$.

More generally, one can make sense of any homogeneous Chern polynomial of degree $n$ in vector bundles with Griffiths positive analytic singular metrics, say $P(c_i(\hat{E}_j))$ by linear combination of Segre polynomials, then we can write
\[
\int_{\mathfrak{X}}P(\mathbf{c}_i(\hat{E}_j))
=\int_X P(c_i(\hat{E}_j)).
\]

\subsection{General case of \texorpdfstring{$\mathcal{I}$}{I}-good singularities}
We also want to extend \cref{cor:ref1} to $\mathcal{I}$-good singularities. Here we have several difficulties. The first is that we have to work with Weil b-divisors instead of Cartier b-divisors.  The second is that the push-forward in Chow groups do not commute with products. So one should not expect a general intersection product on Weil b-divisors. 

Due to the transcendental nature of the problem, will only define $\mathbf{s}_a(\hat{E})$ as operators on the numerical classes, not on Chow groups. 

We refer to \cite{DF20} for the notion of base-point free classes in $\NS^k(X)_{\mathbb{R}}$ when $k\geq 1$. When $k=0$, a base-point free class is a class with non-negative degree.

 A class $\alpha\in\bCart^k(\mathfrak{X})\coloneqq \varinjlim_{Y\rightarrow X}\NS^k(Y)_{\mathbb{R}}$ is \emph{base point-free} if some (hence any) incarnation of $\alpha$ is base-point free. The closure of the cone of base point-free Cartier classes in $\bWeil^k(\mathfrak{X})\coloneqq \varprojlim_{Y\rightarrow X}\NS^k(Y)_{\mathbb{R}}$ is denoted by $\BPF^k(\mathfrak{X})$. We will need a slight extension $\dBPF(\mathfrak{X})$ whose elements are differences of elements in $\BPF^k(\mathfrak{X})$. In other words, $\dBPF^k(\mathfrak{X})$ is the linear span of $\BPF^k(\mathfrak{X})$. 
 
 Given a proper morphism $f\colon X'\rightarrow X$, we consider the  push-forward $f_*\colon \bWeil^k(\mathfrak{X}')\rightarrow \bWeil^k(\mathfrak{X})$ defined as follows: consider $\alpha\in \bWeil^k(\mathfrak{X}')$,
 for any birational model $Y\rightarrow X$, we have a Cartesian diagram
 \[
\begin{tikzcd}
Y' \arrow[r,"f'"] \arrow[d]\arrow[rd,"\square",phantom] & Y \arrow[d] \\
X' \arrow[r,"f"]           & X          
\end{tikzcd}.
\]
Then we set $(f_*\alpha)_Y\coloneqq f'_*\alpha_{Y'}$. It follows from \cite[Theorem~5.2]{DF20} and \cite[Lemma~3.6]{FL17} that $f_*$ sends $\BPF^k(\mathfrak{X}')$ to $\BPF^{k-d}(\mathfrak{X})$ if $d=\dim X'-\dim X$.

When $k=n$, there is a natural degree map $\int_{\mathfrak{X}}\colon \dBPF^n(\mathfrak{X})\rightarrow \mathbb{R}$ defined as the degree of any of its components. The degree is well-defined by the formula after \cite[Definition~1.4]{Ful}. 
 
The author does not know how to define flat pull-back in the current setting (although we expect that the pull-back of Cartier b-divisors can be extended to Weil b-divisors by continuity),
so the approach below becomes indirect.
 
 Given any $\mathcal{I}$-good Hermitian line bundle $\hat{L}$ on $X$, we define $\mathbf{c}_1(\hat{L})\colon \dBPF^k(\mathfrak{X})\rightarrow \dBPF^{k+1}(\mathfrak{X})$ sending $\alpha$ to $(\mathbb{D}(\hat{L})\cdot \alpha)$ in the sense of \cite[Definition~3.1]{DF20}. Note that we made an obvious extension of \cite[Definition~3.1]{DF20} by linearity. We observe that for nef b-divisors $\mathbb{D}_1,\ldots,\mathbb{D}_n$ on $\mathfrak{X}$, we have
\begin{equation}\label{eq:DintequalDint}
    (\mathbb{D}_1,\ldots,\mathbb{D}_n)=\int_{\mathfrak{X}}\mathbf{c}_1(\mathbb{D}_1)\dots\mathbf{c}_1(\mathbb{D}_n)[X].
\end{equation}
In fact, one can easily reduce to the case where all $\mathbb{D}_i$'s are Cartier, in which case \eqref{eq:DintequalDint} is trivial.
 
 More generally, given $\hat{E}\in \Fins(X)$ of rank $r+1$ with $\mathcal{I}$-good metric (in the quasi-positive sense), we define $\mathbf{s}_a(\hat{E})\colon \dBPF^{k}(\RZ(\mathbb{P}E^{\vee}))\rightarrow \dBPF^{k+a}(\mathfrak{X})$:
 \[
 \mathbf{s}_a(\hat{E})\alpha\coloneqq (-1)^a p_*\left(\mathbf{c}_1(\hat{\mathcal{O}}(1))^{r+a}\alpha\right),
 \]
 where $p\colon \mathbb{P}E^{\vee}\rightarrow X$ is the natural projection. The same definition makes sense if $\hat{E}$ is $\mathcal{I}$-good and quasi-positive.
 As \cref{lma:segcommrz}, one can easily show that Segre classes of two elements in $\Fins_{\mathcal{I}}(X)$ commute.
 
Now  \cref{cor:cptCherncurri} can be reformulated as
\begin{corollary}\label{cor:ref2}
Assume that $X$ is projective.
Let $\hat{E}_i\in \widehat{\Vect}_{\mathcal{I}}(X)$ or $\Fins_{\mathcal{I}}(X)$ ($i=1,\ldots,m$).
Consider a monomial in the Segre classes $s_{a_1}(\hat{E}_1)\cdots s_{a_m}(\hat{E}_m)$ with $\sum_i a_i=n$. 
Then we have
\begin{equation}\label{eq:CWfinalRZ}
\int_{\mathfrak{X}}\mathbf{s}_{a_1}(\hat{E}_1)\cdots \mathbf{s}_{a_m}(\hat{E}_m)
=\int_X s_{a_1}(\hat{E}_1)\cdots s_{a_m}(\hat{E}_m).
\end{equation}
\end{corollary}
We need to properly interpret the left-hand side of \eqref{eq:CWfinalRZ}. Let $Y=\mathbb{P}E_1^{\vee}\times_X\cdots \times_X \mathbb{P}E_m^{\vee}$. Then
$\int_{\mathfrak{X}}\mathbf{s}_{a_1}(\hat{E}_1)\cdots \mathbf{s}_{a_m}(\hat{E}_m)$ means $\int_{\mathfrak{X}}\mathbf{s}_{a_1}(\hat{E}_1)\cdots \mathbf{s}_{a_m}(\hat{E}_m)[Y]$. We have omitted obvious pull-backs.


More generally, one can make sense of any homogeneous Chern polynomial of degree $n$ in vector bundles with Griffiths positive $\mathcal{I}$-good singular metrics, say $P(c_i(\hat{E}_j))$ as polynomials in the Segre classes, then we can write
\begin{equation}\label{eq:CWfinalRZ2}
\int_{\mathfrak{X}}P(\mathbf{c}_i(\hat{E}_j))
=\int_X P(c_i(\hat{E}_j)).
\end{equation}

\begin{proof}
We need to identify the left-hand side of \eqref{eq:CWfinalRZ} with the left-hand side of \eqref{eq:DDDequal}.

As we know \eqref{eq:DintequalDint},
by induction on $m$, we only need to prove the following: consider the commutative diagram
\[
\begin{tikzcd}
\mathbb{P}E_1^{\vee}\times_X \mathbb{P}E_2^{\vee} \arrow[r, "q_1"] \arrow[d, "q_2"] \arrow[rd, "\square", phantom] & \mathbb{P}E_1^{\vee} \arrow[d, "p_1"] \\
\mathbb{P}E_2^{\vee} \arrow[r, "p_2"]                                                                              & X                                    
\end{tikzcd},
\]
given any $\beta\in \dBPF(\RZ(\mathbb{P}E_1^{\vee}\times_X \mathbb{P}E_2^{\vee}))$ and $a\in \mathbb{N}$, we have
\[
p_{1*}(\mathbf{s}_a(p_1^*\hat{E}_2)\beta)=\mathbf{s}_a(\hat{E}_2)q_{2*}\beta.
\]
By definition, it suffices to show
\begin{equation}\label{eq:tempproj1}
q_{2*}(\mathbf{c}_1(q_2^*\mathcal{O}_{\mathbb{P}E_2^{\vee}}(1))^{r_2+a}\beta)=\mathbf{c}_1(\mathcal{O}_{\mathbb{P}E_2^{\vee}}(1))^{r_2+a}q_{2*}\beta.
\end{equation}
So we are reduced to the case where $E_2$ is a line bundle. We will prove a more general result: if $q:Y\rightarrow X$ is a smooth morphism of pure relative dimension $m-n$ between smooth projective varieties $Y$ and $X$ of dimension $m$ and $n$ and $\hat{L}\in \widehat{\Pic}(X)$ has $\mathcal{I}$-good singularities (in the quasi-positive sense), then for any $\beta\in \dBPF(\mathfrak{Y})$, we have
\begin{equation}\label{eq:tempproj2}
q_*(\mathbf{c}_1(q^*\hat{L})\beta)=\mathbf{c}_1(\hat{L})q_*\beta.
\end{equation}
This clearly implies \eqref{eq:tempproj1} by induction. In order to prove \eqref{eq:tempproj2}, we may assume that $\hat{L}\in \widehat{\Pic}_{\mathcal{I}}(X)$. By the same approximations as in the proof of \cref{thm:nefbvolume}, we may assume that $\hat{L}$ has analytic singularities. In this case, $\mathbf{c}_1(\hat{L})$ is Cartier and $\mathbf{c}_1(q^*\hat{L})=q^*\mathbf{c}_1(\hat{L})$. So \eqref{eq:tempproj2} reduces to the usual projection formula of numerical classes.
\end{proof}

We remark that \eqref{eq:CWfinalRZ} and \eqref{eq:CWfinalRZ2} establish the relation between analytic objects on one side (non-pluripolar products, Chern currents etc.) and algebraic objects on the other side (b-divisors, characteristic classes on the Riemann--Zariski spaces). It is a final confirmation of the fact that the notion of $\mathcal{I}$-good singularities introduced in \cite{DX21, DX22} is a good one. Intuitively, $\mathcal{I}$-good singularities are the largest class of analytic singularities bearing an algebraic feature.

However, our Chern--Weil formulae do have a drawback: the singularities of a Hermitian pseudo-effective line bundle gives an enhanced b-divisor, not only a b-divisor. So parallel to the case of analytic singularities, one should expect to be able to lift $\mathbf{c}_1$ and $\mathbf{s}_i$ to certain Chow groups. This is unfortunately impossible in the current formulation: the singularity divisors have infinitely many components, while there is no good notion of convergence on the Chow groups. 
For example, by \cite{Mum68}, the higher Chow groups $\CH^i$ ($i>1$) are usually infinite-dimensional. They are usually huge enough and do not seem to admit natural topologies. One could of course formally enlarge the Chow groups to include the infinite sums that should be convergent. But this \emph{ad hoc} approach breaks the intrinsic beauty of the whole theory. There is a different option: one could try to construct a K-theory and a $\lambda$-ring structure using $\mathcal{I}$-good Hermitian vector bundles instead of just vector bundles. This approach requires a huge amount of extra work, which we deliberately avoid in this paper.

\cref{cor:ref2} can be easily extended to quasi-projective varieties as well. We leave the details to the readers. Less intrinsically, one can also apply \cref{cor:ref2} to a compactification $\bar{X}$ of $X$. Both sides of \eqref{eq:CWfinalRZ} are independent of the choice of $\bar{X}$. 

Finally, let us mention that one can more generally develop the theory of Chow groups with supports: $\CH_{\mathfrak{Z}}(\mathfrak{X})$ using the same methods, where $Z\subseteq X$ is a Zariski closed subset. These refinements could be helpful if we want to consider compactification problems.

\clearpage

\printbibliography

\bigskip
  \footnotesize

  Mingchen Xia, \textsc{Institut de Mathématiques de Jussieu-Paris Rive Gauche,  Paris}\par\nopagebreak
  \textit{Email address}, \texttt{mingchen@imj-prg.fr}\par\nopagebreak
  \textit{Homepage}, \url{http://mingchenxia.github.io/home/}.

\end{document}